\documentclass{amsart}

\usepackage{mathrsfs}
\usepackage{mathrsfs}
\usepackage{amsfonts}
\usepackage{amsfonts}
\usepackage{mathrsfs}
\usepackage{amsfonts}
\usepackage{amsfonts}
\usepackage{amsfonts}
\usepackage{mathrsfs}
\usepackage{amsfonts}
\usepackage{amssymb,amsmath}
\usepackage{amsthm}
\usepackage{amsfonts}

\newtheorem{theorem}{Theorem}[section]
\newtheorem{lemma}[theorem]{Lemma}

\theoremstyle{definition}

\newtheorem{proposition}[theorem]{Proposition}

\theoremstyle{remark}
\newtheorem{remark}[theorem]{Remark}

\numberwithin{equation}{section}

\begin{document}

\title[Decay and Stability of Inhomogeneous MHD systems]
 {On the Decay and Stability of Global Solutions to the 3D Inhomogeneous MHD system}

\author[J.X.Jia]{Junxiong Jia}
\address{School of Mathematics and Statistics,
Xi'an Jiaotong University,
 Xi'an
710049, China; Beijing Center for Mathematics and Information Interdisciplinary Sciences (BCMIIS);}
\email{jjx425@gmail.com}
%\thanks{
%This work was supported by the NSFC under the grant no. 11131006 and partially by the National Basic Research Program of China under the grant no. 2013CB329404.
%}

 \author[J.G. Peng]{Jigen Peng}
\address{School of Mathematics and Statistics,
Xi'an Jiaotong University,
 Xi'an
710049, China; Beijing Center for Mathematics and Information Interdisciplinary Sciences (BCMIIS);}
\email{jgpeng@mail.xjtu.edu.cn}

\author[K.X. Li]{Kexue Li}
\address{School of Mathematics and Statistics,
Xi'an Jiaotong University,
 Xi'an
710049, China; Beijing Center for Mathematics and Information Interdisciplinary Sciences (BCMIIS);}
\email{kexueli@gmail.com}
%    Information for third author
%\author[Z.D.Mei]{Zhan-Dong Mei}
%\address{Department of Mathematics,
%Xi'an Jiaotong University,
% Xi'an
%710049, China;}
% \email{mzhd1516@gmail.com}

%    General info

%%\date{January 1, 2001 and, in revised form, June 22, 2001.}

%%\dedicatory{This paper is dedicated to our advisors.}


\begin{abstract}
In this paper, we investigative the large time decay and stability to any given global smooth solutions of the $3$D incompressible inhomogeneous MHD systems.
We proved that given a solution $(a, u, B)$ of (\ref{mhd_a}), the velocity field and magnetic field decay to $0$ with an explicit rate, for $u$ which
coincide with incompressible inhomogeneous Navier-Stokes equations \cite{zhangping}.
In particular, we give the decay rate of higher order derivatives of $u$ and $B$
which is useful to prove our main stability result. For a large solutions of (\ref{mhd_a}) denoted by $(a, u, B)$,
we proved that a small perturbation to the initial data still generates a unique global smooth solution and the smooth solution
keeps close to the reference solution $(a, u, B)$.
Due to the coupling between $u$ and $B$,
we used elliptic estimates to get $\|(u, B)\|_{L^{1}(\mathbb{R}^{+};\dot{B}_{2,1}^{5/2})} < C$, which
is different to Navier-Stokes equations.
\end{abstract}

\maketitle

%%%%%%%%%%%%%%%%%%%%%%%%%%%%%%%%%%%%%%%%%%%%%%%%%%%%%%%%%%%%%%%%%%%%%%%%%%%%%%%%%%%%%%%%%%%%%%%%%%%%%%%%%%%%%%%%%%%%%%%%%%%%%%%%%%%%%%%%%%%%%%%%%%%%%%%%%%%

\section{Introduction and main results}

Magnetic fields influence many fluids. Magnetohydrodynamics (MHD) is concerned with the interaction between fluid flow and
magnetic field. The governing equations of nonhomogeneous MHD can be stated as follows \cite{Davidsom},
\begin{align}
\label{equation}
\begin{cases}
\partial_{t}\rho + \mathrm{div}(\rho u) = 0, \quad (t,x) \in \mathbb{R}^{+}\times \mathbb{R}^{3}, \\
\partial_{t}(\rho u) + \mathrm{div}(\rho u\otimes u) - \mathrm{div}(\mu \mathcal{M}) - (B\cdot\nabla)B + \nabla\Pi = 0, \\
\partial_{t}B - \lambda \Delta B - \mathrm{curl}(u \times B) = 0, \\
\mathrm{div}\,u = \mathrm{div}\,B = 0,  \\
\rho|_{t=0} = \rho_{0},\quad u|_{t=0} = u_{0},\quad B|_{t=0} = B_{0},
\end{cases}
\end{align}
where $\rho$, $u = (u_{1}, u_{2}, u_{3})$ stand for the density and velocity of the fluid, respectively,
$\mathcal{M} = \frac{1}{2}(\partial_{i}u_{j} + \partial_{j}u_{i})$, $\Pi$ is a scalar pressure function.
$B$ is the magnetic field. $\mu(\rho) \geq 0$ denotes the viscosity of fluid, which we assume in this paper is a positive
constant. $\lambda > 0$ is also a constant, which describes the relative strength of advection and diffusion of $B$.

If there is no magnetic field, i.e., $B = 0$, MHD system turns to be nonhomogeneous Navier-Stokes system. Since the second equation and third equation of (\ref{equation}) are similar, the study about MHD system has been along with that for Navier-Stokes one. Let us first recall some results about
Navier stokes equations. When $\rho_{0}$ is bounded away from zero, the global existence of weak solutions was established by Kazhikov \cite{Kazhikov}.
Moreover, Antontsev, Kazhikov and Monakhov \cite{Antontsev} gave the first result on local existence and uniqueness of strong solutions.
For the two dimensional case, they even proved that the strong solution is global. But the global existence of strong or smooth solutions
in 3D is still an open problem.

Recently, R.Danchin proved the global existence in the Besov space framework \cite{Danchin 2012}.
His result states the global in time existence of regular solutions
to the inhomogeneous Navier Stokes equations in $\mathbb{R}^{n}$ in the optimal Besov setting, under suitable smallness of the data.
In particular his results allows the initial densities have a jump at the interface.
At the same time, H. Abidi, G. Gui, P. Zhang \cite{zhangping2012} proved the local well-posedness of three-dimensional incompressible inhomogeneous
Navier-Stokes equations with initial data in the critical Besov spaces,
without assumptions of small density variation. And they also proved the global well-posedness when the initial velocity is
small in $\dot{B}_{2,1}^{1/2}(\mathbb{R}^{3})$. For more results in this direction,
see \cite{CheminZhangping,DanchinZhangping,Danchin 2013}and reference therein.

Now, let us go back to the MHD system (\ref{equation}). As said before, the research for MHD goes along with that for Naiver-Stokes equations.
The results are similar. When we assume $\rho$ is a constant, which means the fluid is homogeneous, the MHD system has been extensively studied.
Duraut and Lions \cite{lions} constructed a class of weak solutions with finite energy and a class of local strong solutions.
Recently, C. Cao, J. Wu \cite{Jiahong Wu 2011} proved global regularity of classical solutions for the MHD equations with mixed partial dissipation and magnetic diffusion. And they also give the global existence,
conditional regularity and uniqueness of a weak solution for 2D MHD equations with only magnetic diffusion.
For more results in this direction,
see \cite{Jiahong Wu 2013,Jiahong Wu}and reference therein.

When the fluid is nonhomogeneous. H. Abidi and M. Paicu \cite{HammadiMHD} proved that the
magneto-hydrodynamic system in $\mathbb{R}^{N}$ with variable density, variable viscosity and variable conductivity has a local weak solution in suitable Besov space if the initial density approaches a constant. And they also proved that the constructed solution exist globally
in time if the initial data are small enough. X. Huang and Y. Wang \cite{Xiangdi Huang 2013} proved the global existence of strong
solution with vacuum to the 2D nonhomogeneous incompressible MHD system, as long as the initial data satisfies some compatibility condition.
In this paper, we only consider non-vacuum case.

Let $a := \frac{1}{\rho} - 1$ and take $\mu = \lambda = 1$, then MHD system becomes
\begin{align}
\label{mhd_a}
\begin{cases}
\partial_{t}a + u\cdot \nabla a = 0, \quad\quad (t,x) \in \mathbb{R}^{+} \times \mathbb{R}^{3},  \\
\partial_{t}u + u\cdot \nabla u - (1+a)(\Delta u - \nabla \Pi) - (1+a)(B\cdot \nabla)B = 0,     \\
\partial_{t}B - \Delta B + (u\cdot \nabla)B - (B\cdot \nabla)u = 0,   \\
\mathrm{div}u = \mathrm{div}B = 0,  \\
(a, u, B)|_{t=0} = (a_{0}, u_{0}, B_{0}).
\end{cases}
\end{align}
or let $\rho := \frac{1}{1+a}$ and $(\rho, u, B)$ solves
\begin{align}
\label{mhd_rho}
\begin{cases}
\partial_{t}\rho + \mathrm{div}(\rho u) = 0, \quad\quad (t,x) \in \mathbb{R}^{+} \times \mathbb{R}^{3},  \\
\rho \partial_{t} u + \rho u \cdot \nabla u - \Delta u - (B\cdot \nabla)B + \nabla \Pi = 0, \\
\partial_{t} B - \Delta B + (u\cdot \nabla)B - (B\cdot \nabla)u = 0,    \\
\mathrm{div}u = \mathrm{div}B = 0,    \\
(\rho, u, B)|_{t=0} = (\rho_{0}, u_{0}, B_{0}).
\end{cases}
\end{align}

In what follows, we shall investigate the large time decay and stability for the above MHD system.
Compared to the classical incompressible Navier-Stokes system (NS), that is, in the case when $a=0$ and $B=0$ in (\ref{mhd_a}), the
MHD system is more complex. Given a global large solution $
u\in L_{\mathrm{loc}}^{\infty}([0,\infty);H^{1}(\mathbb{R}^{3}))\cap L_{\mathrm{loc}}^{2}([0,\infty);H^{2}(\mathbb{R}^{2}(\mathbb{R}^{3})))$,
Ponce, Racke, Sideris and Titi \cite{TiTi1994} proved (NS) global stability under the additional assumption $\int_{0}^{\infty}\|\nabla u(t)\|_{L^{2}}^{4}\,dt < \infty$. Then Gallagher, Iftimie and Planchon \cite{Planchon} removed the additional assumption.
For the inhomogeneous Navier-Stokes equations (INS), H. Abidi, G. Gui and P. Zhang \cite{zhangping} proved its decay and stability for large solutions.

Our first result concerns the global stability of the given solution of (\ref{mhd_a}) when the initial density $\rho_{0}$ is close to
a positive constant. This is a simple generalization of Theorem 1.1 in \cite{zhangping}.

\begin{theorem}
\label{close to 1 main theorem}
Let $\bar{a}_{0} \in B_{2,1}^{5/2}(\mathbb{R}^{3})$, $\bar{u}_{0} \in B_{2,1}^{3/2}(\mathbb{R}^{3})$ and $\bar{B}_{0} \in B_{2,1}^{3/2}(\mathbb{R}^{3})$
with $\mathrm{div}\bar{u}_{0} = \mathrm{div}\bar{B}_{0} = 0$, and let there exist two positive constants $m$ and $M$ so that
\begin{align}
\label{low}
m \leq 1+\bar{a}_{0} \leq M.
\end{align}
We assume that $\bar{a} \in C([0, \infty); B_{2,1}^{5/2}(\mathbb{R}^{3}))$ and
\begin{align*}
\bar{u}, \bar{B} \in C([0,\infty); B_{2,1}^{3/2}(\mathbb{R}^{3})) \cap L_{\text{loc}}^{1}(\mathbb{R}^{+}; \dot{B}_{2,1}^{7/2}(\mathbb{R}^{3}))
\end{align*}
is a given solution of MHD with initial data $(\bar{a}_{0}, \bar{u}_{0})$. Then there exist positive constants $c_{1}$, $C_{1}$ and
a large enough time $T_{0} := T_{0}(\bar{a}_{0}, \bar{u}_{0}, \bar{B}_{0})$ so that if
\begin{align}
\|\bar{a}_{0}\|_{\dot{B}_{2,1}^{3/2}}\mathrm{exp}\left\{ C_{1} \int_{0}^{T_{0}} \|\nabla \bar{u}(\tau)\|_{\dot{B}_{2,1}^{3/2}} \, d\tau \right\} \leq c_{1},
\end{align}
a constant $c_{2}$ exists so that $(a_{0}, u_{0}, B_{0}) := (\bar{a}_{0}+\tilde{a}_{0}, \bar{u}_{0}+\tilde{u}_{0}, \bar{B}_{0}+\tilde{B}_{0})$
generates a unique global solution with
\begin{align}
\begin{split}
& a \in C_{b}([0, \infty); B_{2,1}^{5/2}(\mathbb{R}^{3})),    \\
& u, B \in C_{b}([0, \infty); B_{2,1}^{3/2}(\mathbb{R}^{3})) \cap L^{1}([0, \infty); \dot{B}_{2,1}^{7/2}(\mathbb{R}^{3})),
\end{split}
\end{align}
provided that $(\tilde{a}_{0}, \tilde{u}_{0}, \tilde{B}_{0})$ satisfies
\begin{align}
\|\tilde{a}_{0}\|_{B_{2,1}^{3/2}} + \|\tilde{u}_{0}\|_{B_{2,1}^{1/2}} + \|\tilde{B}_{0}\|_{B_{2,1}^{1/2}} \leq c_{2}.
\end{align}
\end{theorem}

In order to get the stability of large solutions of system (\ref{mhd_a}), here, we need to investigate the decay properties of
the velocity field $u$ and magnetic field $B$. Compared to the (INS) case, our case is more complex and we need to use the coupling
between the equations of $u$ and $B$. Due to the coupling between $u$ and $B$, we can not get the estimate of
$\|(u, B)\|_{L^{1}(\mathbb{R}^{+};\dot{B}_{2,1}^{5/2})} < C$ by using propositions like Proposition 3.6 in \cite{zhangping}, which is
the methods used in (INS) case. In order to overcome this difficulty, we need to give the estimate of $\|\nabla a(t)\|_{L^{\infty}}$,
so that in addition to get the higher order decay properties of $u$ and $B$.
The rigorous statement is the following Theorem.

\begin{theorem}
\label{decay_main_theorem}
For $p \in (1,\frac{6}{5})$, let $a_{0} \in B_{2,1}^{5/2}$ and $u_{0} \in L^{p}(\mathbb{R}^{3}) \cap B_{2,1}^{2}(\mathbb{R}^{3})$
satisfying (\ref{low}) and $\mathrm{div}\,u_{0} = 0$. We assume that $a \in C([0,\infty);B_{2,1}^{5/2})$,
$u \in C([0, \infty);B_{2,1}^{2}(\mathbb{R}^{3})) \cap L_{\mathrm{loc}}^{1}(\mathbb{R}^{+};\dot{B}_{2,1}^{4}(\mathbb{R}^{3}))$
and $B \in C([0, \infty);B_{2,1}^{2}(\mathbb{R}^{3})) \cap L_{\mathrm{loc}}^{1}(\mathbb{R}^{+};\dot{B}_{2,1}^{4}(\mathbb{R}^{3}))$
is a given global solution of (\ref{mhd_a}) with initial data $(a_{0}, u_{0}, B_{0})$. Then there exists a positive time $t_{0}$
such that there hold
\begin{align}
\label{main e 1}
\begin{split}
 \|u(t)\|_{L^{2}} + \|B(t)\|_{L^{2}}& \leq C \, (1+t)^{-\beta(p)},    \\
 \|\nabla u(t)\|_{L^{2}} + \|\nabla B(t)\|_{L^{2}}& \leq C\, (1+t)^{-\frac{1}{2}-\beta(p)}, \quad \text{for } t\geq t_{0},   \\
\int_{t_{0}}^{\infty} (1+t)^{(1+2\, \beta(p))^{-}} \bigg( \|(\partial_{t}u, \partial_{t}B)\|_{L^{2}}^{2}
& + \|(\Delta u, \Delta B)\|_{L^{2}}^{2} + \|\nabla \Pi\|_{L^{2}}^{2} \bigg)\, dt \leq C, \\
\int_{t_{0}}^{\infty} \Big( \|u(t)\|_{L^{\infty}} + \|B(t)\|_{L^{\infty}} + & \|\nabla u(t)\|_{L^{\infty}} + \|\nabla B(t)\|_{L^{\infty}} \Big)\, dt \leq C
\end{split}
\end{align}
and
\begin{align}
\label{main e 3}
\begin{split}
\sup_{t\geq t_{0}}\Big( \|\nabla^{2}u(t)\|_{L^{2}}^{2} & + \|\nabla^{2}B(t)\|_{L^{2}}^{2} \Big)
 + \int_{t_{0}}^{\infty}\|\partial_{t}\nabla u(t)\|_{L^{2}}^{2} + \|\partial_{t}\nabla B(t)\|_{L^{2}}^{2}\, dt  \\
& + \int_{t_{0}}^{\infty}\|\nabla^{3}u(t)\|_{L^{2}}^{2} + \|\nabla^{3}B(t)\|_{L^{2}}^{2} \, dt \leq \, C,
\end{split}
\end{align}
where $\beta(p) = \frac{3}{4}(\frac{2}{p} - 1)$. $(1+2\beta(p))^{-}$ denotes any positive number smaller than $1+2\beta(p)$,
and the constant $C$ depends on the initial data.
\end{theorem}

At the above theorem in hand, we can use the elliptic estimates and various interpolation theorems in Besov space to get the estimate of
$\|(u, B)\|_{L^{1}(\mathbb{R}^{+};\dot{B}_{2,1}^{5/2})} < C$ which is completely different to (INS) case.
Then after complex calculations, we can get the global estimates of the reference solution $(\bar{a}, \bar{u}, \bar{B})$.
At last, similar to Theorem \ref{decay_main_theorem} but need more complex calculations,
we obtain the decay properties of the perturbed solution $(a-\bar{a}, u-\bar{u}, B-\bar{B})$.
Using the decay properties of the reference solution and perturbed solution, we finally got the following theorem.

\begin{theorem}
\label{stability main theorem}
For $p \in (1, \frac{6}{5})$, let $\bar{a}_{0} \in B_{2,1}^{7/2}(\mathbb{R}^{3})$,
$\bar{u}_{0} \in L^{p}(\mathbb{R}^{3}) \cap B_{2,1}^{2}(\mathbb{R}^{3})$, $\bar{B}_{0} \in L^{p}(\mathbb{R}^{3}) \cap B_{2,1}^{2}(\mathbb{R}^{3})$
satisfy $\mathrm{div}\,\bar{u}_{0} = \mathrm{div}\,\bar{B}_{0} = 0$ and (\ref{low}).
We assume that $\bar{a} \in C([0, \infty);B_{2,1}^{7/2}(\mathbb{R}^{3}))$,
$\bar{u} \in C([0, \infty);B_{2,1}^{2}(\mathbb{R}^{3})) \cap L_{\mathrm{loc}}^{1}(\mathbb{R}^{+};\dot{B}_{2,1}^{4})$,
$\bar{B} \in C([0, \infty);B_{2,1}^{2}(\mathbb{R}^{3})) \\ \cap L_{\mathrm{loc}}^{1}(\mathbb{R}^{+};\dot{B}_{2,1}^{4})$
is a given global solution of (\ref{mhd_a}) with initial data $(\bar{a}_{0}, \bar{u}_{0}, \bar{B}_{0})$.
Then there exists a constant $c$ so that if
\begin{align*}
(\tilde{a}_{0}, \tilde{u}_{0}, \tilde{B}_{0}) \in B_{2,1}^{7/2}(\mathbb{R}^{3})\times
\Big( L^{p}(\mathbb{R}^{3}) \cap B_{2,1}^{2}(\mathbb{R}^{3}) \Big) \times
\Big( L^{p}(\mathbb{R}^{3}) \cap B_{2,1}^{2}(\mathbb{R}^{3}) \Big)
\end{align*}
with
\begin{align*}
A_{0} :=  \|(\tilde{u}_{0}, \tilde{B}_{0})\|_{H^{1}} + \|(\tilde{u}_{0}, \tilde{B}_{0})\|_{L^{p}} + \|\tilde{a}_{0}\|_{B_{2,1}^{3/2}} \leq c,
\end{align*}
$(a_{0}, u_{0}, B_{0}) := (\bar{a}_{0}+\tilde{a}_{0}, \bar{u}_{0}+\tilde{u}_{0}, \bar{B}_{0}+\tilde{B}_{0})$ generates a
unique global smooth solution $(a, u, B)$ to (\ref{mhd_a}) that satisfies
\begin{align*}
& a \in C_{b}([0,\infty);B_{2,1}^{7/2}(\mathbb{R}^{3})),       \\
& u \in C_{b}([0,\infty);L^{p}(\mathbb{R}^{3} \cap B_{2,1}^{2}(\mathbb{R}^{3}))) \cap L^{1}(\mathbb{R}^{+};\dot{B}_{2,1}^{4}(\mathbb{R}^{4})), \\
& B \in C_{b}([0,\infty);L^{p}(\mathbb{R}^{3} \cap B_{2,1}^{2}(\mathbb{R}^{3}))) \cap L^{1}(\mathbb{R}^{+};\dot{B}_{2,1}^{4}(\mathbb{R}^{4})).
\end{align*}
Moreover, there holds
\begin{align}
\label{1 3 1}
\|a-\tilde{a}\|_{\tilde{L}^{\infty}(\mathbb{R}^{+};B_{2,1}^{s+1})} \leq C A_{0}^{\frac{5}{4}-\frac{1}{2}s}
\end{align}
for any $s \in [\frac{1}{2}, \frac{5}{2}]$ and
\begin{align}
\label{1 3 2}
\begin{split}
\|(u-\bar{u}, B-\bar{B})\|_{\tilde{L}^{\infty}(\mathbb{R}^{+};B_{2,1}^{s})}
& + \|(u-\bar{u}, B-\bar{B})\|_{L^{\infty}(\mathbb{R}^{+};L^{p})}  \\
& + \|(u-\bar{u}, B-\bar{B})\|_{L^{1}(\mathbb{R}^{+};\dot{B}_{2,1}^{s+2})} \leq C A_{0}^{\frac{4}{3}-\frac{2}{3}s}
\end{split}
\end{align}
for any $s \in [\frac{1}{2}, 2]$.
\end{theorem}

\begin{remark}
The above Theorems may not be obtained by regarding the term $B\cdot \nabla B$ as a source term in the velocity equation.
The reason is that if we regard this term as a source term, we will encounter terms like $\bar{B}\cdot\nabla \tilde{B}$ and $\bar{B} \cdot\nabla \tilde{u}$
in (\ref{fanbo 1}) and (\ref{fanbo 2}). For the appearance of these terms, the Bootstrap argument will not work.
So we consider the linear system (\ref{linearcouple}) is necessary and the higher order decay estimates is also necessary
to get the results for MHD system.
\end{remark}

The paper is organized as follows. In section 2, we will give some notations, a brief introduction to the Besov space and some useful Lemmas.
In section 3, as a warm up, we give the proof of Theorem \ref{close to 1 main theorem}. Then, in section 4, we proved Theorem \ref{decay_main_theorem}
in a series of propositions. Using Theorem \ref{decay_main_theorem}, we got the global estimates of the reference solutions in section 5.
At last, we proved the decay properties of the perturbed solutions and Theorem \ref{stability main theorem} in section 6.
In the appendix, we proved some simple Lemmas which will be used in the above sections.

%%%%%%%%%%%%%%%%%%%%%%%%%%%%%%%%%%%%%%%%%%%%%%%%%%%%%%%%%%%%%%%%%%%%%%%%%%%%%%%%%%%%%%%%%%%%%%%%%%%%%%%%%%%%%%%%%%%%%%%%%%%%%%%%%%%%%%%%%%%%%%%%%%%%%%%%%%%

\section{Preliminaries}

Throughout this paper we will use the following notations.
\begin{itemize}
  \item For any tempered distribution $u$ both $\widehat{u}$ and $\mathcal{F}u$ denote the Fourier transform of $u$.
  \item The norm in the mixed space-time Lebesgue space $L^{p}([0,T];L^{r}(\mathbb{R}^{d}))$ is denoted by $\|\cdot\|_{L^{p}_{T}L^{r}}$ (with the obvious generalization to $\|\cdot\|_{L^{p}_{T}X}$ for any normed space X).
  \item For $X$ a Banach space and $I$ an interval of $\mathbb{R}$, we denote by $C(I; X)$ the set of continuous functions on $I$
  with values in $X$, and by $C_{b}(I; X)$ the subset of bounded functions of $C(I; X)$.
  \item For any pair of operators $P$ and $Q$ on some Banach space $X$, the commutator $[P,Q]$ is given by $PQ-QP$.
  \item $C$ stands for a ``harmless'' constant, and we sometimes use the notation $A \lesssim B$ as an
equivalent of $A \leq C B$. The notation $A \thickapprox B$ means that $A \lesssim B$ and $B \lesssim A$.
  \item $\{ c_{j,r} \}_{j \in \mathbb{Z}}$ a generic element of the sphere of $\ell^{r}(\mathbb{Z})$, and $(c_{k})_{k\in \mathbb{Z}}$
  (respectively, $(d_{j})_{j \in \mathbb{Z}}$) a generic element of the sphere of $\ell^{2}(\mathbb{Z})$ (respectively, $\ell^{1}(\mathbb{Z})$).
  \item Denote $\gamma^{-}$ be any number smaller than $\gamma$.
\end{itemize}

Then, we give a short introduction to the Besov type space. Details about Besov type space can be found in \cite{Danchin notes} or \cite{danchin book}.
There exist two radial positive functions $\chi \in \mathcal{D}(\mathbb{R}^{d})$ and $\varphi \in \mathcal{D}(\mathbb{R}^{d}\backslash \{0\})$
such that
\begin{itemize}
  \item $\chi (\xi) + \sum_{q\geq 0} \varphi (2^{-q} \xi) =1$; $\forall q \geq 1$,
        $\text{supp} \chi \cap \text{supp} \varphi (2^{-q}\cdot) = \phi$,
  \item $\text{supp} \varphi (2^{-j}\cdot) \cap \text{supp} \varphi (2^{-k}\cdot) = \phi$, if $|j-k|\geq 2$,
\end{itemize}

For every $v \in S^{'}(\mathbb{R}^{d})$ we set
\begin{align*}
\Delta_{-1}v = \chi(D)v, \quad \forall q \in \mathbb{N}, \quad \Delta_{j}v = \varphi (2^{-q}D)v \quad \text{and} \quad
S_{j} = \sum_{-1\leq m\leq j-1} \Delta_{m}.
\end{align*}
The homogeneous operators are defined by
\begin{align*}
\dot{\Delta}_{q}v = \varphi (2^{-q}D)v, \quad \dot{S}_{j}=\sum_{m\leq j-1} \dot{\Delta}_{j}v, \quad \forall q \in \mathbb{Z}.
\end{align*}
One can easily verifies that with our choice of $\varphi$,
\begin{align}\label{dd2}
\Delta_{j}\Delta_{k}f=0\quad \text{if} \quad |j-k|\geq 2
\end{align}
\begin{align}\label{ds4}
\Delta_{j}(S_{k-1}f \Delta_{k}f)=0\quad \text{if} \quad |j-k|\geq 5.
\end{align}
As in Bony's decomposition, we split the product $uv$ into three parts
\begin{align*}
uv=T_{u}v + T_{v}u +R(u,v),
\end{align*}
with
\begin{align*}
T_{u}v=\sum_{j}S_{j-1}u \Delta_{j}v,
\end{align*}
\begin{align*}
R(u,v)=\sum_{j} \Delta_{j}u \widetilde{\Delta}_{j}v
\end{align*}
where $\widetilde{\Delta}_{j} = \Delta_{j-1} + \Delta_{j} + \Delta_{j+1}$.

Let us now define inhomogeneous Besov spaces. For $(p,r)\in [1, +\infty]^{2}$ and $s\in \mathbb{R}$ we define the inhomogeneous
Besov space $B_{p,r}^{s}$ as the set of tempered distributions $u$ such that
\begin{align*}
\|u\|_{B_{p,r}^{s}}:=(2^{js}\|\Delta_{j}u\|_{L^{p}})_{\ell^{r}} < +\infty.
\end{align*}
The homogeneous Besov space $\dot{B}_{p,r}^{s}$ is defined as the set of $u \in S^{'}(\mathbb{R}^{d})$ up to polynomials such that
\begin{align*}
\|u\|_{\dot{B}_{p,r}^{s}}:=(2^{js}\|\dot{\Delta}_{j}u\|_{L^{p}})_{\ell^{r}} < +\infty.
\end{align*}
Notice that the usual Sobolev spaces $H^{s}$ coincide with $B_{2,2}^{s}$ for every $s \in \mathbb{R}$ and that the homogeneous
spaces $\dot{H}^{s}$ coincide with $\dot{B}_{2,2}^{s}$.
%
%Let $s \in \mathbb{R}$, $1 \leq p, r \leq \infty$, and $u \in \mathcal{S}'(\mathbb{R}^{3})$. Then $u$ belongs to $\dot{B}_{p,r}^{2}(\mathbb{R}^{3})$
%if and only if there exists $\{ c_{j,r} \}_{j \in \mathbb{Z}}$ such that $\|c_{j,r}\|_{\ell^{r}} = 1$ and
%\begin{align*}
%\|\dot{\Delta}_{j}u\|_{L^{p}} \leq C c_{j,r} 2^{-j r} \|u\|_{\dot{B}_{p,r}^{s}} \quad \text{for all } j \in \mathbb{Z}.
%\end{align*}

We shall need some mixed space-time spaces. Let $T>0$ and $\rho \geq 1$, we denote by $L^{\rho}_{T}B_{p,r}^{s}$ the space of
distribution $u$ such that
\begin{align*}
\|u\|_{L^{\rho}_{T}\dot{B}_{p,r}^{s}}:=\|(2^{js}\|\dot{\Delta}_{j}u\|_{L^{p}})_{\ell^{r}}\|_{L^{\rho}_{T}} <+\infty.
\end{align*}
We say that $u$ belongs to the space $\widetilde{L}_{T}^{\rho}B_{p,r}^{s}$ if
\begin{align*}
\|u\|_{\widetilde{L}^{\rho}_{T}\dot{B}_{p,r}^{s}}:=(2^{js}\|\dot{\Delta}_{j}u\|_{L^{\rho}_{T}L^{p}})_{\ell^{r}} <+\infty,
\end{align*}
which appeared firstly in \cite{chemin}.
Through a direct application of the Minkowski inequality, the following links between these spaces is true \cite{T.Hmidi2010JDE}.
Let $\varepsilon >0$, then
\begin{align*}
L^{\rho}_{T}B_{p,r}^{s} \hookrightarrow \widetilde{L}^{\rho}_{T}B_{p,r}^{s} \hookrightarrow L^{\rho}_{T}B_{p,r}^{s-\varepsilon},
\quad \text{if} \,\, r\geq \rho,
\end{align*}
\begin{align*}
L^{\rho}_{T}B_{p,r}^{s+\varepsilon} \hookrightarrow \widetilde{L}^{\rho}_{T}B_{p,r}^{s} \hookrightarrow
L^{\rho}_{T}B_{p,r}^{s}, \quad \text{if} \,\, \rho \geq r.
\end{align*}

\begin{lemma}
\label{bernsteininequality} \cite{danchin book}
Let $\mathcal{B}$ be a ball and $\mathcal{C}$ be a ring of $\mathbb{R}^{3}$. A constant $C$ exists so that for any positive real number $\lambda$,
any nonnegative integer $k$, any smooth homogeneous function $\sigma$ of degree $m$, and any couple of read number $(a, b)$ with $b \geq a \geq 1$,
there hold
\begin{align*}
\begin{split}
& \mathrm{supp}\, \hat{u} \subset \lambda \mathcal{B} \, \Rightarrow \, \sup_{|\alpha| = k} \|\partial^{\alpha} u\|_{L^{b}}
\leq C^{k+1} \lambda^{k+3(\frac{1}{a}-\frac{1}{b})} \|u\|_{L^{a}},  \\
& \mathrm{supp} \, \hat{u} \subset \lambda \mathcal{C} \, \Rightarrow \, C^{-1-k} \lambda^{k} \|u\|_{L^{a}} \leq
\sup_{|\alpha|=k} \|\partial^{\alpha} u\|_{L^{a}} \leq C^{1+k} \lambda^{k} \|u\|_{L^{a}},   \\
& \mathrm{supp} \, \hat{u} \subset \lambda \mathcal{C} \, \Rightarrow \, \|\sigma(D)u\|_{L^{b}} \leq C_{\sigma, m}\lambda^{m+3(\frac{1}{a}-\frac{1}{b})}
\|u\|_{L^{a}}.
\end{split}
\end{align*}
\end{lemma}

\begin{lemma}
\label{comunatorestimate}
\cite{zhangping}
Let $r \in [1, \infty]$, $u \in \dot{B}_{2,r}^{s}(\mathbb{R}^{3})$, and $v \in \dot{B}_{2,1}^{5/2}(\mathbb{R}^{3})$ with $\mathrm{div}v = 0$.
Then there hold the following:

(i) If $-\frac{5}{2} < s < \frac{5}{2}$ (or $s = \frac{5}{2}$ with $r = 1$),
\begin{align}
\|[\Delta_{q}, v\cdot \nabla]u\|_{L^{2}} \lesssim c_{q,r} 2^{-sq} \|v\|_{\dot{B}_{2,1}^{5/2}} \|u\|_{\dot{B}_{2,r}^{s}}.
\end{align}

(ii) If $s > -\frac{5}{2}$ and $v \in \dot{B}_{2,r}^{s}(\mathbb{R}^{3})$, $\nabla u \in L^{\infty}(\mathbb{R}^{3})$
\begin{align}
\|[\Delta_{q}, v\cdot \nabla]u\|_{L^{2}} \lesssim c_{q,r} 2^{-sq} \left( \|v\|_{\dot{B}_{2,1}^{5/2}} \|u\|_{\dot{B}_{2,r}^{s}}
+ \|v\|_{\dot{B}_{2,r}^{s}}\|\nabla u\|_{L^{\infty}} \right).
\end{align}

(iii) If $s > -1$,
\begin{align}
\|[\Delta_{q}, v\cdot \nabla]v\|_{L^{2}} \lesssim c_{q,r} 2^{-sq} \|\nabla v\|_{L^{\infty}} \|v\|_{\dot{B}_{2,r}^{s}}.
\end{align}
\end{lemma}

\begin{lemma}
\cite{zhangping}
\label{transport_estimate}
Let $v$ be a divergence-free vector field with $\nabla v \in L^{1}([0,T]; \dot{B}_{2,1}^{3/2})$. For $s \in (-\frac{5}{2}, \frac{5}{2}]$,
given $f_{0} \in \dot{B}_{2,1}^{s}$, $F \in L^{1}([0,T]; \dot{B}_{2,1}^{s})$, the transport equation
\begin{align}
\begin{cases}
\partial_{t}f + v \cdot \nabla f = F \quad \text{in} \,\, \mathbb{R}^{+}\times \mathbb{R}^{3},  \\
f|_{t=0} = f_{0},
\end{cases}
\end{align}
has a unique solution $f \in C([0,T]; \dot{B}_{2,1}^{s})$. Moreover, there holds for all $t\in [0,T]$
\begin{align}
\label{transport 1}
\begin{split}
\|f\|_{\tilde{L}_{T}^{\infty}(\dot{B}_{2,1}^{s})} \leq & \, \|f_{0}\|_{\dot{B}_{2,1}^{s}} + C \int_{0}^{t} \|f(\tau)\|_{\dot{B}_{2,1}^{s}}
\|\nabla v(\tau)\|_{\dot{B}_{2,1}^{3/2}} \, d\tau   \\
& \, + C \|F\|_{L_{t}^{1}(\dot{B}_{2,1}^{s})}.
\end{split}
\end{align}
If $s \in (0, \frac{5}{2}]$, there also holds
\begin{align}
\label{transport 2}
\begin{split}
\|f\|_{\tilde{L}^{\infty}_{t}(B_{2,1}^{s})} \leq &\, \|f_{0}\|_{B_{2,1}^{s}} + C\int_{0}^{t}\|f(t')\|_{B_{2,1}^{s}}\|\nabla v(t')\|_{\dot{B}_{2,1}^{3/2}}\,dt' \\
&\, + C \|F\|_{L^{1}_{t}(B_{2,1}^{s})}.
\end{split}
\end{align}
\end{lemma}

\begin{lemma}
\cite{Danchin notes}
\label{transport_estimate 2}
Let $1 \leq p \leq p_{1} \leq \infty$, $1 \leq r \leq \infty$ and $p' := (1-1/p)^{-1}$. Assume that
\begin{align*}
\sigma > -N \min\Big( \frac{1}{p_{1}}, \frac{1}{p^{'}} \Big) \quad \text{or} \quad \sigma > -1 -N\min\Big( \frac{1}{p_{1}}, \frac{1}{p^{'}} \Big)
\quad \text{if} \quad \mathrm{div}\,v = 0.
\end{align*}
Then for the transport equation:
\begin{align*}
\begin{cases}
\partial_{t}f + v\cdot\nabla f = g, \\
f|_{t=0} = f_{0},
\end{cases}
\end{align*}
there exists a constant $C$ depending only on $N$, $p$, $p_{1}$, $r$ and $\sigma$, such that the following estimates hold true:
\begin{align}
\label{transport 3}
\|f\|_{\tilde{L}_{t}^{\infty}(B_{p,r}^{\sigma})} \leq
\Big( \|f_{0}\|_{B_{p,r}^{\sigma}} + \int_{0}^{t}e^{-C\int_{0}^{\tau}Z(\tau')\,d\tau'}\|g(\tau)\|_{B_{p,r}^{\sigma}}\,d\tau \Big)
e^{C\int_{0}^{t}Z(\tau)\,d\tau},
\end{align}
with
\begin{align*}
\begin{cases}
Z(t) = \|\nabla v(t)\|_{B_{p_{1},\infty}^{\frac{N}{p_{1}}}\cap L^{\infty}} \quad \text{if} \quad \sigma < 1 + \frac{N}{p_{1}}, \\
Z(t) = \|\nabla v(t)\|_{B_{p_{1},1}^{\sigma - 1}} \quad \text{if} \quad \sigma > 1 + \frac{N}{p_{1}} \quad
\text{or} \quad \left\{ \sigma = 1 + \frac{N}{p_{1}} \quad \text{and} \quad r = 1 \right\}.
\end{cases}
\end{align*}
If $f = v$ then for all $\sigma > 0$ ($\sigma > -1$ if $\mathrm{div}\,v = 0$) estimates (\ref{transport 1}) hold with
\begin{align*}
Z(t) = \|\nabla v(t)\|_{L^{\infty}}.
\end{align*}
\end{lemma}

\begin{lemma}
\cite{zhangping}
\label{pressure_estimate}
Let $s \in (-\frac{3}{2}, 2)$, $\vec{F} = (F_{1}, F_{2}, F_{3}) \in L_{T}^{1}(\dot{B}_{2,1}^{s})$, $a \in \tilde{L}_{T}^{\infty}(\dot{H}^{2})$
with $\underline{a} := \mathrm{inf}_{(t,x) \in [0,T]\times \mathbb{R}^{3}}(1 + a(t,x)) > 0$, and $\Pi \in \tilde{L}_{T}^{1}(\dot{H}^{s + 1/2})$,
which solves
\begin{align}
\mathrm{div}\big( (1+a)\nabla \Pi \big) = \mathrm{div} \vec{F}.
\end{align}
Then there holds
\begin{align}
\underline{a}\|\nabla \Pi\|_{L_{T}^{1}(\dot{B}_{2,1}^{s})} \lesssim \|\vec{F}\|_{L_{T}^{1}(\dot{B}_{2,1}^{s})} +
\|a\|_{\tilde{L}_{T}^{\infty}(\dot{H}^{2})} \|\nabla \Pi\|_{\tilde{L}_{T}^{1}(\dot{H}^{s-1/2})}.
\end{align}
\end{lemma}

\begin{lemma}
\label{linear estimate momentum}
For $s \in (-\frac{3}{2},1)$, $r = 1 \,\, \text{or}\,\, 2$. Let $u_{0} \in \dot{B}_{2,r}^{s}$, $B_{0} \in \dot{B}_{2,r}^{s}$ and
$v \in L_{T}^{1}(\dot{B}_{2,1}^{5/2})$, $w \in L_{T}^{1}(\dot{B}_{2,1}^{5/2})$ be two divergence-free vector field.
Letting $f \in \tilde{L}_{T}^{1}(\dot{B}_{2,r}^{s})$ and $a \in L_{T}^{\infty}(\dot{H}^{2}) \cap L_{T}^{\infty}(\dot{H}^{s+3/2})$ with
$1+a \geq \underline{c} >0$, we assume that $u \in L_{T}^{\infty}(\dot{B}_{2,r}^{s}) \cap L_{T}^{1}(\dot{B}_{2,r}^{s+\frac{3}{2}})$,
$B \in L_{T}^{\infty}(\dot{B}_{2,r}^{s}) \cap L_{T}^{1}(\dot{B}_{2,r}^{s+\frac{3}{2}})$ and $\Pi \in L_{T}^{1}(\dot{H}^{1})$ solve
\begin{align}
\label{linearcouple}
\begin{cases}
\partial_{t} u - w \cdot \nabla B + v \cdot \nabla u - \Delta u + \nabla \Pi = f + a (\Delta u - \nabla \Pi), \quad \mathbb{R}^{+} \times \mathbb{R}^{3},  \\
\partial_{t} B + v \cdot \nabla B - w \cdot \nabla u - \Delta B = g, \\
\mathrm{div} u = \mathrm{div} B = 0,    \\
(u, B)|_{t=0} = (u_{0},B_{0}).
\end{cases}
\end{align}
Then there holds
\begin{align}
\label{linearestimate}
\begin{split}
&  \|(u, B)\|_{\tilde{L}_{T}^{\infty}(\dot{B}_{2,r}^{s})} + \|(u, B)\|_{\tilde{L}_{T}^{1}(\dot{B}_{2,r}^{s+2})}  \\
\lesssim & \, \mathrm{exp}\left( C \int_{0}^{T} \|v(t)\|_{\dot{B}_{2,1}^{5/2}} + \|w(t)\|_{\dot{B}_{2,1}^{5/2}} \, dt \, \right)
\bigg\{ \|(u_{0},B_{0})\|_{\dot{B}_{2,r}^{s}}   \\
&\,\,\,\,\,\,\,\,\,\,\,\,\,\,\,\quad\quad\quad\quad\quad + \|(f, g)\|_{\tilde{L}_{T}^{1}(\dot{B}_{2,r}^{s})}
+ \|a\|_{L_{T}^{\infty}(\dot{H}^{s+3/2})} \|\nabla \Pi\|_{L_{T}^{1}(L^{2})}   \\
&\,\,\,\,\,\quad\quad\quad\quad\quad\quad\quad\quad\quad\quad\quad\quad\quad\quad\quad
+ \|a\|_{L_{T}^{\infty}(\dot{H}^{2})} \|u\|_{L_{T}^{1}(\dot{B}_{2,r}^{s+3/2})} \bigg\}
\end{split}
\end{align}
\end{lemma}
\begin{proof}
We apply the operator $\Delta_{q}$ to (\ref{linearcouple}); then a standard commutator process gives
\begin{align}
\label{equationlocalu}
\begin{split}
& \partial_{t}\Delta_{q}u - (w\cdot \nabla)\Delta_{q}B + (v\cdot \nabla)\Delta_{q}u - \mathrm{div}((1+a)\nabla \Delta_{q}u)  \\
& \quad\quad\quad\quad + \Delta_{q} \nabla ((1+a)\Pi) = -[w, \Delta_{q}]\cdot \nabla B + [v, \Delta_{q}]\cdot \nabla u - (\nabla a \cdot \nabla)\Delta_{q}u   \\
& \quad\quad\quad\quad\quad\quad\quad\quad\quad\quad\quad\quad\quad\, + \Delta_{q} (\nabla a \cdot \Pi) + R_{q} + \Delta_{q}f,
\end{split}
\end{align}
\begin{align}
\label{equationlocalb}
\begin{split}
\partial_{t}\Delta_{q}B + v\cdot \nabla \Delta_{q}B - w\cdot \nabla \Delta_{q}u & - \Delta \Delta_{q}B  \\
& = \Delta_{q}g + [v, \Delta_{q}]\nabla B - [w, \Delta_{q}]\nabla u,
\end{split}
\end{align}
with
\begin{align*}
R_{q} = \mathrm{div}[\Delta_{q}, a]\nabla u + [\nabla a\cdot \nabla, \Delta_{q}] u.
\end{align*}
Thanks to the fact that $\mathrm{div}u = \mathrm{div}v = \mathrm{div}w = \mathrm{div} B = 0$ and $1 + a \geq \underline{c}$, we get by
taking the $L^{2}$ inner product of (\ref{equationlocalu}), (\ref{equationlocalb}) with $\Delta_{q}u$ and $\Delta_{q}B$ separately that
\begin{align*}
\begin{split}
& \frac{1}{2} \frac{d}{dt} \left( \|\Delta_{q}u(t)\|_{L^{2}}^{2} + \|\Delta_{q}B\|_{L^{2}}^{2} \right) +
c \, 2^{2q}\Big( \|\Delta_{q}u\|_{L^{2}}^{2} + \|\Delta_{q}B\|_{L^{2}}^{2} \Big) \\
\lesssim & \, \big(\|\Delta_{q} u\|_{L^{2}} + \|\Delta_{q}B\|_{L^{2}}\big)\big( \|[w, \Delta_{q}] \nabla B\|_{L^{2}} + \|[v, \Delta_{q}]\nabla u\|_{L^{2}}  \\
& \, + \|[v, \Delta_{q}]\nabla B\|_{L^{2}} + \|[w, \Delta_{q}]\nabla u\|_{L^{2}} + \|R_{q}\|_{L^{2}} + \|\Delta_{q}(\Pi \cdot \nabla a)\|_{L^{2}}  \\
& \, + \|\Delta_{q}f\|_{L^{2}} + \|\Delta_{q}g\|_{L^{2}}\big) + \|\Delta a\|_{L^{2}}\|\Delta_{q}u\|_{L^{4}}^{2}.
\end{split}
\end{align*}
Thanks to Lemma \ref{bernsteininequality}, we have
\begin{align*}
\begin{split}
\|\Delta a\|_{L^{2}} \|\Delta_{q}u\|_{L^{4}}^{2} \lesssim 2^{(3/2)q} \|\Delta a\|_{L^{2}} \|\Delta_{q}u\|_{2}^{2}.
\end{split}
\end{align*}
It follows from the product law that for all $s \in (-\frac{3}{2}, 1)$ that
\begin{align*}
\begin{split}
\|\Pi \nabla a(t)\|_{\dot{B}_{2,1}^{s}} \lesssim \|\nabla a(t)\|_{\dot{H}^{s+\frac{1}{2}}} \|\Pi(t)\|_{\dot{H}^{1}}.
\end{split}
\end{align*}
Similar to the proof in Proposition 3.6 in \cite{zhangping}, we have
\begin{align*}
\begin{split}
\|R_{q}\|_{L^{2}} \lesssim d_{q}(t) 2^{-q s} \|a(t)\|_{\dot{H}^{2}} \|u(t)\|_{\dot{H}^{s+3/2}}.
\end{split}
\end{align*}
Using Lemma \ref{comunatorestimate}, we have the following four inequalities:
\begin{align*}
\begin{split}
& \|[w, \Delta_{q}]\nabla B\|_{L^{2}} \lesssim c_{q,r} 2^{-qs} \|w\|_{\dot{B}_{2,1}^{5/2}} \|B\|_{\dot{B}_{2,r}^{s}},   \\
& \|[v, \Delta_{q}]\nabla u\|_{L^{2}} \lesssim c_{q,r} 2^{-qs} \|v\|_{\dot{B}_{2,1}^{5/2}} \|u\|_{\dot{B}_{2,r}^{s}},   \\
& \|[v, \Delta_{q}]\nabla B\|_{L^{2}} \lesssim c_{q,r} 2^{-qs} \|v\|_{\dot{B}_{2,1}^{5/2}} \|B\|_{\dot{B}_{2,r}^{s}},   \\
& \|[w, \Delta_{q}]\nabla u\|_{L^{2}} \lesssim c_{q,r} 2^{-qs} \|w\|_{\dot{B}_{2,1}^{5/2}} \|u\|_{\dot{B}_{2,r}^{s}}.
\end{split}
\end{align*}
Combining the above four estimates, we arrive at
\begin{align*}
& \, \|(u, B)\|_{\tilde{L}_{T}^{\infty}(\dot{B}_{2,r}^{s})} + \|(u, B)\|_{\tilde{L}_{T}^{1}(\dot{B}_{2,r}^{s+2})}   \\
\lesssim & \, \|(u_{0}, B_{0})\|_{\dot{B}_{2,r}^{s}} + \|(f, g)\|_{\tilde{L}_{T}^{1}(\dot{B}_{2,r}^{s})}
+ \|a\|_{L_{T}^{\infty}(\dot{H}^{s+3/2})} \|\nabla \Pi\|_{L_{T}^{1}(L^{2})}  \\
& \quad + \|a\|_{L_{T}^{\infty}(\dot{H}^{2})} \|u\|_{L_{T}^{1}(\dot{B}_{2,r}^{s+3/2})}  \\
& \quad + \int_{0}^{T} \big( \|w\|_{\dot{B}_{2,1}^{5/2}} + \|v\|_{\dot{B}_{2,1}^{5/2}} \big) \big( \|u\|_{\dot{B}_{2,r}^{s}} +
\|B\|_{\dot{B}_{2,r}^{s}} \big) \, d\tau .
\end{align*}
Through Gronwall's inequality, we complete the proof.
\end{proof}
\begin{remark}
It is easy to observe from the following
\begin{align*}
\|\Pi \nabla a\|_{\tilde{L}_{T}^{1}(\dot{B}_{2,r}^{s})} \lesssim \|\nabla a\|_{\tilde{L}_{T}^{\infty}\dot{H}^{1}}
\|\Pi\|_{\tilde{L}_{T}^{1}(\dot{H}^{s+1/2})} \quad \text{for all} \, \, s \in (-\frac{3}{2},1),
\end{align*}
that
\begin{align}
\label{linear_estimate_var}
\begin{split}
&  \|(u, B)\|_{\tilde{L}_{T}^{\infty}(\dot{B}_{2,r}^{s})} + \|(u, B)\|_{\tilde{L}_{T}^{1}(\dot{B}_{2,r}^{s+2})}  \\
\lesssim & \, \mathrm{exp}\left( C \int_{0}^{T} \|(v(t), B(t))\|_{\dot{B}_{2,1}^{5/2}} \, dt \, \right)
\bigg\{ \|(u_{0},B_{0})\|_{\dot{B}_{2,r}^{s}}  \\
& + \|(f, g)\|_{\tilde{L}_{T}^{1}(\dot{B}_{2,r}^{s})}
+ \|a\|_{\tilde{L}_{T}^{\infty}(\dot{H}^{2})}\big( \|\Pi\|_{\tilde{L}_{T}^{1}(\dot{H}^{s+1/2})} + \|u\|_{L_{T}^{1}(\dot{B}_{2,r}^{s+3/2})} \big) \bigg\}.
\end{split}
\end{align}
\end{remark}
\begin{remark}
Note that $\mathrm{div} u = 0$, taking $\mathrm{div}$ to the first equation of (\ref{linearcouple}), we obtain
\begin{align*}
\mathrm{div}\big((1+a)\nabla\Pi \big) = \mathrm{div} \big( f + a\Delta u +w \cdot \nabla B - v \cdot \nabla u \big),
\end{align*}
then is follows from Lemma \ref{pressure_estimate} that for $s \in (-\frac{3}{2}, \frac{3}{2})$,
\begin{align}
\begin{split}
\|\nabla \Pi\|_{L_{T}^{1}(\dot{B}_{2,1}^{s})} \lesssim &\, \|v\|_{L_{T}^{\infty}(\dot{H}^{s})} \|u\|_{L_{T}^{1}(\dot{B}_{2,1}^{5/2})}
+ \|w\|_{L_{T}^{\infty}(\dot{H}^{s})} \|B\|_{L_{T}^{1}(\dot{B}_{2,1}^{5/2})}    \\
&\, + \|a\|_{\tilde{L}_{T}^{\infty}(\dot{B}_{2,1}^{3/2})} \|u\|_{\tilde{L}_{T}^{1}\dot{H}^{s+2}}
+ \|a\|_{\tilde{L}_{T}^{\infty}(\dot{H}^{2})} \|\nabla \Pi\|_{\tilde{L}_{T}^{1}(\dot{H}^{s-1/2})}  \\
&\, + \|f\|_{\tilde{L}_{T}^{1}(\dot{B}_{2,1}^{s})}.
\end{split}
\end{align}
\end{remark}
\begin{remark}
\label{s=1 linear couple}
If the parameter $s = 1$, then we have the following estimation
\begin{align}
\label{estimate s = 1}
\begin{split}
&\, \|u\|_{\tilde{L}_{T}^{\infty}(\dot{B}_{2,1}^{1})} + \|u\|_{\tilde{L}_{T}^{1}(\dot{B}_{2,1}^{3})}
+ \|B\|_{\tilde{L}_{T}^{\infty}(\dot{B}_{2,1}^{1})} + \|B\|_{\tilde{L}_{T}^{1}(\dot{B}_{2,1}^{3})}  \\
\lesssim &\, \exp{\left( C\int_{0}^{T}\|v\|_{\dot{B}_{2,1}^{5/2}} + \|w\|_{\dot{B}_{2,1}^{5/2}} \, dt \right)}
\bigg( \|u_{0}\|_{\dot{B}_{2,1}^{1}} + \|B_{0}\|_{\dot{B}_{2,1}^{1}}   \\
&\, + \|f\|_{\tilde{L}_{T}^{1}(\dot{B}_{2,1}^{1})} + \|g\|_{\tilde{L}_{T}^{1}(\dot{B}_{2,1}^{1})}
+ \|a\|_{L_{T}^{\infty}(\dot{B}_{2,1}^{2})}\|\Pi\|_{L_{T}^{1}(\dot{B}_{2,1}^{3/2})}    \\
&\, + \|a\|_{L_{T}^{\infty}(\dot{B}_{2,1}^{2})}\|u\|_{L_{T}^{1}(\dot{B}_{2,1}^{5/2})} \bigg).
\end{split}
\end{align}
The proof is similar, so we only give the main estimations in the appendix.
\end{remark}

%%%%%%%%%%%%%%%%%%%%%%%%%%%%%%%%%%%%%%%%%%%%%%%%%%%%%%%%%%%%%%%%%%%%%%%%%%%%%%%%%%%%%%%%%%%%%%%%%%%%%%%%%%%%%%%%%%%%%%%%%%%%%%%%%%%%%%%%%%%%%%%%%%%%%%%%%%%

\section{Stability of Global Solutions with Densities Close to 1}

The aim of this section is to investigate the global stability of the given solution of (\ref{mhd_a}) with the initial density of which is close to 1, namely Theorem \ref{close to 1 main theorem}. Next, we give the detailed statement.

\begin{proof}
To deal with the global well-posedness of $(\ref{mhd_a})$ with initial data $(a_{0}, u_{0}, B_{0})$ given by the theorem, we need some global-in-time control of
the reference solution $(\bar{a}, \bar{u}, \bar{B})$. In what follows, we shall always denote $\bar{\rho} := \frac{1}{1+\bar{a}}$.
Then we get by a standard energy estimate to ($\ref{mhd_rho}$) that
\begin{align*}
& \frac{1}{2}\frac{d}{dt} \|\sqrt{\bar{\rho}} \bar{u}(t)\|_{L^{2}}^{2} + \|\nabla \bar{u}(t)\|_{L^{2}}^{2}
- \int_{\mathbb{R}^{3}} (\bar{B}\cdot \nabla)\bar{B} \, \bar{u} \, dx = 0,   \\
& \frac{1}{2}\frac{d}{dt} \|\bar{B}(t)\|_{L^{2}}^{2} + \|\nabla \bar{B}(t)\|_{L^{2}}^{2}
- \int_{\mathbb{R}^{3}} (\bar{B}\cdot \nabla)\bar{u} \, \bar{B} \, dx = 0
\end{align*}
Combining the above two energy equality, we have
\begin{align}
\label{basic diff}
\frac{1}{2}\frac{d}{dt}\left( \|\sqrt{\bar{\rho}}\bar{u}(t)\|_{L^{2}}^{2} + \|\bar{B}(t)\|_{L^{2}}^{2} \right)
+ \|\nabla \bar{u}(t)\|_{L^{2}}^{2} + \|\nabla \bar{B}(t)\|_{L^{2}}^{2} = 0, \quad \text{for} \,\, t > 0.
\end{align}
After integration, we have
\begin{align}
\label{basic_energy_equality}
\frac{1}{2}\|(\sqrt{\bar{\rho}}\bar{u}(t),\bar{B}(t))\|_{L^{2}}^{2}
+ \int_{0}^{t} \|(\nabla \bar{u}(\tau), \nabla \bar{B}(\tau))\|_{L^{2}}^{2}\, d\tau = \frac{1}{2}\|(\sqrt{\bar{\rho}_{0}}\bar{u}_{0}, \bar{B}_{0})\|_{L^{2}}^{2}.
\end{align}
The transport equation in (\ref{mhd_rho}) gives
\begin{align}
\label{basic_energy_rho}
\|\bar{\rho}(t) - 1\|_{L^{p}} = \|\bar{\rho}_{0} - 1\|_{L^{p}},  \quad \text{for all} \, \, p \in [1, \infty].
\end{align}
Using the interpolation inequality
\begin{align*}
& \|\bar{u}\|_{\dot{B}_{2,1}^{1/2}} \lesssim \|\bar{u}\|_{L^{2}}^{1/2} \|\nabla \bar{u}\|_{L^{2}}^{1/2},  \\
& \|\bar{B}\|_{\dot{B}_{2,1}^{1/2}} \lesssim \|\bar{B}\|_{L^{2}}^{1/2} \|\nabla \bar{B}\|_{L^{2}}^{1/2},
\end{align*}
and from (\ref{basic_energy_equality}), we deduce that
\begin{align*}
\int_{0}^{t} \|\bar{u}(\tau)\|_{\dot{B}_{2,1}^{1/2}}^{4} \, d\tau & \lesssim \int_{0}^{t} \|\bar{u}(\tau)\|_{L^{2}}^{2} \|\nabla \bar{u}(\tau)\|_{L^{2}}^{2}\, d\tau   \\
& \lesssim \|\bar{u}_{0}\|_{L^{2}}^{4} + \|\bar{B}_{0}\|_{L^{2}}^{4}, \quad \text{for} \,\, t>0,    \\
\int_{0}^{t} \|\bar{B}(\tau)\|_{\dot{B}_{2,1}^{1/2}}^{4} \, d\tau & \lesssim \int_{0}^{t} \|\bar{B}(\tau)\|_{L^{2}}^{2} \|\nabla \bar{B}(\tau)\|_{L^{2}}^{2}\, d\tau   \\
& \lesssim \|\bar{u}_{0}\|_{L^{2}}^{4} + \|\bar{B}_{0}\|_{L^{2}}^{4}, \quad \text{for} \,\, t>0,
\end{align*}
Hence, for any $\epsilon > 0$, there exists $T_{0} = T_{0}(\epsilon) > 0$ such that
\begin{align}
\label{small}
\begin{split}
& \|\bar{u}(T_{0})\|_{\dot{B}_{2,1}^{1/2}} < \epsilon,    \\
& \|\bar{B}(T_{0})\|_{\dot{B}_{2,1}^{1/2}} < \epsilon.
\end{split}
\end{align}
On the other hand, applying Lemma \ref{transport_estimate} to the transport equation in (\ref{mhd_a}) gives
\begin{align*}
\|\bar{a}(t)\|_{\tilde{L}_{t}^{\infty}(\dot{B}_{2,1}^{3/2})} \leq \|\bar{a}_{0}\|_{\dot{B}_{2,1}^{3/2}}
+ C \int_{0}^{t} \|\bar{a}(\tau)\|_{\dot{B}_{2,1}^{3/2}} \|\nabla \bar{u}(\tau)\|_{\dot{B}_{2,1}^{3/2}} \, d\tau,
\end{align*}
for any $t \geq 0$. Then applying Gronwall's inequality yields
\begin{align*}
\|\bar{a}\|_{\tilde{L}_{T_{0}}^{\infty}(\dot{B}_{2,1}^{3/2})} \leq \|\bar{a}_{0}\|_{\dot{B}_{2,1}^{3/2}}
\mathrm{exp} \left\{ C \int_{0}^{T_{0}} \|\nabla \bar{u}(\tau)\|_{\dot{B}_{2,1}^{3/2}} \, d\tau \right\}.
\end{align*}
Following the same idea, it is easy to obtain that for $t \geq T_{0}$
\begin{align}
\label{close to 1 transport}
\|\bar{a}\|_{\tilde{L}^{\infty}([T_{0},t]; \dot{B}_{2,1}^{3/2})} \leq \|\bar{a}(T_{0})\|_{\dot{B}_{2,1}^{3/2}}
+ C \|\bar{a}\|_{L^{\infty}([T_{0}, t]; \dot{B}_{2,1}^{3/2})} \|\nabla \bar{u}\|_{L^{1}([T_{0}, t]; \dot{B}_{2,1}^{3/2})}.
\end{align}
Note that for $\bar{a}$ small, we can rewrite the momentum equation and magnetic field equation in (\ref{mhd_a}) as
\begin{align*}
& \partial_{t}\bar{u} + (\bar{u}\cdot \nabla)\bar{u} - \Delta \bar{u} + \nabla \bar{\Pi} = \bar{a} (\Delta \bar{u} - \nabla \bar{\Pi})
+ (\bar{B}\cdot \nabla)\bar{B} + \bar{a} (\bar{B}\cdot \nabla) \bar{B}, \\
& \partial_{t}\bar{B} + (\bar{u}\cdot \nabla)\bar{B} - \Delta \bar{B} = (\bar{B}\cdot \nabla)\bar{u},
\end{align*}
from which we get by using a standard energy estimate about heat equation that
\begin{align}
\label{close to 1 momentum}
\begin{split}
& \|(\bar{u},\bar{B})\|_{\tilde{L}^{\infty}([T_{0},t]; \dot{B}_{2,1}^{1/2})} + \|(\bar{u}, \bar{B})\|_{L^{1}([T_{0}, t]; \dot{B}_{2,1}^{5/2})}
+ \|\nabla \bar{\Pi}\|_{L^{1}([T_{0}, t]; \dot{B}_{2,1}^{1/2})} \\
\lesssim & \, \|(\bar{u}(T_{0}), \bar{B}(T_{0}))\|_{\dot{B}_{2,1}^{1/2}} + \|\bar{a} \bar{B}\nabla \bar{B}\|_{L^{1}([T_{0}, t]; \dot{B}_{2,1}^{1/2})}
+ \|\bar{a} (\Delta \bar{u} - \nabla \bar{\Pi})\|_{L^{1}([T_{0}, t]; \dot{B}_{2,1}^{1/2})}  \\
&\, + \left\|\left((\bar{u}\cdot \nabla ) \bar{u}, (\bar{B}\cdot \nabla )\bar{B}, (\bar{u}\cdot \nabla )\bar{B},
(\bar{B}\cdot \nabla )\bar{u}\right)\right\|_{L^{1}([T_{0},t]; \dot{B}_{2,1}^{1/2})}.
\end{split}
\end{align}
For any $t \geq T_{0}$, we denote
\begin{align*}
\bar{Z}(t) := & \,\|\bar{a}\|_{\tilde{L}^{\infty}([T_{0}, t]; \dot{B}_{2,1}^{3/2})} + \|(\bar{u}, \bar{B})\|_{\tilde{L}^{\infty}([T_{0}, t]; \dot{B}_{2,1}^{1/2})}
+ \|\nabla \Pi\|_{L^{1}([T_{0}, t]; \dot{B}_{2,1}^{1/2})}   \\
&\, + \|(\bar{u}, \bar{B})\|_{L^{1}([T_{0}, t]; \dot{B}_{2,1}^{5/2})}.
\end{align*}
From the product law, we get
\begin{align*}
& \|(\bar{u}\cdot \nabla)\bar{u}\|_{L^{1}([T_{0}, t]; \dot{B}_{2,1}^{1/2})} \lesssim \|\bar{u}\|_{L^{\infty}([T_{0}, t]; \dot{B}_{2,1}^{1/2})}
\|\nabla \bar{u}\|_{L^{1}([T_{0}, t]; \dot{B}_{2,1}^{3/2})},    \\
& \|(\bar{B}\cdot \nabla)\bar{B}\|_{L^{1}([T_{0}, t]; \dot{B}_{2,1}^{1/2})} \lesssim  \|\bar{B}\|_{L^{\infty}([T_{0}, t]; \dot{B}_{2,1}^{1/2})}
\|\nabla \bar{B}\|_{L^{1}([T_{0}, t]; \dot{B}_{2,1}^{3/2})},    \\
& \|\bar{a} \, (\bar{B} \cdot \nabla )\bar{B}\|_{L^{1}([T_{0}, t]; \dot{B}_{2,1}^{1/2})}  \\
& \quad\quad\quad\quad \lesssim
\|\bar{a}\|_{L^{\infty}([T_{0}, t]; \dot{B}_{2,1}^{3/2})} \|\bar{B}\|_{L^{\infty}([T_{0}, t]; \dot{B}_{2,1}^{1/2})}
\|\nabla \bar{B}\|_{L^{1}([T_{0}, t]; \dot{B}_{2,1}^{3/2})},    \\
& \|\bar{a} (\Delta \bar{u} - \nabla \bar{\Pi})\|_{L^{1}([T_{0}, t]; \dot{B}_{2,1}^{1/2})}  \\
& \quad\quad\quad\quad \lesssim
\|\bar{a}\|_{L^{\infty}([T_{0}, t]; \dot{B}_{2,1}^{3/2})} \left( \|\bar{u}\|_{L^{1}([T_{0}, t]; \dot{B}_{2,1}^{5/2})}
+ \|\nabla \bar{\Pi}\|_{L^{1}([T_{0}, t]; \dot{B}_{2,1}^{1/2})} \right)   \\
& \|(\bar{u}\cdot \nabla)\bar{B}\|_{L^{1}([T_{0}, t]; \dot{B}_{2,1}^{1/2})} \lesssim \|\bar{u}\|_{L^{\infty}([T_{0}, t]; \dot{B}_{2,1}^{1/2})}
\|\nabla \bar{B}\|_{L^{1}([T_{0}, t]; \dot{B}_{2,1}^{3/2})},    \\
& \|(\bar{B}\cdot \nabla)\bar{u}\|_{L^{1}([T_{0}, t]; \dot{B}_{2,1}^{1/2})} \lesssim \|\bar{B}\|_{L^{\infty}([T_{0}, t]; \dot{B}_{2,1}^{1/2})}
\|\nabla \bar{u}\|_{L^{1}([T_{0}, t]; \dot{B}_{2,1}^{3/2})},    \\
\end{align*}
Then we get by summing up (\ref{close to 1 transport}), (\ref{close to 1 momentum}) and the above estimates that
\begin{align*}
\bar{Z}(t) \leq \|\bar{a}(T_{0})\|_{\dot{B}_{2,1}^{3/2}} + \|(\bar{u}(T_{0}), \bar{B}(T_{0}))\|_{\dot{B}_{2,1}^{1/2}}
+ C \, \bar{Z}(t)^{2} + C \, \bar{Z}(t)^{3}.
\end{align*}
Let
\begin{align}
\label{definition_of_T}
\bar{T} := \sup_{t > T_{0}} \left\{ t \, : \, \bar{Z}(t) \leq 2\, \left( \|\bar{a}(T_{0})\|_{\dot{B}_{2,1}^{3/2}} +
\|(\bar{u}(T_{0}), \bar{B}(T_{0}))\|_{\dot{B}_{2,1}^{1/2}} \right) \right\}
\end{align}
With no loss of generality, we can assume that
\begin{align}
\|\bar{a}(T_{0})\|_{\dot{B}_{2,1}^{3/2}} + \|\bar{u}(T_{0})\|_{\dot{B}_{2,1}^{1/2}} + \|\bar{B}(T_{0})\|_{\dot{B}_{2,1}^{1/2}} \leq 1.
\end{align}
Then, if $\bar{T} < \infty$, for $T_{0} < t < \bar{T}$, we have
\begin{align*}
\bar{Z}(t) \leq \left( \|\bar{a}(T_{0})\|_{\dot{B}_{2,1}^{3/2}} + \|(\bar{u}(T_{0}),\bar{B}(T_{0}))\|_{\dot{B}_{2,1}^{1/2}} \right)
\left( 1+6C\bar{Z}(t) \right)
\end{align*}
Taking $\epsilon$ in (\ref{small}) small enough so that $\epsilon \leq \frac{1}{108 C}$, then if $c_{1}$ in Theorem \ref{close to 1 main theorem}
is so small that $c_{1} \leq \frac{1}{108 C}$, we have
\begin{align}
\label{bootstrap}
\bar{Z}(t) \leq \frac{3}{2} \left( \|\bar{a}(T_{0})\|_{\dot{B}_{2,1}^{3/2}} + \|(\bar{u}(T_{0}),\bar{B}(T_{0}))\|_{\dot{B}_{2,1}^{1/2}} \right)
\quad \text{for} \,\, T_{0} \leq t \leq \bar{T}.
\end{align}
This contradicts with the definition of (\ref{definition_of_T}), and therefore $\bar{T} = \infty$.
Moreover, there hold
\begin{align}
\label{global_estimate_small}
\begin{split}
& \|\bar{a}\|_{\tilde{L}^{\infty}(\mathbb{R}^{+}; \dot{B}_{2,1}^{3/2})} \leq 2 \, \|\bar{a}_{0}\|_{\dot{B}_{2,1}^{3/2}}
\mathrm{exp}\left\{ C \int_{0}^{T_{0}} \|\nabla \bar{u}(\tau)\|_{\dot{B}_{2,1}^{3/2}}\, d\tau \right\} + 4 \epsilon,    \\
& \|(\bar{u}, \bar{B})\|_{\tilde{L}^{\infty}(\mathbb{R}^{+}; \dot{B}_{2,1}^{1/2})} + \|(\bar{u}, \bar{B})\|_{L^{1}(\mathbb{R}^{+}; \dot{B}_{2,1}^{5/2})}
+ \|\nabla \bar{\Pi}\|_{L^{1}(\mathbb{R}^{+}; \dot{B}_{2,1}^{1/2})} \leq C.
\end{split}
\end{align}
With ($\ref{global_estimate_small}$), we can solve the global well-posedness with initial data
$(a_{0}, u_{0}, B_{0}) = (\bar{a}_{0}+\tilde{a}_{0}, \bar{u}_{0}+\tilde{u}_{0}, \bar{B}_{0}+\tilde{B}_{0})$ for
$(\tilde{a}_{0}, \tilde{u}_{0}, \tilde{B}_{0})$ sufficiently small. Let $\tilde{u} := u - \bar{u}$, $\tilde{B} := B - \bar{B}$, then
$(a, \tilde{u}, \tilde{B})$ solves
\begin{align}
\label{close to 1 perturbation}
\begin{cases}
\partial_{t}a + (\bar{u} + \tilde{u}) \nabla a = 0, \\
\partial_{t}\tilde{u} - \Delta \tilde{u} + \nabla \tilde{\Pi} = -(\bar{u} + \tilde{u})\nabla \tilde{u} - (\tilde{u}\cdot \nabla)\bar{u}
+ a \, (\Delta \tilde{u} - \nabla \tilde{\Pi})  \\
\quad\quad\quad\quad\quad\quad\quad\quad\,\, + (a-\bar{a})(\Delta \bar{u} - \nabla \bar{\Pi})
+ (a-\bar{a})(\bar{B}\cdot\nabla\bar{B})    \\
\quad\quad\quad\quad\quad\quad\quad\quad\,\, + (1+a)(\bar{B} + \tilde{B})\nabla \tilde{B}
+ (1+a)(\tilde{B}\cdot \nabla) \bar{B},  \\
\partial_{t}\tilde{B} - \Delta \tilde{B} = -(\tilde{u} + \bar{u})\nabla \tilde{B} - (\tilde{u}\cdot \nabla)\bar{B}
+ (\tilde{B}+\bar{B})\nabla \tilde{u} + (\tilde{B}\cdot \nabla)\bar{u}, \\
\mathrm{div} \tilde{u} = \mathrm{div} \tilde{B} = 0,    \\
(a, \tilde{u}, \tilde{B})|_{t=0} = (\bar{a}_{0} + \tilde{a}_{0}, \tilde{u}_{0}, \tilde{B}_{0}).
\end{cases}
\end{align}
Applying Lemma \ref{transport_estimate} to the first equation of (\ref{close to 1 perturbation}) gives
\begin{align}
\begin{split}
\|a\|_{\tilde{L}_{t}^{\infty}(\dot{B}_{2,1}^{3/2})} \leq &\, \|a_{0}\|_{\dot{B}_{2,1}^{3/2}} + C \int_{0}^{t} \|a(\tau)\|_{\dot{B}_{2,1}^{3/2}}
\|\bar{u}(\tau)\|_{\dot{B}_{2,1}^{5/2}} \, d\tau    \\
&\, + C \, \|a\|_{L_{t}^{\infty}(\dot{B}_{2,1}^{3/2})} \|\tilde{u}\|_{L_{t}^{1}(\dot{B}_{2,1}^{5/2})},
\end{split}
\end{align}
which along with (\ref{global_estimate_small}) and the Gronwall's inequality ensures
\begin{align}
\label{perturb a}
\|a\|_{\tilde{L}_{t}^{\infty}(\dot{B}_{2,1}^{3/2})} \lesssim \left( \|a_{0}\|_{\dot{B}_{2,1}^{3/2}} + \|a\|_{L_{t}^{\infty}(\dot{B}_{2,1}^{3/2})}
\|\tilde{u}\|_{L_{t}^{1}(\dot{B}_{2,1}^{5/2})} \right).
\end{align}
Reformulate equation (\ref{close to 1 perturbation}) as follows
\begin{align}
\label{var close to 1 perturbation}
\begin{cases}
\partial_{t}\tilde{u} - (\bar{B}\cdot\nabla)\tilde{B} + (\bar{u}\cdot\nabla)\tilde{u} - \Delta \tilde{u} + \nabla \tilde{\Pi} = \tilde{f}
+ a\, (\Delta \tilde{u} - \nabla \tilde{\Pi}),  \\
\partial_{t}\tilde{B} + (\bar{u}\cdot\nabla)\tilde{B} - (\bar{B}\cdot\nabla)\tilde{u} - \Delta \tilde{B} = \tilde{g},
\end{cases}
\end{align}
with
\begin{align*}
& \tilde{f} = -(\tilde{u}\cdot\nabla)\tilde{u} - (\tilde{u}\cdot\nabla)\bar{u} + (a-\bar{a})(\Delta \bar{u} - \nabla \bar{\Pi})
+ (a-\bar{a})(\bar{B}\cdot\nabla\bar{B}) \\
& \quad\quad + (\tilde{B}\cdot\nabla)\tilde{B} + (\tilde{B}\cdot\nabla)\bar{B}
+ a \,\tilde{B}\cdot\nabla\tilde{B} + a \, \bar{B}\cdot\nabla\tilde{B} + a\, \tilde{B}\cdot\nabla\bar{B}  \\
& \tilde{g} = -(\tilde{u}\cdot\nabla)\tilde{B} - (\tilde{u}\cdot\nabla)\bar{B} + (\tilde{B}\cdot\nabla)\tilde{u} + (\tilde{B}\cdot\nabla)\bar{u}.
\end{align*}
Using standard estimates, we easily obtain that
\begin{align*}
&\, \|(\tilde{u}, \tilde{B})\|_{\tilde{L}_{t}^{\infty}(\dot{B}_{2,1}^{1/2})} + \|(\tilde{u}, \tilde{B})\|_{\tilde{L}_{t}^{1}(\dot{B}_{2,1}^{5/2})}
+ \|\nabla\tilde{\Pi}\|_{L_{t}^{1}(\dot{B}_{2,1}^{1/2})} \\
\lesssim &\, \|(\tilde{u}_{0}, \tilde{B}_{0})\|_{\dot{B}_{2,1}^{1/2}} + \int_{0}^{t} \|(\tilde{u}, \tilde{B})\|_{\dot{B}_{2,1}^{1/2}}
\|(\tilde{u}, \tilde{B})\|_{\dot{B}_{2,1}^{5/2}} \, d\tau   \\
&\, + \int_{0}^{t} \|(\tilde{u}, \tilde{B})\|_{\dot{B}_{2,1}^{1/2}} \|(\bar{u}, \bar{B})\|_{\dot{B}_{2,1}^{5/2}} \, d\tau
+ \int_{0}^{t} \|(a-\bar{a})(\Delta \bar{u} - \nabla \bar{\Pi})\|_{\dot{B}_{2,1}^{1/2}} \, d\tau    \\
&\, +\int_{0}^{t} \|a \, (\Delta \tilde{u} - \nabla \tilde{\Pi})\|_{\dot{B}_{2,1}^{1/2}}\, d\tau
+ \int_{0}^{t} \|(a-\bar{a})\bar{B}\cdot\nabla\bar{B}\|_{\dot{B}_{2,1}^{1/2}}\,d\tau    \\
&\, + \int_{0}^{t}\|a\|_{\dot{B}_{2,1}^{3/2}}\|\tilde{B}\|_{\dot{B}_{2,1}^{1/2}}\|\nabla\tilde{B}\|_{\dot{B}_{2,1}^{3/2}}\,d\tau
+ \int_{0}^{t}\|a\|_{\dot{B}_{2,1}^{3/2}}\|\bar{B}\|_{\dot{B}_{2,1}^{1/2}}\|\nabla\tilde{B}\|_{\dot{B}_{2,1}^{3/2}}\,d\tau  \\
&\, + \int_{0}^{t}\|a\|_{\dot{B}_{2,1}^{3/2}}\|\tilde{B}\|_{\dot{B}_{2,1}^{1/2}}\|\nabla\bar{B}\|_{\dot{B}_{2,1}^{3/2}}\,d\tau
\end{align*}
Then Gronwall's inequality ensures that
\begin{align}
\label{perture u B}
\begin{split}
&\, \|(\tilde{u}, \tilde{B})\|_{\tilde{L}_{t}^{\infty}(\dot{B}_{2,1}^{1/2})} + \|(\tilde{u}, \tilde{B})\|_{L_{t}^{1}(\dot{B}_{2,1}^{5/2})}
+ \|\nabla \tilde{\Pi}\|_{L_{t}^{1}(\dot{B}_{2,1}^{1/2})} \\
\lesssim &\, \|(\tilde{u}_{0}, \tilde{B}_{0})\|_{\dot{B}_{2,1}^{1/2}} + \|(\tilde{u}, \tilde{B})\|_{L_{t}^{\infty}(\dot{B}_{2,1}^{1/2})}
\|(\tilde{u}, \tilde{B})\|_{L_{t}^{1}(\dot{B}_{2,1}^{5/2})} \\
&\, + \|a\|_{\tilde{L}_{t}^{\infty}(\dot{B}_{2,1}^{3/2})} \left( \|(\tilde{u},\tilde{B})\|_{L_{t}^{1}(\dot{B}_{2,1}^{5/2})}
+ \|\tilde{B}\|_{\tilde{L}_{t}^{\infty}(\dot{B}_{2,1}^{1/2})} + \|\nabla \tilde{\Pi}\|_{L_{t}^{1}(\dot{B}_{2,1}^{1/2})} \right)  \\
&\, + \|\bar{a}\|_{L_{t}^{\infty}(\dot{B}_{2,1}^{3/2})}
+ \|a\|_{\tilde{L}_{t}^{\infty}(\dot{B}_{2,1}^{3/2})}\|\tilde{B}\|_{\tilde{L}_{t}^{\infty}(\dot{B}_{2,1}^{1/2})}
\|\tilde{B}\|_{L_{t}^{1}(\dot{B}_{2,1}^{5/2})}  \\
&\, + \int_{0}^{t} \|a\|_{\dot{B}_{2,1}^{3/2}} \left( \|\bar{u}\|_{\dot{B}_{2,1}^{5/2}}
+ \|\nabla \bar{\Pi}\|_{\dot{B}_{2,1}^{1/2}}+\|\bar{B}\|_{\dot{B}_{2,1}^{1/2}}\|\nabla\bar{B}\|_{\dot{B}_{2,1}^{3/2}} \right) \, d\tau
\end{split}
\end{align}
Let
\begin{align*}
\tilde{Z}(t) := \|a\|_{\tilde{L}_{t}^{\infty}(\dot{B}_{2,1}^{3/2})} + \|(\tilde{u}, \tilde{B})\|_{\tilde{L}_{t}^{\infty}(\dot{B}_{2,1}^{1/2})}
+ \|(\tilde{u}, \tilde{B})\|_{L_{t}^{1}(\dot{B}_{2,1}^{5/2})} + \|\nabla \tilde{\Pi}\|_{L_{t}^{1}(\dot{B}_{2,1}^{1/2})}.
\end{align*}
Summing up (\ref{perturb a}) and (\ref{perture u B}), and then applying Gronwall's inequality, we obtain
\begin{align}\label{fanbo 1}
\tilde{Z}(t) \leq \, C \, \left( \|(\tilde{u}_{0}, \tilde{B}_{0})\|_{\dot{B}_{2,1}^{1/2}} + \|\tilde{a}_{0}\|_{\dot{B}_{2,1}^{3/2}}
+ \|\bar{a}\|_{L_{t}^{\infty}(\dot{B}_{2,1}^{3/2})} + \tilde{Z}(t)^{2} + \tilde{Z}(t)^{3} \right).
\end{align}
Using similar arguments used in proving (\ref{bootstrap}), we can show that if
$\|(\tilde{u}_{0}, \tilde{B}_{0})\|_{\dot{B}_{2,1}^{1/2}} + \|\tilde{a}_{0}\|_{\dot{B}_{2,1}^{3/2}} + 2 c_{1} + 4 \epsilon \leq \frac{1}{12 \, C^{2} }$,
there holds
\begin{align}
\label{small global}
\tilde{Z}(t) \leq \, 2\, C \, \left( \|(\tilde{u}_{0}, \tilde{B}_{0})\|_{\dot{B}_{2,1}^{1/2}} + \|\tilde{a}_{0}\|_{\dot{B}_{2,1}^{3/2}}
+ \|\bar{a}\|_{\tilde{L}_{t}^{\infty}(\dot{B}_{2,1}^{3/2})} \right)\quad \text{for all } t>0.
\end{align}
With (\ref{small global}), we can prove the propagation of regularity for smoother initial data.
Applying Lemma \ref{transport_estimate}, we obtain for any $t>0$,
\begin{align*}
\|a\|_{\tilde{L}_{t}^{\infty}(\dot{B}_{2,1}^{5/2})} \leq \|a_{0}\|_{\dot{B}_{2,1}^{5/2}} + C \int_{0}^{t} \|a(\tau)\|_{\dot{B}_{2,1}^{5/2}}
\|\nabla u(\tau)\|_{\dot{B}_{2,1}^{3/2}}\, d\tau.
\end{align*}
This along with (\ref{global_estimate_small}) and (\ref{small global}) implies
\begin{align}
\label{part 1 a}
\begin{split}
\|a\|_{\tilde{L}_{t}^{\infty}(\dot{B}_{2,1}^{5/2})} \leq \,C \, \|a_{0}\|_{\dot{B}_{2,1}^{5/2}}.
\end{split}
\end{align}
Noticing that $\|a(t)\|_{L^{2}} = \|a_{0}\|_{L^{2}}$, this along with the above inequality shows that
$a \in \tilde{L}^{\infty}(\mathbb{R}^{+}; B_{2,1}^{5/2}(\mathbb{R}^{3}))$.
Applying standard energy estimate to the second and third equation of (\ref{mhd_a}) gives
\begin{align*}
&\, \|(u,B)\|_{\tilde{L}_{t}^{\infty}(\dot{B}_{2,1}^{s})} + \|(u, B)\|_{L_{t}^{1}(\dot{B}_{2,1}^{s+2})} + \|\nabla \Pi\|_{L_{t}^{1}(\dot{B}_{2,1}^{s})} \\
\leq &\, \|(u_{0}, B_{0})\|_{\dot{B}_{2,1}^{s}} + \|a\|_{L_{t}^{\infty}(\dot{B}_{2,1}^{3/2})}\left( \|u\|_{L_{t}^{1}(\dot{B}_{2,1}^{s+2})}
+ \|\nabla \Pi\|_{L_{t}^{1}(\dot{B}_{2,1}^{s})} \right) \\
&\, + C\, \left(1+\|a\|_{L_{t}^{\infty}(\dot{B}_{2,1}^{3/2})}\right) \int_{0}^{t} \|(\nabla u(\tau), \nabla B(\tau))\|_{L^{\infty}}
\|(u(\tau), B(\tau))\|_{\dot{B}_{2,1}^{s}} \, d\tau     \\
&\, + \int_{0}^{t} \|(\nabla u, \nabla B)\|_{\dot{B}_{2,1}^{3/2}} \|B\|_{\dot{B}_{2,1}^{s}}\, d\tau,
\end{align*}
for $s \in [0,\frac{3}{2}]$ and $t>0$, this together with (\ref{global_estimate_small}) and (\ref{small global}) gives rise to
\begin{align}
\label{part 1 u B}
\|(u, B)\|_{\tilde{L}_{t}^{\infty}(\dot{B}_{2,1}^{s})} + \|(u, B)\|_{L_{t}^{1}(\dot{B}_{2,1}^{s+2})} + \|\nabla \Pi\|_{L_{t}^{1}(\dot{B}_{2,1}^{s})} \leq C
\quad \text{for all } t>0.
\end{align}
With (\ref{part 1 a}) and (\ref{part 1 u B}), we can easily prove by a classical argument that
\begin{align*}
& a\in C_{b}([0, \infty); B_{2,1}^{5/2}),      \\
& u,B \in C_{b}([0, \infty) ; B_{2,1}^{3/2}).
\end{align*}
Now, the proof of Theorem \ref{close to 1 main theorem} is completed.
\end{proof}

%%%%%%%%%%%%%%%%%%%%%%%%%%%%%%%%%%%%%%%%%%%%%%%%%%%%%%%%%%%%%%%%%%%%%%%%%%%%%%%%%%%%%%%%%%%%%%%%%%%%%%%%%%%%%%%%%%%%%%%%%%%%%%%%%%%%%%%%%%%%%%%%%%%%%%%%%%%

\section{Decay in Time Estimates for the Reference Solutions}

In this section, we will give the decay estimates, namely Theorem \ref{decay_main_theorem}. The main ingredient of the proof will be H. Abidi, G. Gui
and P. Zhang's approach in \cite{zhangping}. The difference is that we need to give the decay estimates for more higher order derivatives of momentum and
magnetic fields, which is required for the global in time estimates proved in the next section.

In what follows, we shall always denote $\rho(t,x) := \frac{1}{1+a(t,x)}$ so that we can use both (\ref{mhd_a}) and (\ref{mhd_rho}) just
according to our convenience. In order to make our presentation clearly, we divided the proof of Theorem \ref{decay_main_theorem} into the following
propositions:

\begin{proposition}
\label{decay pro 1}
Under the same assumptions of Theorem \ref{decay_main_theorem}, there exists $t_{0} > 0$ and two positive constants $e_{1}$ and $e_{2}$ such that
there holds
\begin{align}
\begin{split}
\frac{d}{dt} \|(\nabla u(t), \nabla B(t))\|_{L^{2}}^{2}
&\, + e_{1}\|(\partial_{t} u(t), \partial_{t} B(t))\|_{L^{2}}^{2}   \\
& \, + e_{2} \|(\nabla^{2}u(t), \nabla^{2}B(t))\|_{L^{2}}^{2} \leq 0 \quad \text{for all } t\geq t_{0},
\end{split}
\end{align}
or consequently
\begin{align}
\label{4.2}
\begin{split}
& \sup_{t\geq t_{0}}\|(\nabla u(t), \nabla B(t))\|_{L^{2}}^{2}
 + \int_{t_{0}}^{\infty}e_{1}\|(\partial_{t}u(t), \partial_{t}B(t))\|_{L^{2}}^{2} \, d\tau   \\
& \quad\quad\quad\quad\quad
+ \int_{t_{0}}^{\infty}e_{2}\|(\nabla^{2}u(t), \nabla^{2}B(t))\|_{L^{2}}^{2} \, d\tau \leq \|(\nabla u(t_{0}), \nabla B(t_{0}))\|_{L^{2}}^{2}.
\end{split}
\end{align}
\end{proposition}
\begin{proof}
Taking $L^{2}$ inner product of the momentum equation of (\ref{mhd_rho}) with $\frac{1}{\rho}\Delta u$ that
\begin{align}
\label{decay 1 momentum}
\begin{split}
\frac{1}{2}\frac{d}{dt}\|\nabla u(t)\|_{L^{2}}^{2} + \|\frac{1}{\sqrt{\rho}}\Delta u(t)\|_{L^{2}}^{2}
= &\, \int_{\mathbb{R}^{3}} (u\cdot \nabla)u \, \Delta u \, dx  \\
+ &\, \int_{\mathbb{R}^{3}}\frac{1}{\rho}\,\nabla \Pi \, \Delta u \, dx - \int_{\mathbb{R}^{3}}\frac{1}{\rho}(B\cdot \nabla)B \, \Delta u \, dx.
\end{split}
\end{align}
Taking $L^{2}$ inner product of the magnetic fields equation of (\ref{mhd_rho}) with $\Delta B$ that
\begin{align}
\label{decay 1 magnetic}
\begin{split}
\frac{1}{2}\frac{d}{dt}\|\nabla B(t)\|_{L^{2}}^{2} + \|\Delta B(t)\|_{L^{2}}^{2} = \int_{\mathbb{R}^{3}} (u\cdot \nabla)B \, \Delta B \, dx
- \int_{\mathbb{R}^{3}} (B\cdot \nabla)u \, \Delta B \, dx.
\end{split}
\end{align}
Thanks to the equation of (\ref{mhd_rho}) and $\mathrm{div} u = 0$, we have
\begin{align}
\label{decay 1 pressure}
\begin{split}
&\, \|\Delta u(t)\|_{L^{2}} + \|\nabla \Pi\|_{L^{2}}  \\
\leq &\, \sqrt{2} \|\Delta u(t) - \nabla \Pi(t)\|_{L^{2}}   \\
\leq &\, \sqrt{2} \|\rho \partial_{t}u + \rho u \cdot \nabla u - B\cdot \nabla B\|_{L^{2}}  \\
\leq &\, C\, \|\sqrt{\rho} \partial_{t}u\|_{L^{2}} + C\, \|u(t)\|_{L^{2}}^{1/2}\|\nabla u(t)\|_{L^{2}}^{1/2} \|\Delta u(t)\|_{L^{2}} \\
&\, \quad\quad\quad + C \, \|B(t)\|_{L^{2}}^{1/2}\|\nabla B(t)\|_{L^{2}}^{1/2} \|\Delta B(t)\|_{L^{2}}.
\end{split}
\end{align}
Combining (\ref{decay 1 momentum}), (\ref{decay 1 magnetic}) and (\ref{decay 1 pressure}), we obtain
\begin{align}
\label{decay 1 first}
\begin{split}
&\, \frac{d}{dt}\|(\nabla u(t), \nabla B(t))\|_{L^{2}}^{2} + c\, \|(\nabla^{2}u(t), \nabla^{2}B(t))\|_{L^{2}}^{2}   \\
\leq &\, C\, \|(u(t), B(t))\|_{L^{2}}^{1/2} \|(\nabla u(t), \nabla B(t))\|_{L^{2}}^{1/2} \|(\nabla^{2}u(t), \nabla^{2}B(t))\|_{L^{2}}^{2}   \\
&\,  + C\, \|\sqrt{\rho}\partial_{t}u(t)\|_{L^{2}}^{2},
\end{split}
\end{align}
for some positive constant $c$.

On the other hand, taking $L^{2}$ inner product of the momentum equations in (\ref{mhd_rho}) with $\partial_{t}u$, we obtain
\begin{align*}
\begin{split}
\|\sqrt{\rho} \partial_{t}u(t)\|_{L^{2}}^{2} + \frac{1}{2}\frac{d}{dt} \|\nabla u(t)\|_{L^{2}}^{2}
= - \int_{\mathbb{R}^{3}} \rho \, u \cdot \nabla u \, \partial_{t}u \, dx + \int_{\mathbb{R}^{3}} B\cdot \nabla B \, \partial_{t} u\, dx.
\end{split}
\end{align*}
Taking $L^{2}$ inner product of the magnetic field equations in (\ref{mhd_rho}) with $\partial_{t}B$, we have
\begin{align*}
\begin{split}
\|\partial_{t}B(t)\|_{L^{2}}^{2} + \frac{1}{2}\frac{d}{dt}\|\nabla B(t)\|_{L^{2}}^{2} =
-\int_{\mathbb{R}^{3}} u\cdot \nabla B \, \partial_{t}B\, dx + \int_{\mathbb{R}^{3}} B\cdot \nabla u \, \partial_{t}B \, dx.
\end{split}
\end{align*}
The above two equalities gives
\begin{align}
\label{decay 1 second}
\begin{split}
\frac{d}{dt}\|(\nabla u(t), \nabla B(t))\|_{L^{2}}^{2} &\, + \|(\sqrt{\rho} \partial_{t}u(t), \partial_{t}B(t))\|_{L^{2}}^{2}   \\
&\, \leq C \, \|(u, B)\|_{L^{2}} \|(\nabla u, \nabla B)\|_{L^{2}} \|(\nabla^{2}u, \nabla^{2}B)\|_{L^{2}}^{2}.
\end{split}
\end{align}
The above inequalities along with (\ref{decay 1 first}) ensures a positive constant $e_{1}$ such that
\begin{align}
\label{decay 1 last}
\begin{split}
 \frac{d}{dt}\|(\nabla u, \nabla B)\|_{L^{2}}^{2} + e_{1}\|(\partial_{t}u, \partial_{t}B)\|_{L^{2}}^{2}
 + \mathcal{A}\|(\nabla^{2}u, \nabla^{2}B)\|_{L^{2}}^{2} \leq 0.
\end{split}
\end{align}
where
\begin{align*}
\mathcal{A}:= \frac{c}{2\, C} - \frac{1}{2}\|(u, B)\|_{L^{2}}^{1/2}\|(\nabla u, \nabla B)\|_{L^{2}}^{1/2}
- C \, \|(u, B)\|_{L^{2}}\|(\nabla u, \nabla B)\|_{L^{2}}
\end{align*}
By (\ref{basic_energy_equality}), for any $\eta > 0$, there exists $t_{0} = t_{0}(\eta) > 0$ such that
\begin{align*}
\|\nabla u(t_{0})\|_{L^{2}} + \|\nabla B(t_{0})\|_{L^{2}} \leq \eta.
\end{align*}
Now choosing $\eta > 0$ small enough such that
\begin{align}
\label{eta}
\eta^{1/2} \|(u_{0}, B_{0})\|_{L^{2}}^{1/2}\left(1 + \|(u_{0}, B_{0})\|_{L^{2}}^{1/2} \eta^{1/2} \right) \leq \frac{c}{16\, C^{2}},
\end{align}
we define
\begin{align}
\tau^{*}:=\sup\left\{ t \geq t_{0} \, : \, \|(\nabla u(t), \nabla B(t))\|_{L^{2}} \leq 2 \eta \right\}.
\end{align}
We claim that $\tau^{*} = \infty$. Indeed, if $\tau^{*} < \infty$, (\ref{decay 1 last}) and (\ref{eta}) imply that
\begin{align*}
\begin{split}
\frac{d}{dt}\|(\nabla u, \nabla B)\|_{L^{2}}^{2} & + e_{1} \|(\partial_{t}u, \partial_{t}B)\|_{L^{2}}^{2}     \\
& + \frac{c}{4\,C} \|(\nabla^{2}u, \nabla^{2}B)\|_{L^{2}}^{2} \leq 0 \quad \text{for } t \in [t_{0}, \tau^{*}],
\end{split}
\end{align*}
which gives
\begin{align*}
\|(\nabla u, \nabla B)\|_{L^{2}}^{2} + \int_{t_{0}}^{\tau^{*}} e_{1} \|(\partial_{t}u, \partial_{t}B)\|_{L^{2}}^{2} \, d\tau
& + \int_{t_{0}}^{\tau^{*}} \frac{c}{4\, C} \|(\nabla^{2}u, \nabla^{2}B)\|_{L^{2}}^{2}\, d\tau    \\
& \leq \|(\nabla u(t_{0}), \nabla B(t_{0}))\|_{L^{2}}^{2} \leq \eta^{2}.
\end{align*}
This is a contradiction to the definition of $\tau^{*}$, and thus $\tau^{*} = \infty$. Then the proof is completed.
\end{proof}

\begin{proposition}
\label{decay pro 2}
Under the assumptions of Theorem \ref{decay_main_theorem}, there holds $u,B \in C([0, \infty); L^{p}(\mathbb{R}^{3}))$ where
$p \in (1,\frac{6}{5})$.
\end{proposition}
\begin{proof}
Multiplying the $u^{i}$ equation in (\ref{mhd_rho}) by $|u^{i}|^{p-1}\mathrm{sgn}(u^{i})$ for $i=1,2,3$ and integrating the resulting equation
over $\mathbb{R}^{3}$, we obtain that
\begin{align*}
\begin{split}
\frac{1}{p}\frac{d}{dt}\int_{\mathbb{R}^{3}}\rho |u^{i}|^{p} \, dx & + \frac{4\, (p-1)}{p^{2}} \int_{\mathbb{R}^{3}} |\nabla |u^{i}|^{p/2}|^{2} \,dx  \\
& \lesssim \int_{\mathbb{R}^{3}}B\cdot\nabla B^{i} \, |u^{i}|^{p-1}\mathrm{sgn}(u^{i})\, dx + \|\nabla \Pi\|_{L^{p}}\|u^{i}\|_{L^{p}}^{p-1}.
\end{split}
\end{align*}
Using similar ideas, we have
\begin{align*}
\begin{split}
& \frac{1}{p}\frac{d}{dt}\int_{\mathbb{R}^{3}}|B^{i}|^{p}\, dx + \frac{4\, (p-1)}{p^{2}} \int_{\mathbb{R}^{3}}|\nabla |B^{i}|^{p/2}|^{2}\, dx \\
= & - \int_{\mathbb{R}^{3}} u\cdot \nabla B^{i} \, |B^{i}|^{p-1} \mathrm{sgn}(B^{i})\, dx +
\int_{\mathbb{R}^{3}}B\cdot \nabla u^{i}\, |B^{i}|^{p-1} \mathrm{sgn}(B^{i})\, dx
\end{split}
\end{align*}
Using H\'{o}lder's inequalities, we obtain
\begin{align*}
\begin{split}
\frac{d}{dt} \|\rho^{1/p}u\|_{L^{p}}^{p} & + \frac{4\, (p-1)}{p} \int_{\mathbb{R}^{3}}|\nabla |u|^{p/2}|^{2} \, dx    \\
& \lesssim \|B\|_{\frac{2p}{2-p}} \|\nabla B\|_{L^{2}} \|u\|_{L^{p}}^{p-1} + \|\nabla \Pi\|_{L^{p}} \|u\|_{L^{p}}^{p-1}.
\end{split}
\end{align*}
Considering $p \in (1,\frac{6}{5})$, we can easily obtain that
\begin{align}
\label{decay 2 u}
\begin{split}
\frac{d}{dt} \|\rho^{1/p}u\|_{L^{p}}^{p} + & \frac{4\, (p-1)}{p} \int_{\mathbb{R}^{3}}|\nabla |u|^{p/2}|^{2}\, dx     \\
& \lesssim \left( \|B\|_{L^{2}\cap L^{3}}\|\nabla B\|_{L^{2}} + \|\nabla \Pi\|_{L^{p}} \right)\|u\|_{L^{p}}^{p-1}.
\end{split}
\end{align}
Also by simple calculation, we have
\begin{align}
\label{decay 2 b}
\begin{split}
\frac{d}{dt} \|B\|_{L^{p}}^{p} & + \frac{4\, (p-1)}{p} \int_{\mathbb{R}^{3}} |\nabla |B|^{p/2}|^{2} \, dx   \\
& \lesssim \|u\|_{L^{2}\cap L^{3}} \|\nabla B\|_{L^{2}} \|B\|_{L^{p}}^{p-1} + \|B\|_{L^{2}\cap L^{3}} \|\nabla u\|_{L^{2}}\|B\|_{L^{p}}^{p-1}.
\end{split}
\end{align}
Summing up (\ref{decay 2 u}) and (\ref{decay 2 b}), we easily obtain that
\begin{align}
\label{L p estimate}
\begin{split}
\|(u, B)\|_{L_{t}^{\infty}(L^{p})} \lesssim &\, \|(u_{0}, B_{0})\|_{L^{p}} + \|\nabla \Pi\|_{L_{t}^{1}(L^{p})}
+ \int_{0}^{t}\|B\|_{L^{2}\cap L^{3}}\|\nabla u\|_{L^{2}}\, d\tau   \\
&\, + \int_{0}^{t}\left( \|B\|_{L^{2}\cap L^{3}} + \|u\|_{L^{2}\cap L^{3}} \right) \|\nabla B\|_{L^{2}} \, d\tau
\end{split}
\end{align}
On the other hand, applying the operator $\mathrm{div}$ to the first equation in (\ref{mhd_a}), we get
\begin{align*}
\Delta \Pi = \mathrm{div} \left( -u\cdot \nabla u + (1+a)B\cdot \nabla B + a\, (\Delta u - \nabla \Pi) \right),
\end{align*}
which together with the classical elliptic estimates implies
\begin{align}
\label{decay 2 pressure}
\begin{split}
\|\nabla \Pi\|_{L^{p}} & \lesssim \|u\cdot \nabla u\|_{L^{p}} + \|a\, (\Delta u - \nabla \Pi)\|_{L^{p}}
+ \|(1+a)B\cdot \nabla B\|_{L^{p}}  \\
& \lesssim \|u\|_{L^{2}\cap L^{3}}\|\nabla u\|_{L^{2}} + \|a\|_{L^{2}\cap L^{3}}\|\Delta u - \nabla \Pi\|_{L^{2}}   \\
& \quad\,\,\, + \left( 1 + \|a\|_{L^{\infty}} \right)\|B\|_{L^{2}\cap L^{3}} \|\nabla B\|_{L^{2}}.
\end{split}
\end{align}
Due to (\ref{decay 1 second}), we have
\begin{align}
\label{decay 2 partial u to t}
\begin{split}
& \|(\partial_{t}u, \partial_{t}B)\|_{L_{t}^{2}(L^{2})}^{2}   \\
\lesssim &\, \|(\nabla u_{0}, \nabla B_{0})\|_{L^{2}}^{2}   \\
&\,\,\,\,\,\, +\|(u, B)\|_{L_{t}^{\infty}(L^{2})}\|(\nabla u, \nabla B)\|_{L_{t}^{\infty}(L^{2})}\|(\nabla^{2}u, \nabla^{2}B)\|_{L_{t}^{2}(L^{2})}^{2}  \\
\leq &\, C.
\end{split}
\end{align}
which together with (\ref{decay 2 pressure}) gives rise to
\begin{align}
\begin{split}
\|\Delta u - \nabla \Pi\|_{L_{t}^{2}(L^{2})}^{2} \lesssim &\, \|\partial_{t}u\|_{L_{t}^{2}(L^{2})}^{2}
+ \|u\|_{L_{t}^{\infty}(L^{2})} \|\nabla u\|_{L_{t}^{\infty}(L^{2})} \|\Delta u\|_{L_{t}^{2}(L^{2})}^{2}    \\
&\, + \|B\|_{L_{t}^{\infty}(L^{2})}\|\nabla B\|_{L_{t}^{\infty}(L^{2})}\|\Delta B\|_{L_{t}^{2}(L^{2})}^{2}  \\
\leq &\, C.
\end{split}
\end{align}
Therefore, thanks to $H^{1}(\mathbb{R}^{3}) \hookrightarrow L^{2}\cap L^{3}(\mathbb{R}^{3})$, together with (\ref{decay 2 pressure}), we get
\begin{align*}
\|\nabla \Pi\|_{L_{t}^{1}(L^{p})} \leq C(t),
\end{align*}
so that we have
\begin{align*}
\|u\|_{L_{t}^{\infty}(L^{p})} + \|B\|_{L_{t}^{\infty}(L^{p})} \leq C(t)
\end{align*}
which, together with (\ref{decay 2 partial u to t}) and the classical Aubin-Lions lemma implies
$u \in C([0, \infty); L^{p}(\mathbb{R}^{3}))$.
\end{proof}

\begin{proposition}
\label{decay pro 3}
Under the assumptions of Theorem \ref{decay_main_theorem}, there holds
\begin{align*}
 \|u(t)\|_{L^{2}} + \|B(t)\|_{L^{2}}& \leq C \, (1+t)^{-\beta(p)},    \\
 \|\nabla u(t)\|_{L^{2}} + \|\nabla B(t)\|_{L^{2}}& \leq C\, (1+t)^{1/2-\beta(p)}, \quad \text{for } t\geq t_{0},   \\
\int_{t_{0}}^{\infty} (1+t)^{(1+2\, \beta(p))^{-}} \bigg( \|(\partial_{t}u, \partial_{t}B)\|_{L^{2}}^{2}
& + \|(\Delta u, \Delta B)\|_{L^{2}}^{2} + \|\nabla \Pi\|_{L^{2}}^{2} \bigg)\, dt \leq C
\end{align*}
where $C$ depends on $m, M, \|a_{0}\|_{L^{2}}, \|a_{0}\|_{L^{\infty}}, \|u_{0}\|_{L^{p}}$ and $\|u_{0}\|_{H^{1}}$.
\end{proposition}
\begin{proof}
Step 1 : Rough decay estimate of $\|u(t)\|_{L^{2}}$ and $\|B(t)\|_{L^{2}}$. We split the phase space $\mathbb{R}^{3}$ into two
time-dependent regions so that
\begin{align*}
& \|\nabla u(t)\|_{L^{2}}^{2} = \int_{S_{1}(t)} |\xi|^{2}|\hat{u}(\xi, t)|^{2}\, d\xi + \int_{S_{1}(t)^{c}} |\xi|^{2}|\hat{u}(\xi,t)|^{2}\, d\xi,   \\
& \|\nabla B(t)\|_{L^{2}}^{2} = \int_{S_{2}(t)} |\xi|^{2}|\hat{B}(\xi, t)|^{2}\, d\xi + \int_{S_{2}(t)^{c}} |\xi|^{2}|\hat{B}(\xi,t)|^{2}\, d\xi,
\end{align*}
where $S_{1}(t):= \left\{ \xi : |\xi| \leq \sqrt{\bar{\rho}/2}\,g(t) \right\}$, $S_{2}(t):= \left\{ \xi : |\xi| \leq \sqrt{1/2}\,g(t) \right\}$
with $\bar{\rho} := \sup_{(t,x)\in \mathbb{R}^{+}\times \mathbb{R}^{3}} \rho(x,t) = \sup_{x\in \mathbb{R}^{3}} \rho_{0}(x)$
and $g(t)$ satisfies $g(t) \lesssim \langle t \rangle^{-1/2}$ which will be chosen later on. Here and in what follows, we shall always denote
$1+t$ by $\langle t \rangle$. Then, thanks to (\ref{basic diff}), we obtain
\begin{align}
\label{decay 3 step 1 main}
\begin{split}
& \frac{d}{dt} \left( \|\sqrt{\rho}u(t)\|_{L^{2}}^{2} + \|B(t)\|_{L^{2}}^{2} \right)
+ g^{2}(t)\left( \|\sqrt{\rho}u(t)\|_{L^{2}}^{2} + \|B(t)\|_{L^{2}}^{2} \right) \\
\leq &\, \bar{\rho} g^{2}(t) \int_{S_{1}(t)} |\hat{u}(\xi, t)|^{2}\, d\xi + g^{2}(t)\int_{S_{2}(t)} |\hat{B}(\xi, t)|^{2}\, d\xi.
\end{split}
\end{align}
Next, we need to deal with the low-frequency part of $u$ and $B$, we rewrite the momentum equation and magnetic field equation as follows
\begin{align*}
u(t) = &\, e^{(t-t_{0})\Delta}u(t_{0})
 + \int_{t_{0}}^{t} e^{(t-t')\Delta}\mathbb{P}\bigg( \nabla \cdot (-u \otimes u) + a\, (\Delta u - \nabla \Pi)  \\
&\, + \nabla\cdot (B \otimes B) + a\, (B\cdot \nabla B) \bigg)\, dt',        \\
B(t) = &\, e^{(t-t_{0})\Delta}B(t_{0})
+ \int_{t_{0}}^{t}e^{(t-t')\Delta}\bigg( B\cdot \nabla u(t') - u\cdot \nabla B(t') \bigg)\, dt',
\end{align*}
where $\mathbb{P}$ denote the Leray projection operator ant $t_{0}$ is the positive time determined by Proposition \ref{decay pro 1}.
Taking Fourier transform with respect to $x$-variables leads to
\begin{align*}
|\hat{u}(\xi, t)| \lesssim &\, e^{-(t-t_{0})|\xi|^{2}}|\hat{u}_{t_{0}}(\xi)|
 + \int_{t_{0}}^{t}e^{-(t-t')|\xi|^{2}}\bigg( |\xi|\, |\widehat{u \otimes u}(\xi)| + |\xi|\, |\widehat{B \otimes B}(\xi)| \\
& \, + |\mathcal{F}(a\, (\Delta u -\nabla \Pi))(\xi)| + |\mathcal{F}(a\,B\cdot \nabla B)(\xi)| \bigg) \, dt',   \\
|\hat{B}(\xi, t)| \lesssim &\, e^{-(t-t_{0})|\xi|^{2}}|\hat{B}_{t_{0}}(\xi)|
+ \int_{t_{0}}^{t} e^{-(t-t')|\xi|^{2}} \bigg( |\xi|\, |\widehat{B \otimes u}(\xi)| + |\xi|\, |\widehat{u \otimes B}(\xi)| \bigg)\, dt'.
\end{align*}
which implies that
\begin{align}
\label{decay 3 step 1 u b}
\begin{split}
\int_{S_{1}(t)} |\hat{u}(\xi, t)|^{2}\, d\xi \lesssim &\, \int_{S_{1}(t)}e^{-2(t-t_{0})|\xi|^{2}} |\hat{u}_{t_{0}}(\xi)|^{2}\, d\xi
+ g(t)^{5} \left(\int_{t_{0}}^{t}\|\widehat{u \otimes u}\|_{L_{\xi}^{\infty}}\, dt'\right)^{2}    \\
&\, + g(t)^{5} \left( \int_{t_{0}}^{t}\|\widehat{B \otimes B}\|_{L_{\xi}^{\infty}} \, dt' \right)^{2}
+ g(t)^{3} \left( \int_{t_{0}}^{t} \|\mathcal{F}(a\, B\cdot \nabla B)\|_{L_{\xi}^{\infty}}\,dt' \right)^{2}  \\
&\, + g(t)^{3} \left(\int_{t_{0}}^{t} \|\mathcal{F}(a\, (\Delta u - \nabla \Pi))\|_{L_{\xi}^{\infty}}\,dt' \right)^{2}, \\
\int_{S_{2}(t)}|\hat{B}(\xi, t)|^{2}\, d\xi \lesssim &\, \int_{S_{2}(t)} e^{-2 (t-t_{0})|\xi|^{2}}|\hat{B}_{t_{0}}(\xi)|^{2} \, d\xi
+ g(t)^{5} \left( \int_{t_{0}}^{t}\|\widehat{B \otimes u}\|_{L_{\xi}^{\infty}} \, dt' \right)^{2}.
\end{split}
\end{align}
Thanks to Proposition \ref{decay pro 1} and (\ref{basic_energy_equality}), we have
\begin{align*}
& \left( \int_{t_{0}}^{t} \|\mathcal{F}(a\, (\Delta u - \nabla \Pi))\|_{L_{\xi}^{\infty}} \, dt' \right)^{2}\leq   \\
& \quad\quad\quad\quad\quad\quad\quad\quad
 \|a\|_{L_{t}^{\infty}(L^{2})}^{2} \left( \int_{t_{0}}^{t} \|\Delta u - \nabla \Pi\|_{L^{2}}\, dt' \right)^{2} \lesssim t-t_{0},
\end{align*}
and
\begin{align*}
\left( \int_{t_{0}}^{t} \|\mathcal{F}(a\, B\cdot \nabla B)\|_{L_{\xi}^{\infty}} \right)^{2} & \lesssim
\left( \int_{t_{0}}^{t} \|a\, B\cdot \nabla B\|_{L^{1}} \, dt' \right)^{2}  \\
& \lesssim \|a\|_{L_{t}^{\infty}(L^{\infty})}^{2} \left( \int_{t_{0}}^{t} \|B\|_{L^{2}} \|\nabla B\|_{L^{2}} \, dt' \right)^{2} \\
& \lesssim \, t-t_{0},
\end{align*}
while it is easy to see that
\begin{align*}
& \left( \int_{t_{0}}^{t} \|\widehat{u \otimes u}\|_{L_{\xi}^{\infty}} + \|\widehat{B \otimes B}\|_{L_{\xi}^{\infty}}
+ \|\widehat{B \otimes u}\|_{L_{\xi}^{\infty}} \, dt' \right)^{2} \\
\lesssim & \left( \int_{t_{0}}^{t} \|u\|_{L^{2}}^{2}\, dt' \right)^{2} + \left( \int_{t_{0}}^{t} \|B\|_{L^{2}}^{2}\, dt' \right)^{2}
+ \left( \int_{t_{0}}^{t}\|B\|_{L^{2}}\|u\|_{L^{2}}\,dt' \right)^{2} \\
\lesssim &\, (t-t_{0})^{2}.
\end{align*}
By Proposition \ref{decay pro 2}, we know that $u(t_{0}), B(t_{0}) \in L^{p}(\mathbb{R}^{3})$ for $1 < p < \frac{6}{5}$,
and for $\frac{1}{q} := \frac{4}{3}\beta(p) = \frac{2}{p} - 1$ and $\frac{1}{p} + \frac{1}{p'} = 1$,
\begin{align}
\label{u decay beta}
\begin{split}
\int_{S_{1}(t)}e^{-2(t-t_{0})|\xi|^{2}}|\hat{u}(\xi, t_{0})|^{2}\, d\xi \lesssim &\,
\left( \int_{S_{1}(t)}e^{-2q(t-t_{0})|\xi|^{2}}\,d\xi \right)^{1/q} \|\hat{u}(\xi, t_{0})\|_{L^{p'}}^{2}    \\
\lesssim &\, \|u(t_{0})\|_{L^{p}}^{2} \langle t \rangle^{-2 \beta(p)}
\end{split}
\end{align}
and
\begin{align}
\label{b decay beta}
\begin{split}
\int_{S_{2}(t)}e^{-2(t-t_{0})|\xi|^{2}}|\hat{B}(\xi, t_{0})|^{2}\, d\xi \lesssim &\,
\left( \int_{S_{2}(t)}e^{-2q(t-t_{0})|\xi|^{2}}\,d\xi \right)^{1/q} \|\hat{B}(\xi, t_{0})\|_{L^{p'}}^{2}    \\
\lesssim &\, \|B(t_{0})\|_{L^{p}}^{2} \langle t \rangle^{-2 \beta(p)}
\end{split}
\end{align}
where we used the Hausforff-Young inequality \cite{Classical}. Then since $g(t) \lesssim \langle t \rangle^{-1/2}$, we deduce from (\ref{decay 3 step 1 u b}) that
\begin{align}
\label{decay 3 step 1 low u}
\begin{split}
\int_{S_{1}(t)} |\hat{u}(\xi, t)|^{2} \, d\xi \lesssim \, \langle t \rangle^{-2 \beta(p)} + \langle t \rangle^{-\frac{1}{2}} \lesssim \, \langle t \rangle^{-\frac{1}{2}}
\quad \text{for } t \geq t_{0},
\end{split}
\end{align}
and
\begin{align}
\label{decay 3 step 1 low b}
\begin{split}
\int_{S_{2}(t)} |\hat{B}(\xi, t)|^{2} \, d\xi \lesssim \, \langle t \rangle^{-2 \beta(p)} + \langle t \rangle^{-\frac{1}{2}} \lesssim \, \langle t \rangle^{-\frac{1}{2}}
\quad \text{for } t \geq t_{0},
\end{split}
\end{align}
where we used the fact $\frac{1}{2} < \beta(p) < \frac{3}{4}$. Substituting (\ref{decay 3 step 1 low u}) and (\ref{decay 3 step 1 low b})
into (\ref{decay 3 step 1 main}) results in
\begin{align*}
\frac{d}{dt}\|(\sqrt{\rho}u, B)\|_{L^{2}}^{2} + g^{2}(t)\|(\sqrt{\rho}u, B)\|_{L^{2}}^{2} \lesssim \, g^{2}(t)\langle t \rangle^{-\frac{1}{2}}
\lesssim \, \langle t \rangle^{-\frac{3}{2}}\quad \text{for } t \geq t_{0},
\end{align*}
which gives
\begin{align*}
e^{\int_{t_{0}}^{t} g^{2}(t)\, dt'}\|(\sqrt{\rho}u, B)\|_{L^{2}}^{2} \lesssim \, \|(\sqrt{\rho}u(t_{0}), B(t_{0}))\|_{L^{2}}^{2}
+ \int_{t_{0}}^{t} e^{\int_{t_{0}}^{t'} g(\tau)^{2}\, d\tau} \langle t' \rangle^{-\frac{3}{2}}\, dt'
\end{align*}
Taking $g^{2}(t) := \frac{\alpha}{1+t}$ (with $\alpha > \frac{1}{2}$) in the above inequality, we infer
\begin{align*}
\left( \|\sqrt{\rho}u(t)\|_{L^{2}}^{2} + \|B(t)\|_{L^{2}}^{2} \right)\langle t \rangle^{\alpha} \lesssim \, 1 + \langle t \rangle^{\alpha - \frac{1}{2}},
\end{align*}
which gives
\begin{align}
\label{rough decay u b}
\|u(t)\|_{L^{2}} + \|B(t)\|_{L^{2}} \lesssim \langle t \rangle^{-\frac{1}{4}}.
\end{align}

Step 2 : Rough decay estimate of $\|\nabla u(t)\|_{L^{2}}$ and $\|\nabla B(t)\|_{L^{2}}$.
We split the phase space $\mathbb{R}^{3}$ into two time-dependent regions so that
\begin{align*}
\|\nabla^{2}u(t)\|_{L^{2}}^{2} = \int_{S(t)} |\xi|^{4} |\hat{u}(\xi, t)|^{2} \, d\xi + \int_{S(t)^{c}} |\xi|^{2} |\widehat{\nabla u}(\xi, t)|^{2} \, d\xi,  \\
\|\nabla^{2}B(t)\|_{L^{2}}^{2} = \int_{S(t)} |\xi|^{4} |\hat{B}(\xi, t)|^{2} \, d\xi + \int_{S(t)^{c}} |\xi|^{2} |\widehat{\nabla B}(\xi, t)|^{2} \, d\xi,
\end{align*}
where $S(t) := \left\{ \xi : |\xi| \leq \sqrt{\frac{1}{e_{2}}}\, g(t) \right\}$ and $g(t) \lesssim \langle t \rangle^{-\frac{1}{2}}$, which will be chosen later on. Then we deduce from Proposition \ref{decay pro 1} that
\begin{align}
\label{decay 3 step 2 main}
\begin{split}
& \frac{d}{dt}\|(\nabla u, \nabla B)\|_{L^{2}}^{2} + e_{1}\|(\partial_{t}u, \partial_{t}B)\|_{L^{2}}^{2} + g^{2}(t)\|(\nabla u, \nabla B)\|_{L^{2}}^{2}   \\
&\quad\quad\quad\quad \leq g^{4}(t) \int_{S(t)} |\hat{u}(\xi, t)|^{2} + |\hat{B}(\xi ,t)|^{2}\, d\xi \lesssim g^{4}(t)\langle t \rangle^{-\frac{1}{2}}
\lesssim \langle t \rangle^{-\frac{5}{2}},
\end{split}
\end{align}
which gives
\begin{align*}
& e^{\int_{t_{0}}^{t}g(t')^{2}\, dt'} \|(\nabla u, \nabla B)\|_{L^{2}}^{2}
+ e_{1}\int_{t_{0}}^{t} e^{\int_{t_{0}}^{t'} g(\tau)^{2}\, d\tau} \|(\partial_{t}u, \partial_{t}B)\|_{L^{2}}^{2} \, dt'    \\
& \quad\quad\quad\quad\quad\quad\quad
\lesssim \, \|(\nabla u(t_{0}), \nabla B(t_{0}))\|_{L^{2}}^{2} + \int_{t_{0}}^{t} e^{\int_{t_{0}}^{t'}g(\tau)^{2}\,d\tau}\langle t \rangle^{-\frac{5}{2}} \,dt'
\end{align*}
Taking $g^{2}(t) := \frac{\alpha}{1+t}$ (with $\alpha > 1$) in the above inequality, we deduce that
\begin{align}
\label{decay 3 step2 biger than 1}
\begin{split}
\|(\nabla u, \nabla B)\|_{L^{2}}^{2}\langle t \rangle^{\alpha} + e_{1}\int_{t_{0}}^{t} \langle t' \rangle^{\alpha} \|(\partial_{t}u, \partial_{t}B)\|_{L^{2}}^{2}\,dt' \lesssim \, 1 + \int_{t_{0}}^{t} \langle t' \rangle^{\alpha - \frac{5}{2}} \, dt'.
\end{split}
\end{align}
In particular, for $\alpha > \frac{3}{2}$, we get
\begin{align*}
\|(\nabla u, \nabla B)\|_{L^{2}}^{2} \langle t \rangle^{\alpha} \lesssim \, 1 + \langle t \rangle^{\alpha - \frac{3}{2}} \quad \text{for } t \geq t_{0},
\end{align*}
which implies
\begin{align}
\label{rough decay geadient u b}
\|\nabla u(t)\|_{L^{2}} + \|\nabla B(t)\|_{L^{2}} \lesssim \, \langle t \rangle^{-\frac{3}{4}} \quad \text{for } t\geq t_{0}.
\end{align}
Moreover, it follows from (\ref{decay 3 step2 biger than 1}) that
\begin{align*}
\int_{t_{0}}^{t}\langle t' \rangle^{\left(\frac{3}{2}\right)^{-}} \|(\partial_{t}u, \partial_{t}B)\|_{L^{2}}^{2} \, dt' \lesssim \, 1,
\end{align*}
which ensures that
\begin{align}
\begin{split}
&\, \left( \int_{t_{0}}^{t} \|(\partial_{t}u, \partial_{t}B)\|_{L^{2}} \right)^{2}  \\
\lesssim &\, \int_{t_{0}}^{t}\langle t \rangle^{\left( \frac{3}{2} \right)^{-}}\|(\partial_{t}u, \partial_{t}B)\|_{L^{2}}^{2}\, dt'
\int_{t_{0}}^{t}\langle t \rangle^{-\left( \frac{3}{2} \right)^{-}}\, dt'
\lesssim \, 1.
\end{split}
\end{align}

Step 3 : Improved decay estimates of $\|u(t)\|_{L^{2}}$, $\|\nabla u(t)\|_{L^{2}}$ and $\|B(t)\|_{L^{2}}, \|\nabla B(t)\|_{L^{2}}$.
We shall utilize an iteration argument. Thanks to (\ref{rough decay u b}), we obtain
\begin{align}
\label{decay 3 step 3 u u}
\begin{split}
\left( \int_{t_{0}}^{t} \|\mathcal{F}(u \otimes u)(t')\|_{L_{\xi}^{\infty}} \, dt' \right)^{2}
\lesssim & \left( \int_{t_{0}}^{t} \|u(t')\|_{L^{2}}^{2}\, dt' \right)^{2}    \\
\lesssim & \left( \int_{t_{0}}^{t} \langle t \rangle^{-\frac{1}{2}} \, dt' \right)^{2} \lesssim \, \langle t \rangle,
\end{split}
\end{align}
and similar we have
\begin{align}
\label{decay 3 step 3 u b}
\begin{split}
\left( \int_{t_{0}}^{t} \|\mathcal{F}(u \otimes B)(t')\|_{L_{\xi}^{\infty}} \, dt' \right)^{2}
\lesssim \, \langle t \rangle.
\end{split}
\end{align}
Same argument as in \cite{zhangping}, we have
\begin{align}
\label{decay 3 step 3 a u pi}
\left( \int_{t_{0}}^{t} \|\mathcal{F}(a\, (\Delta u - \nabla \Pi))(t')\|_{L_{\xi}^{\infty}} \, dt' \right)^{2} \leq \, C.
\end{align}
Using (\ref{rough decay u b}) and (\ref{rough decay geadient u b}), we obtain
\begin{align}
\label{decay 3 step 3 a b b}
\begin{split}
&\, \left(\int_{t_{0}}^{t} \|\mathcal{F}(a\, B\cdot \nabla B)(t')\|_{L_{\xi}^{\infty}} \, dt' \right)^{2} \\
\lesssim &\, \left( \int_{t_{0}}^{t} \|a\, B \cdot \nabla B\|_{L^{1}} \,dt' \right)^{2}   \\
\lesssim &\, \left( \int_{t_{0}}^{t} \|B\|_{L^{2}} \|\nabla B\|_{L^{2}} \, dt' \right)^{2}  \\
\lesssim &\, \left( \int_{t_{0}}^{t} \langle t \rangle^{-\frac{1}{4}} \langle t \rangle^{-\frac{3}{4}} \right)^{2}
\lesssim \, \left( \mathrm{ln}\langle t \rangle \right)^{2}
\end{split}
\end{align}
Substituting (\ref{u decay beta}), (\ref{b decay beta}), (\ref{decay 3 step 3 u u}), (\ref{decay 3 step 3 u b}), (\ref{decay 3 step 3 a u pi})
and (\ref{decay 3 step 3 a b b}) into (\ref{decay 3 step 1 u b}) results in
\begin{align}
\label{decay 3 step 3 low u}
\int_{S(t)} |\hat{u}(\xi, t)|^{2} \, d\xi \lesssim \langle t \rangle^{-2 \beta(p)} + \langle t \rangle^{-\frac{3}{2} + \epsilon}
\lesssim \langle t \rangle^{-2 \beta(p)},
\end{align}
and
\begin{align}
\label{decay 3 step 3 low b}
\int_{S(t)} |\hat{B}(\xi, t)|^{2}\, d\xi \lesssim \langle t \rangle^{-2 \beta(p)} + \langle t \rangle^{-\frac{3}{2}} \lesssim \langle t \rangle^{-2 \beta(p)},
\end{align}
where $\epsilon > 0$ is a small enough constant.
From this and (\ref{decay 3 step 1 main}), we have
\begin{align*}
\frac{d}{dt} \|(\sqrt{\rho}u, B)\|_{L^{2}}^{2} + g^{2}(t)\|(\sqrt{\rho}u, B)\|_{L^{2}}^{2}
\lesssim \, g^{2}(t) \langle t \rangle^{-2 \beta(p)} \lesssim \, \langle t \rangle^{-1-2\beta(p)},
\end{align*}
which implies
\begin{align*}
e^{\int_{t_{0}}^{t}g(t')^{2}\, dt'}\|(\sqrt{\rho} u, B)\|_{L^{2}}^{2}
\lesssim \, \|(\sqrt{\rho}u(t_{0}), B(t_{0}))\|_{L^{2}}^{2} + \int_{t_{0}}^{t} e^{\int_{t_{0}}^{t'}g(\tau)^{2}\, d\tau} \langle t \rangle^{-1-2\beta(p)}\, dt'
\end{align*}
Taking $g^{2}(t) = \frac{\alpha}{1+t}$ (with $\alpha > 2 \beta(p)$) in the above inequality leads to
\begin{align*}
\left(\|\sqrt{\rho}u(t)\|_{L^{2}}^{2} + \|B(t)\|_{L^{2}}^{2} \right) \langle t \rangle^{\alpha} \lesssim 1 + \langle t \rangle^{\alpha - 2\beta(p)},
\end{align*}
which in particular gives
\begin{align*}
\|u(t)\|_{L^{2}} + \|B(t)\|_{L^{2}} \lesssim \langle t \rangle^{-\beta(p)}.
\end{align*}
From (\ref{decay 3 step 3 low u}), (\ref{decay 3 step 3 low b}) and (\ref{decay 3 step 2 main}), we infer
\begin{align*}
\frac{d}{dt}\|(\nabla u, \nabla B)\|_{L^{2}}^{2} & + e_{1} \|(\partial_{t}u, \partial_{t}B)\|_{L^{2}}^{2} \\
& + g^{2}(t)\|(\nabla u, \nabla B)\|_{L^{2}}^{2} \lesssim \, g^{4}(t) \langle t \rangle^{-2 \beta(p)}
\lesssim \, \langle t \rangle^{-2-2\beta(p)},
\end{align*}
which implies
\begin{align*}
e^{\int_{t_{0}}^{t}g(t')^{2}\, dt'}\|(\nabla u, \nabla B)\|_{L^{2}}^{2} & + e_{1} \int_{t_{0}}^{t} e^{\int_{t_{0}}^{t'} g(\tau)^{2}\, d\tau}
\|(\partial_{t}u, \partial_{t}B)\|_{L^{2}}^{2} \, dt'   \\
& \lesssim \, \|(\nabla u(t_{0}), \nabla B(t_{0}))\|_{L^{2}}^{2} + \int_{t_{0}}^{t} e^{\int_{t_{0}}^{t'}g(\tau)\, d\tau}
\langle t \rangle^{-2-2\beta(p)}\, dt'
\end{align*}
Taking $g^{2}(t) = \frac{\alpha}{1+t}$ (with $\alpha > 1$) in the above inequality, we obtain
\begin{align}
\label{decay 3 step 3 geadient u b}
\begin{split}
\|(\nabla u, \nabla B)\|_{L^{2}}^{2} \langle t \rangle^{\alpha} + & e_{1} \int_{t_{0}}^{t} \langle t' \rangle^{\alpha}\|(\partial_{t}u, \partial_{t}B)\|_{L^{2}}^{2} \, dt' \\
& \quad\quad\quad\quad \lesssim \, 1 + \int_{t_{0}}^{t} \langle t' \rangle^{\alpha - 2 - 2\beta(p)}\, dt'.
\end{split}
\end{align}
In particular, taking $\alpha > 1+2\beta(p)$ in (\ref{decay 3 step 3 geadient u b}), we get
\begin{align}
\label{decay geadient u b}
\|\nabla u(t)\|_{L^{2}} + \|\nabla B(t)\|_{L^{2}} \lesssim \langle t \rangle^{-\frac{1}{2}-\beta(p)}.
\end{align}
Taking $\alpha \in (\frac{3}{2}, 1+2\beta(p))$ in (\ref{decay 3 step 3 geadient u b}) results in
\begin{align*}
\int_{t_{0}}^{t} \langle t' \rangle^{(1+2\beta(p))^{-}} \left( \|\partial_{t}u(t')\|_{L^{2}}^{2}
+ \|\partial_{t}B(t')\|_{L^{2}}^{2} \right) \, dt' \leq C.
\end{align*}
Combining the following fact
\begin{align*}
\|u\cdot \nabla u\|_{L^{2}}^{2} \leq \,C\,\|u\|_{L^{\infty}}^{2}\|\nabla u\|_{L^{2}}^{2} \leq \, C\, \|\Delta u\|_{L^{2}}\|\nabla u\|_{L^{2}}^{3},  \\
\|B\cdot \nabla B\|_{L^{2}}^{2} \leq \,C\,\|B\|_{L^{\infty}}^{2}\|\nabla B\|_{L^{2}}^{2} \leq \, C\, \|\Delta B\|_{L^{2}}\|\nabla B\|_{L^{2}}^{3},
\end{align*}
with (\ref{decay 1 pressure}), we have
\begin{align}
\label{decay pressure}
\begin{split}
& \int_{t_{0}}^{\infty} \langle t' \rangle^{(1+2\beta(p))^{-}}\left( \|\Delta u(t')\|_{L^{2}}^{2}+\|\Delta B(t')\|_{L^{2}}^{2} + \|\nabla \Pi(t')\|_{L^{2}}^{2} \right)\, dt'    \\
\lesssim &\, \int_{t_{0}}^{\infty} \langle t \rangle^{(1+2\beta(p))^{-}} \|(\partial_{t}u,\partial_{t}B)\|_{L^{2}}^{2}\, dt'
+ \int_{t_{0}}^{\infty} \langle t \rangle^{-\frac{1}{2}(1+2\beta(p))^{-}} \|(\Delta u, \Delta B)\|_{L^{2}}\, dt' \leq \, C,
\end{split}
\end{align}
where we used (\ref{decay geadient u b}) and (\ref{4.2}).
\end{proof}

\begin{remark}
\label{decay remark}
Thanks to (\ref{decay pressure}), it is easy to observe that
\begin{align}
\begin{split}
\int_{t_{0}}^{\infty}\langle t' \rangle^{\alpha} \left( \|\Delta u(t')\|_{L^{2}}^{2} + \|\Delta B(t')\|_{L^{2}}^{2}
+ \|\nabla \Pi(t')\|_{L^{2}}^{2} \right)\, dt' \lesssim \, \langle t \rangle^{\alpha - (1+2\beta(p))},
\end{split}
\end{align}
for $\alpha \in (1+2\beta(p), 2+6\beta(p))$. Moreover, without loss of generality, we may assume that
\begin{align}
\|\Delta u(t_{0})\|_{L^{2}} + \|\Delta B(t_{0})\|_{L^{2}} \leq \, 1.
\end{align}
\end{remark}

\begin{proposition}
\label{decay pro 4}
Under the assumptions of Theorem \ref{decay_main_theorem}, there holds
\begin{align}
\int_{t_{0}}^{\infty} \left( \|u(t)\|_{L^{\infty}} + \|B(t)\|_{L^{\infty}} + \|\nabla u(t)\|_{L^{\infty}} + \|\nabla B(t)\|_{L^{\infty}} \right)\, dt \leq C
\end{align}
and
\begin{align}
\label{laplace u b final}
\int_{t_{0}}^{\infty}\left( \|\Delta u(t)\|_{L^{2}}^{\theta} + \|\Delta B(t)\|_{L^{2}}^{\theta}
+ \|\nabla \Pi(t)\|_{L^{2}}^{\theta} \right)\, dt \leq \, C,
\end{align}
for $\frac{2}{3} \leq \theta \leq 2$.
\end{proposition}
\begin{proof}
Step 1 : Estimate of $\|u_{t}(t)\|_{L^{2}}$, $\|u(t)\|_{\dot{H}^{2}}$ and $\|B_{t}(t)\|_{L^{2}}$, $\|B(t)\|_{\dot{H}^{2}}$.
We get by first applying $\partial_{t}$ to the momentum equation of (\ref{mhd_rho}) and then taking $L^{2}$ inner product of the resulting
equation with $\partial_{t}u$ that
\begin{align}
\label{decay 4 step 1 u}
\begin{split}
& \frac{1}{2}\frac{d}{dt}\|\sqrt{\rho}u_{t}(t)\|_{L^{2}}^{2} + \|\nabla u_{t}(t)\|_{L^{2}}^{2} =  - \int_{\mathbb{R}^{3}} \rho_{t}u_{t} \cdot
(u\cdot \nabla u) \\
& \quad\quad\quad\quad\quad
-\int_{\mathbb{R}^{3}} \rho u_{t}\cdot (u_{t}\cdot \nabla u) + \rho_{t} |u_{t}|^{2} - B_{t}\cdot \nabla B \cdot u_{t}
- B \cdot \nabla B_{t} \cdot u_{t} \, dx.
\end{split}
\end{align}
Applying $\partial_{t}$ to the magnetic field equation of (\ref{mhd_rho}) and then taking $L^{2}$ inner product of the resulting equation
with $\partial_{t}B$ that
\begin{align}
\label{decay 4 step 1 b}
\begin{split}
\frac{1}{2}\frac{d}{dt}\|B_{t}(t)\|_{L^{2}}^{2} & + \|\nabla B_{t}(t)\|_{L^{2}}^{2} =     \\
& -\int_{\mathbb{R}^{3}} u_{t}\cdot \nabla B \cdot B_{t} - B_{t}\cdot \nabla u \cdot B_{t} - B\cdot \nabla u_{t}\cdot B_{t} \, dx.
\end{split}
\end{align}
Using similar methods in \cite{zhangping}, we obtain
\begin{align*}
\left| \int_{\mathbb{R}^{3}} \rho_{t}u_{t}\cdot (u\cdot\nabla u)\, dx \right| \leq \frac{1}{8} \|\nabla u_{t}\|_{L^{2}}^{2} +
C\, \|\nabla u\|_{L^{2}}^{4} \|\nabla^{2}u\|_{L^{2}}^{2},
\end{align*}
\begin{align*}
\left| \int_{\mathbb{R}^{3}}\rho u_{t}\cdot (u_{t}\cdot \nabla u)\, dx \right| \leq \frac{1}{8} \|\nabla u_{t}\|_{L^{2}}^{2} +
C\, \|\sqrt{\rho} u_{t}\|_{L^{2}}^{2} \|\nabla u\|_{L^{2}}^{4},
\end{align*}
\begin{align*}
\left| \int_{\mathbb{R}^{3}} \rho_{t} |u_{t}|^{2} \, dx \right| \leq \frac{1}{8} \|\nabla u_{t}\|_{L^{2}}^{2}
+ C\, \|\nabla u\|_{L^{2}}^{4} \|\sqrt{\rho} u_{t}\|_{L^{2}}^{2},
\end{align*}
\begin{align*}
\left| \int_{\mathbb{R}^{2}} B_{t}\cdot \nabla B\cdot u_{t} \, dx \right| \leq \frac{1}{8} \|\nabla u_{t}\|_{L^{2}}^{2}
+ \frac{1}{8} \|\nabla B_{t}\|_{L^{2}}^{2} + C \, \|B_{t}\|_{L^{2}}^{2} \|\nabla B\|_{L^{2}}^{4},
\end{align*}
and
\begin{align*}
\left| \int_{\mathbb{R}^{3}} B_{t} \cdot (B_{t}\cdot \nabla u)\, dx \right| \leq \frac{1}{8} \|\nabla B_{t}\|_{L^{2}}^{2}
+ C\, \|B_{t}\|_{L^{2}}^{2} \|\nabla u\|_{L^{2}}^{4}.
\end{align*}
Substituting the above five inequalities into (\ref{decay 4 step 1 u}), (\ref{decay 4 step 1 b}), we easily obtain
\begin{align*}
\frac{d}{dt}\|(\sqrt{\rho} u_{t}, B_{t})\|_{L^{2}}^{2} + \|(\nabla u_{t}, \nabla B_{t})\|_{L^{2}}^{2}
\lesssim &\, \|\nabla u\|_{L^{2}}^{4}\|\nabla^{2}u\|_{L^{2}}^{2} + \|\sqrt{\rho}u_{t}\|_{L^{2}}^{2}\|\nabla u\|_{L^{2}}^{4} \\
&\, + \|B_{t}\|_{L^{2}}^{2}\|\nabla B\|_{L^{2}}^{4} + \|B_{t}\|_{L^{2}}^{2} \|\nabla u\|_{L^{2}}^{4}.
\end{align*}
Integrating the above inequality from $t_{0}$ to $t$, we get
\begin{align*}
&\, \|\sqrt{\rho} u_{t}(t)\|_{L^{2}}^{2} + \|B_{t}(t)\|_{L^{2}}^{2} + \int_{t_{0}}^{t}\|(\nabla u_{t}, \nabla B_{t})\|_{L^{2}}^{2}\, dt'    \\
\leq &\, \|\sqrt{\rho}u_{t}(t_{0})\|_{L^{2}}+\|B_{t}(t_{0})\|_{L^{2}}^{2} + C\, \int_{t_{0}}^{t}\|\nabla^{2}u\|_{L^{2}}^{2}\|\nabla u\|_{L^{2}}^{4}\, dt'   \\
&\, + C\, \int_{t_{0}}^{t}\|\sqrt{\rho}u_{t}\|_{L^{2}}^{2}\|\nabla u\|_{L^{2}}^{4}\, dt'
+ C\, \int_{t_{0}}^{t} \|B_{t}\|_{L^{2}}^{2} \|\nabla B\|_{L^{2}}^{4} \, dt' \\
&\, + C\, \int_{t_{0}}^{t} \|B_{t}\|_{L^{2}}^{2}\|\nabla u\|_{L^{2}}^{4}\, dt'.
\end{align*}
Applying Gronwall's inequality gives
\begin{align*}
&\, \|\sqrt{\rho}u_{t}(t)\|_{L^{2}}^{2} + \|B_{t}(t)\|_{L^{2}}^{2} + \|(\nabla u_{t}, \nabla B_{t})\|_{L^{2}([t_{0},t];L^{2})}^{2}  \\
\leq &\, e^{ C\int_{t_{0}}^{t} \|(\nabla u, \nabla B)\|_{L^{2}}^{4}\, dt'}
\left( \|(\sqrt{\rho}u_{t}(t_{0}), B_{t}(t_{0}))\|_{L^{2}}^{2} + C \int_{t_{0}}^{t} \|\nabla^{2} u\|_{L^{2}}^{2}
\|\nabla u\|_{L^{2}}^{4} \, dt' \right)
\end{align*}
It then follows from (\ref{basic_energy_equality}) and Proposition \ref{decay pro 1} that
\begin{align*}
\int_{t_{0}}^{t}\|\nabla u\|_{L^{2}}^{4}\, dt' \leq \sup_{t'\in [t_{0}, t]} \|\nabla u(t')\|_{L^{2}}^{2} \|\nabla u\|_{L_{t}^{2}(L^{2})}^{2} \leq C, \\
\int_{t_{0}}^{t}\|\nabla B\|_{L^{2}}^{4}\, dt' \leq \sup_{t'\in [t_{0}, t]} \|\nabla B(t')\|_{L^{2}}^{2} \|\nabla B\|_{L_{t}^{2}(L^{2})}^{2} \leq C,
\end{align*}
and
\begin{align*}
\int_{t_{0}}^{t} \|\nabla^{2}u(t')\|_{L^{2}}^{2}\|\nabla u(t')\|_{L^{2}}^{4}\, dt' \leq \,C\, \sup_{t'\in [t_{0}, t]}
\|\nabla u(t')\|_{L^{2}}^{4} \|\Delta u\|_{L^{2}([t_{0}, t]; L^{2})}^{2} \leq C.
\end{align*}

Applying Remark \ref{decay remark} and the standard energy estimate to the momentum and magnetic field equation of (\ref{mhd_rho})
at $t = t_{0}$ yields
\begin{align*}
\|\sqrt{\rho} u_{t}(t_{0})\|_{L^{2}}^{2} & \lesssim \, \|(u\cdot\nabla u)(t_{0})\|_{L^{2}}^{2} + \|(B\cdot\nabla B)(t_{0})\|_{L^{2}}^{2}
+ \|\Delta u (t_{0})\|_{L^{2}}^{2}  \\
& \lesssim \, \|u(t_{0})\|_{L^{4}}^{2}\|\nabla u(t_{0})\|_{L^{4}}^{2} + \|B(t_{0})\|_{L^{4}}^{2}\|\nabla B(t_{0})\|_{L^{4}}^{2}
 + \|\Delta u(t_{0})\|_{L^{2}}^{2} \\
& \lesssim \, \|u(t_{0})\|_{L^{2}}^{1/2} \|\nabla u(t_{0})\|_{L^{2}}^{2} \|\Delta u(t_{0})\|_{L^{2}}^{3/2}    \\
& \quad\,\,  + \|B(t_{0})\|_{L^{2}}^{1/2} \|\nabla B(t_{0})\|_{L^{2}}^{2} \|\Delta B(t_{0})\|_{L^{2}}^{3/2} + \|\Delta u(t_{0})\|_{L^{2}}^{2}  \\
& \leq \, C,
\end{align*}
and
\begin{align*}
\|B_{t}(t_{0})\|_{L^{2}}^{2} \lesssim &\, \|\Delta B(t_{0})\|_{L^{2}}^{2} + \|(u\cdot \nabla)B(t_{0})\|_{L^{2}}^{2} + \|(B\cdot \nabla)u(t_{0})\|_{L^{2}}^{2}   \\
\lesssim &\, 1 + \|u(t_{0})\|_{L^{4}}^{2} \|\nabla B(t_{0})\|_{L^{4}}^{2} + \|B(t_{0})\|_{L^{4}}^{2} \|\nabla u(t_{0})\|_{L^{4}}^{2}  \leq \, C.
\end{align*}
Hence, we finally have
\begin{align}
\label{decay 4 step main}
\begin{split}
\sup_{t\geq t_{0}} \left( \|\sqrt{\rho}u_{t}(t)\|_{L^{2}}^{2} + \|B_{t}(t)\|_{L^{2}}^{2} \right)
+ \int_{t_{0}}^{\infty} \left( \|\nabla u_{t}\|_{L^{2}}^{2} + \|\nabla B_{t}\|_{L^{2}}^{2} \right)\, dt' \leq \, C.
\end{split}
\end{align}
Notice from the momentum equation of (\ref{mhd_rho}) that
\begin{align*}
\|\nabla^{2}u(t)\|_{L^{2}} + \|\nabla \Pi(t)\|_{L^{2}} \lesssim &\, \|\sqrt{\rho} u_{t}(t)\|_{L^{2}} + \|(u\cdot\nabla)u(t)\|_{L^{2}}
+ \|(B\cdot \nabla)B(t)\|_{L^{2}}   \\
\leq &\, C \bigg(\|\sqrt{\rho}u_{t}(t)\|_{L^{2}} + \|u(t)\|_{L^{2}} \|\nabla u(t)\|_{L^{2}}^{4}  \\
&\, + \|B(t)\|_{L^{2}}\|\nabla B(t)\|_{L^{2}}^{4}\bigg) + \frac{1}{6} \|(\nabla^{2}u(t), \nabla^{2}B(t))\|_{L^{2}},
\end{align*}
and
\begin{align*}
\|\nabla^{2} B(t)\|_{L^{2}} \leq &\, C \,\bigg( \|B_{t}(t)\|_{L^{2}} + \|(u\cdot\nabla)B(t)\|_{L^{2}} + \|(B\cdot\nabla)u(t)\|_{L^{2}} \bigg)  \\
\leq &\, C\, \bigg( \|B_{t}(t)\|_{L^{2}} + \|u(t)\|_{L^{2}}\|\nabla u(t)\|_{L^{2}}^{4} + \|B(t)\|_{L^{2}}\|\nabla B(t)\|_{L^{2}}^{4} \bigg) \\
&\, + \frac{1}{6} \|\nabla^{2}B(t)\|_{L^{2}} + \frac{1}{6} \|\nabla^{2}B(t)\|_{L^{2}},
\end{align*}
which along with Proposition \ref{decay pro 1} and (\ref{decay 4 step main}) implies that
\begin{align}
\label{decay 4 step 1 last}
\sup_{t\geq t_{0}} \left( \|\nabla^{2}u(t)\|_{L^{2}} + \|\nabla^{2}B(t)\|_{L^{2}} + \|\nabla \Pi(t)\|_{L^{2}} \right) \leq \, C.
\end{align}

Step 2 : Estimate of $\int_{t_{0}}^{\infty}\|(u(t'), B(t'))\|_{L^{\infty}}\, dt'$,
and $\int_{t_{0}}^{\infty}\|(\nabla u(t'),\nabla B(t'))\|_{L^{\infty}} dt'$.
Let $(v, q)$ solve
\begin{align*}
- \Delta v + \nabla q = f, \quad \mathrm{div} v = 0;
\end{align*}
then we have $\nabla q = -\nabla (-\Delta)^{-1} \mathrm{div} f$, which implies that for any $r \in (1, \infty)$,
\begin{align*}
\|\nabla q\|_{L^{r}} \leq \, C\, \|f\|_{L^{r}} \quad \text{and} \quad \|\Delta v\|_{L^{r}} \leq \, C\, \|f\|_{L^{r}}.
\end{align*}
From this and the momentum equation of (\ref{mhd_rho}), we infer
\begin{align*}
&\, \|\nabla^{2} u(t)\|_{L^{6}} + \|\nabla \Pi(t)\|_{L^{6}} \\
\lesssim &\, \|u_{t}(t)\|_{L^{6}} + \|(u\cdot\nabla)u(t)\|_{L^{6}} + \|(B\cdot\nabla)B(t)\|_{L^{6}} \\
\lesssim &\, \|\nabla u_{t}(t)\|_{L^{2}} + \|\nabla u(t)\|_{L^{2}}^{1/2} \|\nabla^{2}u(t)\|_{L^{2}}^{3/2}
+ \|\nabla B(t)\|_{L^{2}}^{1/2} \|\nabla^{2}B(t)\|_{L^{2}}^{3/2}    \\
&\, + \|u(t)\|_{L^{2}}^{1/4}\|\nabla u(t)\|_{L^{2}}^{3/4}\|\nabla^{2}u(t)\|_{L^{2}}^{1/4}\|\nabla^{2}u(t)\|_{L^{6}}^{3/4}   \\
&\, + \|B(t)\|_{L^{2}}^{1/4}\|\nabla B(t)\|_{L^{2}}^{3/4}\|\nabla^{2}B(t)\|_{L^{2}}^{1/4}\|\nabla^{2}B(t)\|_{L^{6}}^{3/4}
\end{align*}
Taking the $L^{2}$ norm for the time variables on $[t_{0}, t]$, we get by using (\ref{decay 4 step main}) and (\ref{decay 4 step 1 last}) that
\begin{align}
\label{decay 4 step 2 u pi}
\begin{split}
&\, \|\nabla^{2}u\|_{L^{2}([t_{0},t];L^{6})}^{2} + \|\nabla \Pi\|_{L^{2}([t_{0},t];L^{6})}^{2} \\
\leq &\, C\, \bigg\{ \|\nabla u_{t}\|_{L^{2}([t_{0},t];L^{2})}^{2}
+ \sup_{t'\in [t_{0},t]}\|\nabla u(t')\|_{L^{2}} \sup_{t'\in [t_{0},t]}\|\nabla^{2}u(t')\|_{L^{2}}\|\nabla^{2}u\|_{L^{2}([t_{0},t];L^{2})}^{2}  \\
&\quad\, + \sup_{t'\in [t_{0},t]}\|u(t')\|_{L^{2}}^{2} \sup_{t'\in [t_{0},t]}\|\nabla u(t')\|_{L^{2}}^{6} \|\nabla^{2} u\|_{L^{2}([t_{0},t];L^{2})}^{2} \\
&\quad\, + \sup_{t'\in [t_{0},t]}\|\nabla B(t')\|_{L^{2}} \sup_{t'\in [t_{0},t]}\|\nabla^{2}B(t')\|_{L^{2}}\|\nabla^{2}B\|_{L^{2}([t_{0},t];L^{2})}^{2}  \\
&\quad\, + \sup_{t'\in [t_{0},t]}\|B(t')\|_{L^{2}}^{2} \sup_{t'\in [t_{0},t]}\|\nabla B(t')\|_{L^{2}}^{6} \|\nabla^{2} B\|_{L^{2}([t_{0},t];L^{2})}^{2}
\bigg\} + \frac{1}{4} \|\nabla^{2} B\|_{L^{2}([t_{0},t];L^{6})}^{2}.
\end{split}
\end{align}
Using similar methods, we can easily obtain
\begin{align}
\label{decay 4 step 2 b}
\begin{split}
&\, \|\nabla^{2}B\|_{L^{2}([t_{0},t];L^{6})}^{2}    \\
\leq &\, C\, \bigg\{ \|\nabla B_{t}\|_{L^{2}([t_{0},t];L^{2})}^{2} + \sup_{t'\in [t_{0},t]} \|\nabla u(t')\|_{L^{2}}
\sup_{t'\in [t_{0},t]} \|\nabla^{2}u(t')\|_{L^{2}} \|\nabla^{2}B\|_{L^{2}([t_{0},t];L^{2})}^{2} \\
&\quad\, + \sup_{t'\in [t_{0},t]}\|u(t')\|_{L^{2}}^{2} \sup_{t'\in [t_{0},t]}\|\nabla u(t')\|_{L^{2}}^{6} \|\nabla^{2}B\|_{L^{2}([t_{0},t];L^{2})}^{2}  \\
&\quad\, + \sup_{t'\in [t_{0},t]} \|\nabla B(t')\|_{L^{2}} \sup_{t'\in [t_{0},t]} \|\nabla^{2}B(t')\|_{L^{2}} \|\nabla^{2}u\|_{L^{2}([t_{0},t];L^{2})}^{2} \\
&\quad\, + \sup_{t'\in [t_{0},t]}\|B(t')\|_{L^{2}}^{2} \sup_{t'\in [t_{0},t]}\|\nabla B(t')\|_{L^{2}}^{6} \|\nabla^{2}u\|_{L^{2}([t_{0},t];L^{2})}^{2} \bigg\} \\
&\quad\, + \frac{1}{4} \|\nabla^{2}B\|_{L^{2}([t_{0},t];L^{6})}^{2} + \frac{1}{4} \|\nabla^{2}u\|_{L^{2}([t_{0},t];L^{6})}^{2}.
\end{split}
\end{align}
Summing up (\ref{decay 4 step 2 u pi}) and (\ref{decay 4 step 2 b}), we get
\begin{align}
\label{gradient u gradient pi}
\begin{split}
&\, \|\nabla^{2}u\|_{L^{2}([t_{0},t];L^{6})}^{2} + \|\nabla^{2}B\|_{L^{2}([t_{0},t];L^{6})}^{2} + \|\nabla \Pi\|_{L^{2}([t_{0},t];L^{6})}^{2}   \\
\lesssim &\, \|\nabla u_{t}\|_{L^{2}([t_{0},t];L^{2})}^{2} + \|\nabla B_{t}\|_{L^{2}([t_{0},t];L^{2})}^{2}  \\
&\, + \sup_{t'\in [t_{0},t]}\|\nabla u(t')\|_{L^{2}}\sup_{t'\in [t_{0},t]}\|\nabla^{2}u(t')\|_{L^{2}}
\|(\nabla^{2} u, \nabla^{2} B)\|_{L^{2}([t_{0},t];L^{2})}^{2} \\
&\, + \sup_{t'\in [t_{0},t]}\|u(t')\|_{L^{2}}^{2}\sup_{t'\in [t_{0},t]}\|\nabla u(t')\|_{L^{2}}^{6}
\|(\nabla^{2} u, \nabla^{2} B)\|_{L^{2}([t_{0},t];L^{2})}^{2}   \\
&\, + \sup_{t'\in [t_{0},t]}\|\nabla B(t')\|_{L^{2}}\sup_{t'\in [t_{0},t]}\|\nabla^{2}B(t')\|_{L^{2}}
\|(\nabla^{2} u, \nabla^{2} B)\|_{L^{2}([t_{0},t];L^{2})}^{2} \\
&\, + \sup_{t'\in [t_{0},t]}\|B(t')\|_{L^{2}}^{2}\sup_{t'\in [t_{0},t]}\|\nabla B(t')\|_{L^{2}}^{6}
\|(\nabla^{2} u, \nabla^{2} B)\|_{L^{2}([t_{0},t];L^{2})}^{2}   \\
\leq &\, C.
\end{split}
\end{align}
Now, let us turn to the estimate of $\int_{t_{0}}^{t}\|\nabla u(t')\|_{L^{\infty}}\, dt'$ and $\int_{t_{0}}^{t}\|\nabla B(t')\|_{L^{\infty}}\, dt'$.
Indeed, simple interpolation in three dimensions gives
\begin{align*}
\int_{t_{0}}^{t}\|\nabla u(t')\|_{L^{\infty}}\, dt' \lesssim &\, \int_{t_{0}}^{t} \|\nabla u(t')\|_{L^{6}}^{1/2}\|\nabla^{2}u(t')\|_{L^{6}}^{1/2}\, dt' \\
\lesssim &\, \int_{t_{0}}^{t}\|\Delta u(t')\|_{L^{2}}^{1/2}\|\nabla^{2}u(t')\|_{L^{6}}^{1/2} \, dt' \\
\lesssim &\, \int_{t_{0}}^{t} \|\nabla^{2} u(t')\|_{L^{6}}^{2}\, dt' + \int_{t_{0}}^{t}\|\Delta u(t')\|_{L^{2}}^{2/3}\, dt',
\end{align*}
and
\begin{align*}
\int_{t_{0}}^{t}\|\nabla B(t')\|_{L^{\infty}}\, dt' \lesssim &\, \int_{t_{0}}^{t} \|\nabla B(t')\|_{L^{6}}^{1/2}\|\nabla^{2}B(t')\|_{L^{6}}^{1/2}\, dt' \\
\lesssim &\, \int_{t_{0}}^{t}\|\Delta B(t')\|_{L^{2}}^{1/2}\|\nabla^{2}B(t')\|_{L^{6}}^{1/2} \, dt' \\
\lesssim &\, \int_{t_{0}}^{t} \|\nabla^{2} B(t')\|_{L^{6}}^{2}\, dt' + \int_{t_{0}}^{t}\|\Delta B(t')\|_{L^{2}}^{2/3}\, dt'.
\end{align*}
Thanks to (\ref{decay pressure}), we get
\begin{align}
\label{laplace u b}
\begin{split}
&\, \int_{t_{0}}^{t} \left( \|\Delta u(t')\|_{L^{2}}^{2/3} + \|\Delta B(t')\|_{L^{2}}^{2/3} + \|\nabla \Pi(t')\|_{L^{2}}^{2/3} \right) \, dt'   \\
\lesssim &\, \left( \int_{t_{0}}^{t} \langle t' \rangle^{(1+2\beta(p))^{-}} \left( \|\Delta u(t')\|_{L^{2}}^{2}
+ \|\Delta B(t')\|_{L^{2}}^{2} + \|\nabla \Pi(t')\|_{L^{2}}^{2} \right)\, dt' \right)^{\frac{1}{3}} \\
\leq &\, C,
\end{split}
\end{align}
which along with (\ref{gradient u gradient pi}) ensures that
\begin{align}
\int_{t_{0}}^{t}\|\nabla u(t)\|_{L^{\infty}} + \|\nabla B(t)\|_{L^{\infty}} \,dt' \leq \,C.
\end{align}
Due to (\ref{decay 1 pressure}) and magnetic field equation, we have
\begin{align*}
&\, \int_{t_{0}}^{t} \left( \|\Delta u(t')\|_{L^{2}}^{2} + \|\Delta B(t')\|_{L^{2}}^{2} + \|\nabla \Pi(t')\|_{L^{2}}^{2} \right) \,dt'   \\
\leq &\,C\int_{t_{0}}^{t}\big( \|(\partial_{t}u(t'), \partial_{t} B(t'))\|_{L^{2}}^{2} + \|(u\cdot\nabla)u(t')\|_{L^{2}}^{2}
+ \|(B\cdot \nabla)B(t')\|_{L^{2}}^{2}  \\
&\,  + \|(u\cdot\nabla)B(t')\|_{L^{2}}^{2} + \|(B\cdot\nabla)u(t')\|_{L^{2}}^{2} \big)\, dt'    \\
\leq &\, C \int_{t_{0}}^{t}\|(\partial_{t}u(t'), \partial_{t}B(t'))\|_{L^{2}}^{2}
+ \|u(t')\|_{L^{6}}^{2}(\|\nabla u(t')\|_{L^{3}}^{2} + \|\nabla B(t')\|_{L^{3}}^{2})    \\
&\, + \|B(t')\|_{L^{6}}^{2}(\|\nabla u(t')\|_{L^{3}}^{2} + \|\nabla B(t')\|_{L^{3}}^{2}) \, dt' \\
\leq &\, C.
\end{align*}
This together with (\ref{laplace u b}) proves (\ref{laplace u b final}).

At last, using (\ref{decay pro 3}), we can get that
\begin{align*}
\|u(t)\|_{L^{\infty}} \lesssim \|u(t)\|_{L^{6}}^{1/2} \|\nabla u(t)\|_{L^{6}}^{1/2} & \lesssim \|\nabla u(t)\|_{L^{2}} + \|\Delta u(t)\|_{L^{2}}  \\
& \lesssim \langle t \rangle^{-\frac{1}{2}-\beta(p)} + \|\Delta u(t)\|_{L^{2}}, \\
\|B(t)\|_{L^{\infty}} \lesssim \|B(t)\|_{L^{6}}^{1/2} \|\nabla B(t)\|_{L^{6}}^{1/2} & \lesssim \|\nabla B(t)\|_{L^{2}} + \|\Delta B(t)\|_{L^{2}}  \\
& \lesssim \langle t \rangle^{-\frac{1}{2}-\beta(p)} + \|\Delta B(t)\|_{L^{2}},
\end{align*}
which along with (\ref{laplace u b final}) and $p \in (1, \frac{6}{5})$ yields
\begin{align*}
\int_{t_{0}}^{\infty}\|u(t)\|_{L^{\infty}} + \|B(t)\|_{L^{\infty}}\, dt \leq \, C.
\end{align*}
Hence, the proof is completed.
\end{proof}

\begin{proposition}
\label{decay pro 5}
Under the assumptions of Theorem \ref{decay_main_theorem}, there holds
\begin{align}
\begin{split}
\|\nabla a(t)\|_{L^{q}} \leq \,C\,\|\nabla a_{0}\|_{L^{q}},
\end{split}
\end{align}
and
\begin{align}
\begin{split}
\|a(t)\|_{\dot{H}^{2}}\leq \,C\, \|a_{0}\|_{\dot{H}^{2}}.
\end{split}
\end{align}
\end{proposition}
\begin{proof}
According to the transport equation in (\ref{mhd_rho}) and $\mathrm{div} u = 0$, we know
\begin{align*}
\|\nabla a(t)\|_{L^{q}} \leq \|\nabla a_{0}\|_{L^{q}} + C \int_{0}^{t}\|\nabla u\|_{L^{\infty}}\|\nabla a\|_{L^{q}}\, d\tau,
\end{align*}
for $q \in [1, \infty]$.
Then Gronwall's inequality and Proposition \ref{decay pro 4} implies
\begin{align*}
\|\nabla a(t)\|_{L^{q}} \leq \|\nabla a_{0}\|_{L^{q}} e^{C\int_{0}^{t}\|\nabla u(t')\|_{L^{\infty}}\, dt'} \leq \,C\, \|\nabla a_{0}\|_{L^{q}}.
\end{align*}
On the other hand, differentiating $\partial_{t}a + u\cdot\nabla a = 0$ twice with respect to the spatial variables, we get by a standard enery
estimate that
\begin{align*}
\|a(t)\|_{\dot{H}^{2}} \leq &\, \|a_{0}\|_{\dot{H}^{2}}
+ C\int_{0}^{t}\|\nabla u\|_{L^{\infty}}\|a\|_{\dot{H}^{2}} + \|\nabla a\|_{L^{6}}\|\nabla^{2}u\|_{L^{3}}\, dt' \\
\leq &\, \|a_{0}\|_{\dot{H}^{2}} + C\int_{0}^{t}\|\nabla u\|_{L^{\infty}}\|a\|_{\dot{H}^{2}}        \\
&\, \quad\quad\quad\quad\quad\quad\quad\quad\quad + \|\Delta a\|_{L^{2}}(\|\Delta u\|_{L^{2}}^{2/3} + \|\Delta u\|_{L^{6}}^{2})\, dt'.
\end{align*}
Applying the Gronwall's inequality, Proposition \ref{decay pro 3}, Proposition \ref{decay pro 4} and (\ref{gradient u gradient pi}), we obtain
\begin{align*}
\|a(t)\|_{\dot{H}^{2}}\leq &\, \|a_{0}\|_{\dot{H}^{2}}e^{C\int_{0}^{t}\|\nabla u\|_{L^{\infty}} + \|\Delta u\|_{L^{2}}^{2/3} + \|\Delta u\|_{L^{6}}^{2} \,dt' }    \\
\leq &\, C \|a_{0}\|_{\dot{H}^{2}}.
\end{align*}
Hence, the proof is completed.
\end{proof}

\begin{proposition}
\label{decay pro 6}
Under the assumptions of Theorem \ref{decay_main_theorem}, there holds
\begin{align*}
\sup_{t\geq t_{0}}\left( \|\nabla^{2}u(t)\|_{L^{2}}^{2} + \|\nabla^{2}B(t)\|_{L^{2}}^{2} \right)
& + c_{1}\int_{t_{0}}^{\infty}\|\partial_{t}\nabla u(t)\|_{L^{2}}^{2} + \|\partial_{t}\nabla B(t)\|_{L^{2}}^{2}\, dt  \\
& + c_{2}\int_{t_{0}}^{\infty}\|\nabla^{3}u(t)\|_{L^{2}}^{2} + \|\nabla^{3}B(t)\|_{L^{2}}^{2} \, dt \leq \, C.
\end{align*}
where $C$ depend on $m$, $M$, $\|a_{0}\|_{L^{2}}$, $\|a_{0}\|_{L^{\infty}}$, $\|(u_{0}, B_{0})\|_{L^{p}}$, $\|(u_{0}, B_{0})\|_{H^{2}}$.
\end{proposition}
\begin{proof}
Taking spatial derivative to the momentum equation in (\ref{mhd_rho}), we have
\begin{align*}
\rho \partial_{t} \partial_{j}u^{i} + \partial_{j}\rho\cdot \partial_{t}u^{i} & + \partial_{j}(\rho u)\cdot \nabla u^{i}  \\
& + \rho u\cdot \nabla \partial_{j}u^{i} - \Delta \partial_{j} u^{i} + \partial_{j}\partial_{i}\Pi = \partial_{j}(B\cdot \nabla B^{i}),
\end{align*}
with $i=1,2,3$.
Standard energy estimates yields
\begin{align}
\label{d u basic}
\begin{split}
& \|\sqrt{\rho} \partial_{t}\nabla u(t)\|_{L^{2}}^{2} + \frac{1}{2}\frac{d}{dt}\|\nabla^{2} u(t)\|_{L^{2}}^{2}
= - \int_{\mathbb{R}^{3}}\nabla \rho \, \partial_{t}u \, \partial_{t}\nabla u + \nabla (\rho u)\nabla u\, \partial_{t}\nabla u\,dx  \\
&\quad\quad\quad\quad\quad\quad\quad\quad\quad\quad\quad\quad\quad
- \int_{\mathbb{R}^{3}}(\rho u\cdot \nabla)\nabla u\, \partial_{t}\nabla u - \nabla (B\cdot \nabla B)\, \partial_{t}\nabla u \, dx.
\end{split}
\end{align}
Similar argument gives
\begin{align}
\label{d b basic}
\begin{split}
\|\partial_{t}\nabla B(t)\|_{L^{2}}^{2} + \frac{1}{2}\frac{d}{dt}\|\nabla^{2}B(t)\|_{L^{2}}^{2}
= -\int_{\mathbb{R}^{3}}\nabla (u\cdot\nabla B)\,\partial_{t}\nabla B - \nabla (B\cdot \nabla u)\,\partial_{t}\nabla B \, dx.
\end{split}
\end{align}
Summing (\ref{d u basic}), (\ref{d b basic}), (\ref{basic_energy_rho}) and Proposition \ref{decay pro 5}, we obtain
\begin{align}
\label{4.54}
\begin{split}
&\, \frac{1}{2}\frac{d}{dt}\left( \|\nabla^{2}u(t)\|_{L^{2}}^{2} + \|\nabla^{2}B(t)\|_{L^{2}}^{2} \right)
+ \|\sqrt{\rho} \partial_{t}\nabla u(t)\|_{L^{2}}^{2} + \|\partial_{t}\nabla B(t)\|_{L^{2}}^{2} \\
\lesssim &\, \|\partial_{t} u\|_{L^{2}}^{2} + \|\nabla u\|_{L^{2}}^{2}\|\nabla^{2} u\|_{L^{2}}\|\nabla^{3}u\|_{L^{2}}
+ \|\nabla u\|_{L^{2}}^{3}\|\nabla^{2} u\|_{L^{2}}  \\
&\, + \|u\|_{L^{2}}\|\nabla u\|_{L^{2}}\|\nabla^{3}u\|_{L^{2}}^{2} + \|\nabla B\|_{L^{2}}^{2}\|\nabla^{2}B\|_{L^{2}}\|\nabla^{3}B\|_{L^{2}}    \\
&\, + \|B\|_{L^{2}}\|\nabla B\|_{L^{2}}\|\nabla^{3}B\|_{L^{2}}^{2} + \|\nabla B\|_{L^{2}}^{2}\|\nabla^{2}u\|_{L^{2}}\|\nabla^{3}u\|_{L^{2}} \\
&\, + \|u\|_{L^{2}}\|\nabla u\|_{L^{2}}\|\nabla^{3}B\|_{L^{2}}^{2} + \|B\|_{L^{2}}\|\nabla B\|_{L^{2}}\|\nabla^{3}u\|_{L^{2}}^{2}.
\end{split}
\end{align}
On the other hand, taking partial derivative to the momentum equation in (\ref{mhd_rho}) and multiplying $\frac{1}{\rho}\Delta \nabla u$ on both side,
we can get
\begin{align}
\label{dd u basic}
\begin{split}
\frac{1}{2}\frac{d}{dt} \|\nabla^{2}u(t)\|_{L^{2}}^{2}& + \|\frac{1}{\sqrt{\rho}}\Delta \nabla u\|_{L^{2}}^{2}
= \int_{\mathbb{R}^{3}}\nabla \rho \, \partial_{t}u\, \frac{1}{\rho}\, \Delta \nabla u\, dx \\
& + \int_{\mathbb{R}^{3}} \nabla(\rho u)\, \nabla u\,\frac{1}{\rho}\, \Delta\nabla u\,dx
+ \int_{\mathbb{R}^{3}}\nabla(\nabla \Pi) \,\frac{1}{\rho}\, \Delta\nabla u\,dx     \\
&\, - \int_{\mathbb{R}^{3}}\nabla (B\cdot \nabla B)\,\frac{1}{\rho}\,\Delta\nabla u\,dx
+ \int_{\mathbb{R}^{3}} u\cdot\nabla(\nabla u)\,\Delta\nabla u\,dx
\end{split}
\end{align}
by integrate over $\mathbb{R}^{3}$.
Through similar estimate, we obtain
\begin{align}
\label{dd b basic}
\begin{split}
\frac{1}{2}\frac{d}{dt}\|\nabla^{2}B\|_{L^{2}}^{2} & + \|\nabla^{3}B\|_{L^{2}}^{2}    \\
& = \int_{\mathbb{R}^{3}}\nabla(u\cdot\nabla B)\,\Delta\nabla B\,dx
-\int_{\mathbb{R}^{3}}\nabla(B\cdot\nabla u)\,\Delta \nabla B\, dx.
\end{split}
\end{align}
Combining (\ref{dd u basic}) and (\ref{dd b basic}), we have
\begin{align}
\label{dd u b}
\begin{split}
&\, \frac{1}{2}\frac{d}{dt}\left( \|\nabla^{2}u\|_{L^{2}}^{2} + \|\nabla^{2}B\|_{L^{2}}^{2} \right) + \|\frac{1}{\sqrt{\rho}}\nabla^{3}u\|_{L^{2}}^{2}
+ \|\nabla^{3}B\|_{L^{2}}^{2}   \\
= &\, \int_{\mathbb{R}^{3}} \nabla \rho \, \partial_{t}u\, \frac{1}{\rho}\,\Delta\nabla u\, dx
+ \int_{\mathbb{R}^{3}} \nabla (\rho u)\, \Delta u \, \frac{1}{\rho}\, \Delta \nabla u\, dx \\
&\, + \int_{\mathbb{R}^{3}} u\cdot \nabla (\nabla u)\,\Delta \nabla u\,dx + \int_{\mathbb{R}^{3}}(1+a)\nabla(\nabla \Pi)\, \Delta \nabla u\, dx \\
&\, - \int_{\mathbb{R}^{3}}\nabla(B\cdot\nabla B)\,\frac{1}{\rho}\,\Delta\nabla u\,dx
+ \int_{\mathbb{R}^{3}}\nabla(u\cdot\nabla B)\,\Delta \nabla B\,dx  \\
&\, - \int_{\mathbb{R}^{3}} \nabla(B\cdot\nabla u)\,\Delta \nabla B \, dx.
\end{split}
\end{align}
Next, we give the estimation of each term on the right hand side of (\ref{dd u b}),
\begin{align*}
\int_{\mathbb{R}^{3}} \nabla \rho \, \partial_{t}u \,\frac{1}{\rho}\, \Delta\nabla u\,dx \lesssim \|1+a\|_{L^{\infty}}
\|\nabla a\|_{L^{\infty}} \|\partial_{t}u\|_{L^{2}}\|\nabla^{3}u\|_{L^{2}},
\end{align*}
\begin{align*}
\int_{\mathbb{R}^{3}} \nabla(\rho u)\,\nabla u \,\frac{1}{\rho}\,\Delta\nabla u\,dx \lesssim &\,
\|1+a\|_{L^{\infty}}\|\nabla a\|_{L^{\infty}}\|u\|_{L^{2}}^{1/2}\|\nabla u\|_{L^{2}}^{1/2}\|\nabla^{2}u\|_{L^{2}}\|\nabla^{3}u\|_{L^{2}}    \\
&\, + \left( \|\nabla^{2} u\|_{L^{2}} + \|\nabla^{3}u\|_{L^{2}} \right) \|\nabla u\|_{L^{2}}\|\nabla^{3}u\|_{L^{2}},
\end{align*}
\begin{align*}
\int_{\mathbb{R}^{3}}(u\cdot\nabla)\nabla u\,\Delta\nabla u\,dx \lesssim \|u\|_{L^{2}}^{1/2}\|\nabla u\|_{L^{2}}^{1/2}\|\nabla^{3}u\|_{L^{2}}^{2},
\end{align*}
\begin{align*}
\int_{\mathbb{R}^{3}}\nabla(B\cdot \nabla B)\,\frac{1}{\rho}\,\Delta\nabla u\,dx \lesssim &\, \|1+a\|_{L^{\infty}}\|B\|_{L^{2}}^{1/2}
\|\nabla B\|_{L^{2}}^{1/2}\|\nabla^{3}B\|_{L^{2}}\|\nabla^{3}u\|_{L^{2}}    \\
+ & \|1+a\|_{L^{\infty}}\left( \|\nabla^{2}B\|_{L^{2}} + \|\nabla^{3}B\|_{L^{2}} \right)\|\nabla B\|_{L^{2}}\|\nabla^{3}u\|_{L^{2}},
\end{align*}
\begin{align*}
\int_{\mathbb{R}^{3}}\nabla(u\cdot\nabla B)\,\Delta\nabla B\,dx \lesssim &\, \|u\|_{L^{2}}^{1/2}\|\nabla u\|_{L^{2}}^{1/2}\|\nabla^{3}B\|_{L^{2}}^{2}   \\
&\, + \left( \|\nabla^{2}u\|_{L^{2}} + \|\nabla^{3}u\|_{L^{2}} \right)\|\nabla B\|_{L^{2}}\|\nabla^{3}B\|_{L^{2}},
\end{align*}
\begin{align*}
\int_{\mathbb{R}^{3}}(1+a)\,\nabla(\nabla \Pi) \, \Delta\nabla u\, dx \lesssim \|1+a\|_{L^{\infty}}\|\nabla^{2}\Pi\|_{L^{2}}\|\nabla^{3}u\|_{L^{2}}.
\end{align*}
Moreover, using $\mathrm{div}u = 0$, we can get
\begin{align*}
&\, \|\Delta\nabla u(t)\|_{L^{2}} + \|\nabla(\nabla \Pi(t))\|_{L^{2}}   \\
\leq &\, \sqrt{2}\|\Delta\nabla u - \nabla(\nabla \Pi)\|_{L^{2}}    \\
\leq &\, \sqrt{2}\|\rho\partial_{t}\nabla u + \nabla \rho \,\partial_{t}u + \nabla(\rho u)\,\nabla u
+ \rho u \cdot \nabla (\nabla u) - \nabla(B\cdot \nabla B) \|_{L^{2}}   \\
\leq &\, \|\sqrt{\rho}\partial_{t}\nabla u\|_{L^{2}} + \|\partial_{t}u\|_{L^{2}} + \|u\|_{L^{2}}^{1/2}\|\nabla u\|_{L^{2}}^{1/2}\|\Delta u\|_{L^{2}}
+ \|B\|_{L^{2}}^{1/2}\|\nabla B\|_{L^{2}}^{1/2}\|\nabla^{3}B\|_{L^{2}}   \\
&\, + \|u\|_{L^{2}}^{1/2}\|\nabla u\|_{L^{2}}^{1/2}\|\nabla^{3}u\|_{L^{2}} + \|\nabla u\|_{L^{2}}\|\nabla^{2}u\|_{L^{2}}^{\frac{1}{2}}
\|\nabla^{3}u\|_{L^{2}}^{\frac{1}{2}}   \\
&\, + \|\nabla B\|_{L^{2}}\|\nabla^{2}B\|_{L^{2}}^{\frac{1}{2}}\|\nabla^{3}B\|_{L^{2}}^{\frac{1}{2}}.
\end{align*}
Summing up all the above estimations in (\ref{dd u b}), then after a long and tedious calculations, we have
\begin{align}
\label{dd u b 2}
\begin{split}
&\, \frac{1}{2}\frac{d}{dt}\left( \|\nabla^{2}u(t)\|_{L^{2}}^{2} + \|\nabla^{2}B(t)\|_{L^{2}}^{2} \right)
+\|\frac{1}{\sqrt{\rho}}\nabla^{3}u(t)\|_{L^{2}}^{2} + \|\nabla^{3}B(t)\|_{L^{2}}^{2}   \\
\lesssim &\, \|\partial_{t}u\|_{L^{2}}^{2} + \|u\|_{L^{2}}\|\nabla u\|_{L^{2}}\|\nabla^{2}u\|_{L^{2}}^{2}
+ \|\nabla^{2}u\|_{L^{2}}^{2}\|\nabla u\|_{L^{2}}^{2} + \|\nabla^{2}B\|_{L^{2}}^{2}\|\nabla B\|_{L^{2}}^{2}  \\
&\, + \|\nabla B\|_{L^{2}}^{2}\|\nabla^{2}u\|_{L^{2}}^{2} + \|\nabla u\|_{L^{2}}^{4}\|\nabla^{2}u\|_{L^{2}}^{2}
+ \|\nabla B\|_{L^{2}}^{4}\|\nabla^{2}B\|_{L^{2}}^{2} + \|\sqrt{\rho}\partial_{t}\nabla u\|_{L^{2}}^{2} \\
&\, + \|\nabla u\|_{L^{2}}^{2}\|\nabla^{3}u\|_{L^{2}}^{2} + \|u\|_{L^{2}}^{1/2}\|\nabla u\|_{L^{2}}^{1/2}\|\nabla^{3}u\|_{L^{2}}^{2}
+ \|B\|_{L^{2}}\|\nabla B\|_{L^{2}}\|\nabla^{3} B\|_{L^{2}}^{2}     \\
&\, + \|\nabla B\|_{L^{2}}^{2}\|\nabla^{3}B\|_{L^{2}}^{2} + \|u\|_{L^{2}}^{1/2}\|\nabla u\|_{L^{2}}^{1/2}\|\nabla^{3}B\|_{L^{2}}^{2}
+ \|u\|_{L^{2}}\|\nabla u\|_{L^{2}}\|\nabla^{3}u\|_{L^{2}}^{2}.
\end{split}
\end{align}
We perform $(\ref{4.54}) + \frac{1}{2}(\ref{dd u b 2})$ so that
\begin{align}
\label{dd half}
\begin{split}
&\, \frac{d}{dt}\left( \|\nabla^{2}\|_{L^{2}}^{2} + \|\nabla^{2}B\|_{L^{2}}^{2} \right) + \|\sqrt{\rho}\partial_{t}\nabla u\|_{L^{2}}^{2}
+ \|\partial_{t}\nabla B\|_{L^{2}}^{2} + \|\frac{1}{\sqrt{\rho}}\nabla^{3}u\|_{L^{2}}^{2} + \|\nabla^{3}B\|_{L^{2}}^{2} \\
\lesssim &\, \|\partial_{t}u\|_{L^{2}}^{2} + \left(\|\nabla u\|_{L^{2}}^{4} + \|\nabla B\|_{L^{2}}^{4}\right)
\left( \|\nabla^{2}u\|_{L^{2}}^{2} + \|\nabla^{2}B\|_{L^{2}}^{2} \right) + \|\nabla u\|_{L^{2}}^{3}\|\nabla^{2}u\|_{L^{2}}  \\
&\, + \|u\|_{L^{2}}\|\nabla u\|_{L^{2}}\|\nabla^{2}u\|_{L^{2}}^{2} + \|\nabla^{2}u\|_{L^{2}}^{2}\|\nabla u\|_{L^{2}}^{2}
+ \|\nabla^{2}B\|_{L^{2}}^{2}\|\nabla B\|_{L^{2}}^{2} \\
&\,  + \|\nabla^{2}u\|_{L^{2}}^{2}\|\nabla B\|_{L^{2}}^{2} + \big( \|u\|_{L^{2}}\|\nabla u\|_{L^{2}} + \|\nabla u\|_{L^{2}}
+ \|B\|_{L^{2}}\|\nabla B\|_{L^{2}} + \|\nabla B\|_{L^{2}}^{2}  \\
&\, + \|u\|_{L^{2}}^{1/2}\|\nabla u\|_{L^{2}}^{1/2} \big) \big( \|\nabla^{3}u\|_{L^{2}}^{2} + \|\nabla^{3}B\|_{L^{2}}^{2} \big).
\end{split}
\end{align}
Note that we can take $\eta >0$ in the proof of Proposition \ref{decay pro 1} smaller so that
\begin{align*}
\|u_{0}\|_{L^{2}}\eta +\|B_{0}\|_{L^{2}}\eta + \|u_{0}\|_{L^{2}}^{1/2}\eta^{1/2} + \eta^{2} \leq \frac{c}{16 C^{2}}
\end{align*}
where $C$ is the constant on the right hand side of (\ref{dd half}), $c$ satisfies
\begin{align*}
c\left(\|\nabla^{3}u\|_{L^{2}}^{2} + \|\nabla^{3}B\|_{L^{2}}^{2}\right) \leq \|\frac{1}{\sqrt{\rho}}\nabla^{3}u\|_{L^{2}}^{2} + \|\nabla^{3}B\|_{L^{2}}^{2}.
\end{align*}
From the proof of Proposition \ref{decay pro 1}, we know that
$\tau^{*} = \infty$ as in Position \ref{decay pro 1}.
Hence, there exists two constant $c_{1}$ and $c_{2}$ such that
\begin{align*}
\begin{split}
&\, \frac{d}{dt}\left( \|\nabla^{2}u\|_{L^{2}}^{2} + \|\nabla^{2}B\|_{L^{2}}^{2} \right)
+ c_{1}\left( \|\partial_{t}\nabla u\|_{L^{2}}^{2} + \|\partial_{t}\nabla B\|_{L^{2}}^{2}\right)    \\
&\, \quad\quad\quad\quad\quad\quad\quad\quad\quad\quad\quad\quad\quad\quad\quad\quad
+ c_{2}\left( \|\nabla^{3}u\|_{L^{2}}^{2} + \|\nabla^{3}B\|_{L^{2}}^{2} \right) \\
\lesssim &\, \|\partial_{t}u\|_{L^{2}}^{2} + \left( \|\nabla u\|_{L^{2}}^{4} + \|\nabla B\|_{L^{2}}^{4} \right)
\left( \|\nabla^{2}u\|_{L^{2}}^{2} + \|\nabla^{2}B\|_{L^{2}}^{2} \right)    \\
&\, + \|\nabla u\|_{L^{2}}^{3}\|\nabla^{2}u\|_{L^{2}} + \|u\|_{L^{2}}\|\nabla u\|_{L^{2}}\|\nabla^{2}u\|_{L^{2}}^{2}
+ \|\nabla^{2}u\|_{L^{2}}^{2}\|\nabla u\|_{L^{2}}^{2}   \\
&\, + \|\nabla^{2}B\|_{L^{2}}^{2}\|\nabla B\|_{L^{2}}^{2} + \|\nabla^{2}u\|_{L^{2}}^{2}\|\nabla B\|_{L^{2}}^{2},
\end{split}
\end{align*}
for $t\geq t_{0}$. Integrating the above inequality form $t_{0}$ to $\infty$, we then obtain
\begin{align}
\begin{split}
&\, \sup_{t\geq t_{0}} \left( \|\nabla^{2}u(t)\|_{L^{2}}^{2} + \|\nabla^{3}B(t)\|_{L^{2}}^{2} \right)
+ c_{1}\int_{t_{0}}^{\infty} \|\partial_{t}\nabla u\|_{L^{2}}^{2} + \|\partial_{t}\nabla B(t)\|_{L^{2}}^{2} \, dt   \\
&\quad\quad\quad\quad\quad\quad\quad\quad\quad\quad\quad\quad\quad\quad\,\,\,
+ c_{2}\int_{t_{0}}^{\infty} \|\nabla^{3}u\|_{L^{2}}^{2} + \|\nabla^{3}B\|_{L^{2}}^{2}\,dt   \\
\lesssim &\, \|\nabla^{2}u(t_{0})\|_{L^{2}}^{2} + \|\nabla^{2}B(t_{0})\|_{L^{2}}^{2} + \int_{t_{0}}^{\infty} \|\partial_{t}u\|_{L^{2}}^{2}\,dt  \\
&\, + \sup_{t\in [t_{0},\infty]} \left( \|\nabla u\|_{L^{2}}^{4} + \|\nabla B\|_{L^{2}}^{4} \right)
\int_{t_{0}}^{\infty} \|\nabla^{2}u\|_{L^{2}}^{2} + \|\nabla^{2}B\|_{L^{2}}^{2}\, dt    \\
&\, + \sup_{t\in [t_{0},\infty]}\|\nabla u(t)\|_{L^{2}}^{3}\int_{t_{0}}^{\infty}\|\nabla^{2}u\|_{L^{2}}\,dt
+ \sup_{t\in [t_{0},\infty]} \|u(t)\|_{L^{2}}\|\nabla u(t)\|_{L^{2}}\int_{t_{0}}^{\infty}\|\nabla^{2}u\|_{L^{2}}^{2}\,dt    \\
&\, + \sup_{t\in [t_{0},\infty]}\left(\|\nabla u\|_{L^{2}}^{2} + \|\nabla B(t)\|_{L^{2}}^{2}\right)
\int_{t_{0}}^{\infty} \|\nabla^{2}u\|_{L^{2}}^{2} + \|\nabla^{2} B\|_{L^{2}}^{2}\,dt
\end{split}
\end{align}
At last by Proposition \ref{decay pro 3}, the proof is completed.
\end{proof}

%%%%%%%%%%%%%%%%%%%%%%%%%%%%%%%%%%%%%%%%%%%%%%%%%%%%%%%%%%%%%%%%%%%%%%%%%%%%%%%%%%%%%%%%%%%%%%%%%%%%%%%%%%%%%%%%%%%%%%%%%%%%%%%%%%%%%%%%%%%%%%%%%%%%%%%%%%%

\section{Global in Time Estimates for Reference Solutions}

In this section, we prove the global in time estimates for the reference solution of (\ref{mhd_rho}).
The proof will be based mainly on Theorem \ref{decay_main_theorem}.

\begin{proposition}
\label{global reference 5/2}
Under the assumptions of Theorem \ref{stability main theorem}, there holds
\begin{align}
\|(\bar{u}, \bar{B})\|_{\tilde{L}^{\infty}(\mathbb{R}^{+}; \dot{B}_{2,1}^{1/2})}
+ \|(\bar{u}, \bar{B})\|_{L^{1}(\mathbb{R}^{+}; \dot{B}_{2,1}^{5/2})} \leq C,
\end{align}
for some $C$ depending on the initial data.
\end{proposition}
\begin{proof}
According to imbedding theorems of Besov space and Theorem \ref{decay_main_theorem}, we have
\begin{align*}
\|\bar{u}\|_{\tilde{L}^{\infty}(\mathbb{R}^{+} ; \dot{B}_{2,1}^{1/2})} \lesssim &\, \|\bar{u}\|_{L^{\infty}(\mathbb{R}^{+}; H^{1})}   \\
\lesssim &\, \|\bar{u}\|_{L^{\infty}(\mathbb{R}^{+};L^{2})} + \|\nabla \bar{u}\|_{L^{\infty}(\mathbb{R}^{+};L^{2})} \leq C.
\end{align*}
Same considerations yields
\begin{align*}
\|\bar{B}\|_{\tilde{L}^{\infty}(\mathbb{R}^{+} ; \dot{B}_{2,1}^{1/2})} \leq C.
\end{align*}
By imbedding theorem of Besov space, we know that
\begin{align*}
\|\bar{u}\|_{L^{1}(\mathbb{R}^{+};\dot{B}_{2,1}^{5/2})} \lesssim \|\bar{u}\|_{L^{1}(\mathbb{R}^{+}; \dot{H}^{2})}
+ \|\bar{u}\|_{L^{1}(\mathbb{R}^{+}; \dot{H}^{3})}, \\
\|\bar{B}\|_{L^{1}(\mathbb{R}^{+};\dot{B}_{2,1}^{5/2})} \lesssim \|\bar{B}\|_{L^{1}(\mathbb{R}^{+}; \dot{H}^{2})}
+ \|\bar{B}\|_{L^{1}(\mathbb{R}^{+}; \dot{H}^{3})}.
\end{align*}
Rewrite the momentum equation in (\ref{mhd_rho}) as follows
\begin{align*}
- \Delta \bar{u} + \nabla \bar{\Pi} = - \frac{1}{1+\bar{a}} \, \partial_{t}\bar{u} - \frac{1}{1+\bar{a}} \, \bar{u}\cdot\nabla\bar{u}
+\bar{B}\cdot\nabla\bar{B}.
\end{align*}
From the elliptic estimate, we have
\begin{align*}
\|\bar{u}\|_{L^{1}(\mathbb{R}^{+} ; \dot{H}^{2})} \leq &\, \|\frac{1}{1+\bar{a}}\, \partial_{t}\bar{u}\|_{L^{1}(\mathbb{R}^{+};L^{2})}
+ \|\frac{1}{1+\bar{a}}\, \bar{u}\cdot\nabla\bar{u}\|_{L^{1}(\mathbb{R}^{+} ; L^{2})}   \\
&\, + \|\bar{B}\cdot\nabla\bar{B}\|_{L^{1}(\mathbb{R}^{+} ; L^{2})} \\
\lesssim &\, \|\partial_{t}\bar{u}\|_{L^{1}(\mathbb{R}^{+} ; L^{2})} + \|\bar{u}\cdot\nabla\bar{u}\|_{L^{1}(\mathbb{R}^{+} ; L^{2})}
+ \|\bar{B}\cdot\nabla\bar{B}\|_{L^{1}(\mathbb{R}^{+} ; L^{2})}.
\end{align*}
Next, we estimate the three terms on the right hand side of the above inequality.
Simple calculations yields
\begin{align*}
\int_{t_{0}}^{\infty} \|\partial_{t}\bar{u}\|_{L^{2}}\, dt \lesssim &\,
\left( \int_{t_{0}}^{\infty}\langle t \rangle^{(1+2 \beta(p))^{-}} \|\partial_{t}\bar{u}\|_{L^{2}}^{2} \, dt\right)^{1/2}
\left( \int_{t_{0}}^{\infty} \langle t \rangle^{-(1+2\beta(p))^{-}} \,dt \right)^{1/2}  \\
\leq &\, C,
\end{align*}
\begin{align*}
\int_{t_{0}}^{\infty}\|\bar{u}\cdot\nabla\bar{u}\|_{L^{2}}\, dt \lesssim &\, \int_{t_{0}}^{\infty}\|\bar{u}\|_{L^{4}}\|\nabla\bar{u}\|_{L^{4}}\, dt \\
\lesssim &\, \int_{t_{0}}^{\infty}\|\bar{u}\|_{L^{2}}^{1/4}\|\nabla\bar{u}\|_{L^{2}}\|\nabla^{2}\bar{u}\|_{L^{2}}^{3/4}\, dt    \\
\lesssim &\, \|\bar{u}\|_{L^{2}}^{1/4}\left( \int_{t_{0}}^{\infty}\|\nabla\bar{u}\|_{L^{2}}^{4}\,dt \right)^{1/4}
\left( \int_{t_{0}}^{\infty} \|\nabla^{2}\bar{u}\|_{L^{2}}\, dt \right)^{3/4} \\
\leq &\, C.
\end{align*}
\begin{align*}
\int_{t_{0}}^{\infty}\|\bar{B}\cdot\nabla\bar{B}\|_{L^{2}}\, dt \lesssim &\,
\|\bar{B}\|_{L^{2}}^{1/4} \left( \int_{t_{0}}^{\infty} \|\nabla\bar{B}\|_{L^{2}}^{4} \, dt \right)^{1/4}
\left( \int_{t_{0}}^{\infty} \|\nabla^{2}\bar{B}\|_{L^{2}} \, dt \right)^{3/4}    \\
\leq &\, C.
\end{align*}
Now we give the estimates about $\|\bar{u}\|_{L^{1}(\mathbb{R}^{+} ; \dot{H}^{3})}$,
\begin{align*}
\|\bar{u}\|_{L^{1}(\mathbb{R}^{+} ; \dot{H}^{3})} \lesssim &\, \left\|\nabla \left( \frac{1}{1+\bar{a}}\, \partial_{t}\bar{u} \right)\right\|_{L^{1}(\mathbb{R}^{+};L^{2})}
 + \left\|\nabla \left( \frac{1}{1+\bar{a}}\, \bar{u}\cdot\nabla\bar{u} \right) \right\|_{L^{1}(\mathbb{R}^{+} ; L^{2})}    \\
&\, + \left\| \nabla \left( \bar{B}\cdot\nabla\bar{B} \right) \right\|_{L^{1}(\mathbb{R}^{+} ; L^{2})}  \\
\lesssim &\, \|\nabla \bar{a} \cdot \partial_{t}\bar{u}\|_{L^{1}(\mathbb{R}^{+} ; L^{2})} + \|\partial_{t}\nabla\bar{u}\|_{L^{1}(\mathbb{R}^{+};L^{2})}
+ \|\nabla\bar{a}\cdot\bar{u}\cdot\nabla\bar{u}\|_{L^{1}(\mathbb{R}^{+} ; L^{2})}   \\
&\, + \|\nabla\bar{u}\cdot\nabla\bar{u}\|_{L^{1}(\mathbb{R}^{+} ; L^{2})} + \|\bar{u}\cdot\nabla(\nabla\bar{u})\|_{L^{1}(\mathbb{R}^{+} ; L^{2})}   \\
&\, + \|\nabla\bar{B}\cdot\nabla\bar{B}\|_{L^{1}(\mathbb{R}^{+};L^{2})} + \|\bar{B}\cdot\nabla(\nabla\bar{B})\|_{L^{1}(\mathbb{R}^{+};L^{2})}.
\end{align*}
Now, we give the estimates of the right hand side.
According to Proposition \ref{decay pro 3} to Proposition \ref{decay pro 6}, we have
\begin{align*}
\int_{t_{0}}^{\infty}\|\nabla\bar{a}\cdot\partial_{t}\bar{u}\|_{L^{2}}\, dt \lesssim \int_{t_{0}}^{\infty}\|\partial_{t}\bar{u}\|_{L^{2}}\,dt \leq\, C,
\end{align*}
\begin{align*}
\int_{t_{0}}^{\infty}\|\partial_{t}\nabla\bar{u}\|_{L^{2}}\, dt \leq \,C,
\end{align*}
\begin{align*}
&\, \int_{t_{0}}^{\infty}\|\nabla\bar{a}\cdot\bar{u}\cdot\nabla\bar{u}\|_{L^{2}}\, dt   \\
\lesssim &\, \|\nabla\bar{a}\|_{L^{\infty}(\mathbb{R}^{+} ; L^{2})}\int_{t_{0}}^{\infty}\|\bar{u}\cdot\nabla\bar{u}\|_{L^{2}}\,dt   \\
\lesssim &\, \|\nabla\bar{a}\|_{L^{\infty}(\mathbb{R}^{+} ; L^{2})}
\int_{t_{0}}^{\infty}\|\bar{u}\|_{L^{2}}^{\frac{1}{2}}\|\nabla\bar{u}\|_{L^{2}}^{\frac{1}{2}}\|\nabla^{2}\bar{u}\|_{L^{2}}\,dt  \\
\lesssim &\, \|\nabla\bar{a}\|_{L^{\infty}(\mathbb{R}^{+};L^{2})}\|\bar{u}\|_{L^{\infty}(\mathbb{R}^{+};L^{2})}^{\frac{1}{2}}
\|\nabla\bar{u}\|_{L^{\infty}(\mathbb{R}^{+};L^{2})}^{\frac{1}{2}}\int_{t_{0}}^{\infty}\|\nabla^{2}\bar{u}\|_{L^{2}}\,dt  \\
\leq &\, C,
\end{align*}
\begin{align*}
\int_{t_{0}}^{\infty}\|\nabla\bar{u}\cdot\nabla\bar{u}\|_{L^{2}}\,dt \lesssim &\, \int_{t_{0}}^{\infty}\|\nabla\bar{u}\|_{L^{4}}
\|\nabla\bar{u}\|_{L^{4}} \,dt  \\
\lesssim &\, \int_{t_{0}}^{\infty}\|\nabla\bar{u}\|_{L^{2}}^{1/4}\|\nabla^{2}\bar{u}\|_{L^{2}}^{3/4}\|\nabla\bar{u}\|_{L^{2}}^{1/4}
\|\nabla^{2}\bar{u}\|_{L^{2}}^{3/4}\, dt    \\
\lesssim &\, \int_{t_{0}}^{\infty}\|\nabla\bar{u}\|_{L^{2}}^{1/2}\|\nabla^{2}\bar{u}\|_{L^{2}}^{3/2}\, dt \leq \,C,
\end{align*}
\begin{align*}
&\, \int_{t_{0}}^{\infty}\|\bar{u}\cdot\nabla(\nabla\bar{u})\|_{L^{2}}\,dt  \\
\lesssim &\, \int_{t_{0}}^{\infty}\|\bar{u}\|_{L^{\infty}}\|\nabla^{2}\bar{u}\|_{L^{2}}\,dt \\
\lesssim &\, \int_{t_{0}}^{\infty}\|\nabla\bar{u}\|_{L^{2}}\|\nabla^{2}\bar{u}\|_{L^{2}}\,dt
+ \int_{t_{0}}^{\infty}\|\nabla^{2}\bar{u}\|_{L^{2}}^{2}\,dt \leq \,C.
\end{align*}
Similar to the above estimates of $\bar{u}$, we can obtain
\begin{align*}
\int_{t_{0}}^{\infty}\|\nabla\bar{B}\cdot\nabla\bar{B}\|_{L^{2}}\,dt \leq \,C
\quad \text{and} \quad
\int_{t_{0}}^{\infty}\|\bar{B}\cdot\nabla(\nabla\bar{B})\|_{L^{2}}\,dt \leq \,C.
\end{align*}
Summing up all the above estimates, we finally obtain
\begin{align*}
\|\bar{u}\|_{L^{1}(\mathbb{R}^{+} ; \dot{B}_{2,1}^{5/2})} \leq \, C.
\end{align*}
Using similar ideas as for $\bar{u}$, we have
\begin{align*}
\|\bar{B}\|_{L^{1}(\mathbb{R}^{+};\dot{H}^{2})} \leq &\, \|\partial_{t}\bar{B}\|_{L^{1}(\mathbb{R}^{+};L^{2})}
+ \|\bar{u}\cdot\nabla\bar{B}\|_{L^{1}(\mathbb{R}^{+};L^{2})} + \|\bar{B}\cdot\nabla\bar{u}\|_{L^{1}(\mathbb{R}^{+};L^{2})} \\
\leq \, C,
\end{align*}
and
\begin{align*}
&\, \|\bar{B}\|_{L^{1}(\mathbb{R}^{+};\dot{H}^{3})}     \\
\leq &\, \|\partial_{t}\nabla\bar{B}\|_{L^{1}(\mathbb{R}^{+};L^{2})}
+ \|\nabla\,(\bar{B}\cdot\nabla\bar{u})\|_{L^{1}(\mathbb{R}^{+};L^{2})} + \|\nabla\, (\bar{u}\cdot\nabla\bar{B})\|_{L^{1}(\mathbb{R}^{+};L^{2})} \\
\leq &\, C.
\end{align*}
So we also obtain that
\begin{align*}
\|\bar{B}\|_{L^{1}(\mathbb{R}^{+} ; \dot{B}_{2,1}^{5/2})} \leq \, C.
\end{align*}
Hence, finally the proof of Proposition \ref{global reference 5/2} is completed.
\end{proof}

\begin{proposition}
\label{global reference 2}
Under the assumptions of Theorem \ref{stability main theorem}, there hold
\begin{align}
\label{estimate of a in 5 2}
\|\bar{a}\|_{\tilde{L}^{\infty}(\mathbb{R}^{+} ; B_{2,1}^{5/2})} \leq \, C
\end{align}
and
\begin{align}
& \|\bar{u}\|_{L^{\infty}(\mathbb{R}^{+};L^{p})} + \|\bar{u}\|_{\tilde{L}^{\infty}(\mathbb{R}^{+};B_{2,1}^{2})}
+ \|\bar{u}\|_{L^{1}(\mathbb{R}^{+};\dot{B}_{2,1}^{4})} + \|\nabla \bar{\Pi}\|_{L^{1}(\mathbb{R}^{+};B_{2,1}^{2})} \leq \,C,  \\
& \quad\quad\quad\quad \|\bar{B}\|_{L^{\infty}(\mathbb{R}^{+};L^{p})} + \|\bar{B}\|_{\tilde{L}^{\infty}(\mathbb{R}^{+};B_{2,1}^{2})}
+ \|\bar{B}\|_{L^{1}(\mathbb{R}^{+};\dot{B}_{2,1}^{4})} \leq \,C.
\end{align}
\end{proposition}
\begin{proof}
Thanks to Proposition \ref{global reference 5/2}, we get by applying (\ref{transport_estimate}) to the transport equation
in (\ref{mhd_a}) that
\begin{align}
\label{estimate a}
\begin{split}
\|\bar{a}\|_{\tilde{L}_{t}^{\infty}(B_{2,1}^{5/2})} \leq \|\bar{a}_{0}\|_{B_{2,1}^{5/2}}
\exp{\left\{ C\int_{0}^{t} \|\bar{u}(\tau)\|_{\dot{B}_{2,1}^{5/2}}\,d\tau \right\}} \leq \,C.
\end{split}
\end{align}
Next, let us turn to the estimates of $\bar{u}$ and $\bar{B}$. Indeed, from (\ref{L p estimate}), we know that
\begin{align*}
\begin{split}
\|\bar{u}\|_{L_{t}^{\infty}(L^{p})} + \|\bar{B}\|_{L_{t}^{\infty}(L^{p})} \lesssim &\, \|(u_{0}, B_{0})\|_{L^{p}}
+ \int_{0}^{t}\|(u,B)\|_{L^{2}\cap L^{3}}\|(\nabla u, \nabla B)\|_{L^{2}}\,dt'  \\
&\, + \|\nabla \Pi\|_{L_{t}^{1}(L^{p})}.
\end{split}
\end{align*}
Then by (\ref{decay 2 pressure}), Proposition \ref{decay pro 5}, Theorem \ref{decay_main_theorem} and the above inequality, we deduce that
\begin{align}
\label{lp}
\|\bar{u}\|_{L^{\infty}(\mathbb{R}^{+};L^{p})} + \|\bar{B}\|_{L^{\infty}(\mathbb{R}^{+};L^{p})}\leq \,C.
\end{align}
On the other hand, applying Proposition \ref{linear estimate momentum} to the momentum and magnetic field equation of (\ref{mhd_a})
ensures that
\begin{align*}
&\, \|\bar{u}\|_{\tilde{L}_{T}^{\infty}(\dot{B}_{2,1}^{0})} + \|\bar{u}\|_{\tilde{L}_{T}^{1}(\dot{B}_{2,1}^{2})}
+ \|\bar{B}\|_{\tilde{L}_{T}^{\infty}(\dot{B}_{2,1}^{0})} + \|\bar{B}\|_{\tilde{L}_{T}^{1}(\dot{B}_{2,1}^{2})}  \\
\lesssim &\, \exp{\left( C\int_{0}^{T} \|\bar{u}\|_{\dot{B}_{2,1}^{5/2}} + \|\bar{B}\|_{\dot{B}_{2,1}^{5/2}} \, dt \right)}
\bigg( \|\bar{u}_{0}\|_{\dot{B}_{2,1}^{0}} + \|\bar{B}_{0}\|_{\dot{B}_{2,1}^{0}} + \|\bar{a}\,\bar{B}\cdot\nabla\bar{B}\|_{L^{1}_{T}(\dot{B}_{2,1}^{0})} \\
&\, + \|\bar{a}\|_{\tilde{L}_{T}^{\infty}(\dot{H}^{3/2})}\|\nabla \bar{\Pi}\|_{L_{T}^{1}(L^{2})}
+ \|\bar{a}\|_{L^{\infty}_{T}(\dot{H}^{2})}\|\bar{u}\|_{L^{1}_{T}(\dot{H}^{3/2})} \bigg).
\end{align*}
We know that
\begin{align*}
\|\bar{a}\, \bar{B}\cdot\nabla\bar{B}\|_{L_{T}^{1}(\dot{B}_{2,1}^{0})}
\lesssim &\, \int_{0}^{T}\|\bar{a}\|_{\dot{B}_{2,1}^{1}}\|\bar{B}\cdot\nabla\bar{B}\|_{\dot{B}_{2,1}^{1/2}}\, dt    \\
\lesssim &\, \int_{0}^{T}\|\bar{B}\|_{\dot{B}_{2,1}^{1/2}}\|\nabla\bar{B}\|_{\dot{B}_{2,1}^{3/2}}\, dt \leq \,C.
\end{align*}
Combining the above estimation, (\ref{estimate a}) and Proposition \ref{decay pro 4}, we obtain
\begin{align}
\label{l2}
\|\bar{u}\|_{\tilde{L}_{T}^{\infty}(\dot{B}_{2,1}^{0})} + \|\bar{B}\|_{\tilde{L}_{T}^{\infty}(\dot{B}_{2,1}^{0})} \leq \,C.
\end{align}
Due to just a minor change of the proof in \cite{zhangping}, we can get the following estimation. So, we postpone the proof in the appendix
and give the estimation first
\begin{align}
\label{estimate in paper}
\begin{split}
\|\nabla\bar{u}\|_{\tilde{L}_{t}^{\infty}(\dot{B}_{2,1}^{1/2})} + & \|\nabla\bar{B}\|_{\tilde{L}_{t}^{\infty}(\dot{B}_{2,1}^{1/2})}
+ \|\nabla\bar{u}\|_{L_{t}^{1}(\dot{B}_{2,1}^{5/2})} \\
&\quad +\|\nabla\bar{B}\|_{L_{t}^{1}(\dot{B}_{2,1}^{5/2})} + \|\nabla\bar{\Pi}\|_{L_{t}^{1}(\dot{B}_{2,1}^{3/2})} \leq \,C.
\end{split}
\end{align}
Differentiating the momentum equation and magnetic field equation of (\ref{mhd_a}) with respect to the spatial variables gives
rise to
\begin{align}
\label{d u b}
\begin{split}
\partial_{t}\partial_{i}\bar{u} + \bar{u}\cdot\nabla\partial_{i}\bar{u} - \bar{B}\cdot\nabla\partial_{i}\bar{B}
& - (1+\bar{a})\Delta\partial_{i}\bar{u} + (1+\bar{a})\nabla\partial_{i}\Pi   \\
& = -\partial_{i}\bar{u}\cdot\nabla\bar{u} + \partial_{i}\bar{a}\,\Delta \bar{u} - \partial_{i}\bar{a}\,\nabla\bar{\Pi}
+ \partial_{i}\bar{B}\cdot\nabla\bar{B} \\
& \quad + \partial_{i}\bar{a}\,\bar{B}\cdot\nabla\bar{B} + \bar{a}\,\partial_{i}\bar{B}\cdot\nabla\bar{B}
+ \bar{a}\,\bar{B}\cdot\nabla\partial_{i}\bar{B},  \\
\partial_{t}\partial_{i}\bar{B} - \Delta\partial_{i}\bar{B} + \bar{u}\cdot\nabla\partial_{i}\bar{B} & - \bar{B}\cdot\nabla\partial_{i}\bar{u}
 = \partial_{i}\bar{B}\cdot\nabla\bar{u} - \partial_{i}\bar{u}\cdot\nabla\bar{B}.
\end{split}
\end{align}
Using Remark \ref{s=1 linear couple} and Gronwall's inequality, we will obtain
\begin{align*}
& \|\nabla\bar{u}\|_{\tilde{L}_{T}^{\infty}(\dot{B}_{2,1}^{1})} + \|\nabla\bar{u}\|_{\tilde{L}_{T}^{1}(\dot{B}_{2,1}^{3})}
+ \|\nabla\bar{B}\|_{\tilde{L}_{T}^{\infty}(\dot{B}_{2,1}^{1})} + \|\nabla\bar{B}\|_{\tilde{L}_{T}^{1}(\dot{B}_{2,1}^{3})} \\
\lesssim &\, \exp{\left\{ C\int_{0}^{T} \|\bar{u}\|_{\dot{B}_{2,1}^{5/2}} + \|\bar{B}\|_{\dot{B}_{2,1}^{5/2}}\,dt \right\}}
\bigg\{ \|\nabla \bar{u}_{0}\|_{\dot{B}_{2,1}^{1}} + \|\nabla \bar{B}_{0}\|_{\dot{B}_{2,1}^{1}}  \\
&\, + \|\bar{a}\|_{\tilde{L}_{T}^{\infty}(\dot{B}_{2,1}^{2})}\left( \|\bar{u}\|_{L_{T}^{1}(\dot{B}_{2,1}^{7/2})}
+ \|\Pi\|_{L_{T}^{1}(\dot{B}_{2,1}^{5/2})} + \|\bar{B}\|_{\tilde{L}_{t}^{\infty}(\dot{B}_{2,1}^{3/2})}\|\bar{B}\|_{L_{t}^{1}(\dot{B}_{2,1}^{5/2})} \right) \\
&\, + \|\bar{a}\|_{\tilde{L}_{t}^{\infty}(\dot{B}_{2,1}^{2})}\|\bar{B}\|_{\tilde{L}_{t}^{\infty}(\dot{B}_{2,1}^{3/2})}
\|\bar{B}\|_{L_{t}^{1}(\dot{B}_{2,1}^{5/2})}
+ \|\bar{a}\|_{\tilde{L}_{T}^{\infty}(\dot{B}_{2,1}^{1})}\|\bar{B}\|_{\tilde{L}_{T}^{\infty}(\dot{B}_{2,1}^{3/2})}\|\nabla\bar{B}\|_{L_{T}^{1}(\dot{B}_{2,1}^{5/2})}
\bigg\}.
\end{align*}
Using Proposition \ref{global reference 5/2}, (\ref{estimate a}), (\ref{estimate in paper}) and product laws in Besov space, we get
\begin{align}
\label{estimate contain eta}
\begin{split}
&\, \|\nabla\bar{u}\|_{\tilde{L}_{T}^{\infty}(\dot{B}_{2,1}^{1})} + \|\nabla\bar{u}\|_{\tilde{L}_{T}^{1}(\dot{B}_{2,1}^{3})}
+ \|\nabla\bar{B}\|_{\tilde{L}_{T}^{\infty}(\dot{B}_{2,1}^{1})} + \|\nabla\bar{B}\|_{\tilde{L}_{T}^{1}(\dot{B}_{2,1}^{3})} \\
\leq &\, C + C \left( \|\bar{u}\|_{L_{T}^{1}(\dot{B}_{2,1}^{7/2})} + \|\bar{B}\|_{L_{T}^{1}(\dot{B}_{2,1}^{7/2})} + \|\bar{\Pi}\|_{L_{T}^{1}(\dot{B}_{2,1}^{5/2})} \right)      \\
\leq &\, C + C \left( \|\Delta \bar{u}\|_{L_{T}^{1}(L^{2})} + \|\Delta \bar{B}\|_{L_{T}^{1}(L^{2})} +\|\nabla \bar{\Pi}\|_{L_{T}^{1}(L^{2})} \right) \\
&\, + \eta \bigg( \|\bar{u}\|_{L_{T}^{1}(\dot{B}_{2,1}^{4})}+\|\bar{B}\|_{L_{T}^{1}(\dot{B}_{2,1}^{4})} + \|\bar{\Pi}\|_{L_{T}^{1}(\dot{B}_{2,1}^{3})} \bigg).
\end{split}
\end{align}
Notice that $\mathrm{div} \bar{u} = 0$, we get by taking $\mathrm{div}$ to (\ref{d u b}) that
\begin{align*}
\mathrm{div} \left( (1+\bar{a})\nabla\partial_{i}\bar{\Pi} \right) = & -\mathrm{div} \partial_{i} \left[ (\bar{u}\cdot\nabla)\bar{u} \right]
+ \mathrm{div} \partial_{i} \left[ \bar{a} \, \Delta \bar{u} \right] - \mathrm{div}\left[ \partial_{i}\bar{a}\,\nabla \bar{\Pi} \right] \\
& + \mathrm{div}\partial_{i}\left( \bar{B}\cdot\nabla\bar{B} \right) + \mathrm{div} \left( \partial_{i}\bar{a}\,\bar{B}\cdot\nabla\bar{B} \right)  \\
& + \mathrm{div}\left( \bar{a}\,\partial_{i}\bar{B}\cdot\nabla\bar{B} \right)
+ \mathrm{div}\left( \bar{a}\,\bar{B}\cdot\nabla\partial_{i}\bar{B} \right)
\end{align*}
From this and (\ref{pressure_estimate}), we deduce that
\begin{align*}
\|\nabla^{2}\bar{\Pi}\|_{L_{t}^{1}(\dot{B}_{2,1}^{1})} \lesssim &\, \|\partial_{i}\left( \bar{u}\cdot\nabla\bar{u} \right)\|_{L_{t}^{1}(\dot{B}_{2,1}^{1})}
+ \|\partial_{i}\left( \bar{a}\, \Delta\bar{u} \right)\|_{L_{t}^{1}(\dot{B}_{2,1}^{1})} \\
&\, + \|\partial_{i}\bar{a}\,\nabla\bar{\Pi}\|_{L_{t}^{1}(\dot{B}_{2,1}^{1})}
+ \|\partial_{i}\left( \bar{B}\cdot\nabla\bar{B} \right)\|_{L_{t}^{1}\dot{B}_{2,1}^{1}}  \\
& + \|\partial_{i}\bar{a}\,\bar{B}\cdot\nabla\bar{B}\|_{L_{t}^{1}(\dot{B}_{2,1}^{1})}
+ \|\bar{a}\,\partial_{i}\bar{B}\cdot\nabla\bar{B}\|_{L_{t}^{1}\dot{B}_{2,1}^{1}}   \\
&\, + \|\bar{a}\,\bar{B}\cdot\nabla\partial_{i}\bar{B}\|_{L_{t}^{1}\dot{B}_{2,1}^{1}}
+ \|\bar{a}\|_{\tilde{L}_{t}^{\infty}(\dot{H}^{2})}\|\nabla^{2}\bar{\Pi}\|_{\tilde{L}_{t}^{1}(\dot{H}^{1/2})}.
\end{align*}
Applying the product laws in Besov space yields that for any $\epsilon > 0$
\begin{align}
\label{d d pi}
\begin{split}
&\, \|\nabla^{2}\bar{\Pi}\|_{L_{t}^{1}(\dot{B}_{2,1}^{1})}  \\
\lesssim &\, \|\nabla\bar{u}\|_{L_{t}^{\infty}(\dot{B}_{2,1}^{1})}\|\nabla\bar{u}\|_{L_{t}^{1}(\dot{B}_{2,1}^{3/2})}
+ \|\bar{u}\|_{L_{t}^{\infty}(\dot{B}_{2,1}^{3/2})}\|\nabla^{2}\bar{u}\|_{L_{t}^{1}(\dot{B}_{2,1}^{1})}     \\
&\, + \|\nabla\bar{B}\|_{L_{t}^{\infty}(\dot{B}_{2,1}^{1})}\|\nabla\bar{B}\|_{L_{t}^{1}(\dot{B}_{2,1}^{3/2})}
+ \|\bar{B}\|_{L_{t}^{\infty}(\dot{B}_{2,1}^{3/2})}\|\nabla^{2}\bar{B}\|_{L_{t}^{1}(\dot{B}_{2,1}^{1})}      \\
&\, + \|\nabla\bar{a}\|_{L_{t}^{\infty}(\dot{B}_{2,1}^{3/2})}\|\Delta\bar{u}\|_{L_{t}^{1}(\dot{B}_{2,1}^{1})}
+ \|\bar{a}\|_{L_{t}^{\infty}(\dot{B}_{2,1}^{3/2})}\|\nabla^{3}\bar{u}\|_{L_{t}^{1}(\dot{B}_{2,1}^{1})}     \\
&\, + \|\nabla\bar{a}\|_{\tilde{L}_{t}^{\infty}(\dot{B}_{2,1}^{1})}\|\bar{B}\|_{L_{t}^{\infty}(\dot{B}_{2,1}^{3/2})}\|\nabla\bar{B}\|_{L_{t}^{1}(\dot{B}_{2,1}^{3/2})}
\\
&\, + \|\bar{a}\|_{\tilde{L}_{t}^{\infty}(\dot{B}_{2,1}^{3/2})}\|\nabla\bar{B}\|_{L_{t}^{\infty}(\dot{B}_{2,1}^{1})}\|\nabla\bar{B}\|_{L_{t}^{1}(\dot{B}_{2,1}^{3/2})}
\\
&\, + \|\bar{a}\|_{\tilde{L}_{t}^{\infty}(\dot{B}_{2,1}^{3/2})}\|\bar{B}\|_{L_{t}^{\infty}(\dot{B}_{2,1}^{3/2})}\|\nabla^{2}\bar{B}\|_{L_{t}^{1}(\dot{B}_{2,1}^{1})}
\\
&\, + \|\bar{a}\|_{\tilde{L}_{t}^{\infty}(\dot{B}_{2,1}^{2})}
\left( \epsilon \|\nabla\bar{\Pi}\|_{L_{t}^{1}(\dot{B}_{2,1}^{2})} + C \, \|\nabla\bar{\Pi}\|_{L_{t}^{1}(L^{2})} \right)
\end{split}
\end{align}

Combining (\ref{estimate in paper}), Proposition \ref{global reference 5/2} and taking $\epsilon > 0$ in (\ref{d d pi}) small enough, we will have
\begin{align}
\label{pi}
\begin{split}
\|\nabla\bar{\Pi}\|_{L_{t}^{1}(\dot{B}_{2,1}^{2})} \leq &\, C\,\big( 1 + \|\nabla\bar{u}\|_{L_{t}^{\infty}(\dot{B}_{2,1}^{1})}
+ \|\nabla\bar{B}\|_{L_{t}^{\infty}(\dot{B}_{2,1}^{1})}     \\
&\, + \|\bar{u}\|_{L_{t}^{1}(\dot{B}_{2,1}^{4})} + \|\bar{B}\|_{L_{t}^{1}(\dot{B}_{2,1}^{4})} \big).
\end{split}
\end{align}
Finally, taking $\eta$ in (\ref{estimate contain eta}) small enough and substituting (\ref{pi}) into (\ref{estimate contain eta}), we get
\begin{align*}
\|\nabla\bar{u}\|_{\tilde{L}_{t}^{\infty}(\dot{B}_{2,1}^{1})}+\|\nabla\bar{u}\|_{\tilde{L}_{t}^{1}(\dot{B}_{2,1}^{3})}
+\|\nabla\bar{B}\|_{\tilde{L}_{t}^{\infty}(\dot{B}_{2,1}^{1})} + \|\nabla\bar{B}\|_{\tilde{L}_{t}^{1}(\dot{B}_{2,1}^{3})} \leq C.
\end{align*}
This along with $(\ref{lp})$ and $(\ref{l2})$ completes the proof of the proposition.
\end{proof}

%%%%%%%%%%%%%%%%%%%%%%%%%%%%%%%%%%%%%%%%%%%%%%%%%%%%%%%%%%%%%%%%%%%%%%%%%%%%%%%%%%%%%%%%%%%%%%%%%%%%%%%%%%%%%%%%%%%%%%%%%%%%%%%%%%%%%%%%%%%%%%%%%%%%%%%%%%%

\section{Stability of the Global Large Solutions}

In this section we will give the proof of Theorem \ref{stability main theorem}.
Denoting $\tilde{u} := u-\bar{u}$, $\tilde{B} := B-\bar{B}$ and $\tilde{a} := a-\bar{a}$, we have
\begin{align}
\label{perturbed solution}
\begin{split}
\begin{cases}
\partial_{t}\tilde{a} + (\bar{u} + \tilde{u})\nabla\tilde{a} = -\tilde{u}\cdot\nabla\bar{a},    \\
\partial_{t}\tilde{u} + \tilde{u}\cdot\nabla\tilde{u} + \tilde{u}\cdot\nabla\bar{u} + \bar{u}\cdot\nabla\tilde{u}
- (1+\bar{a}+\tilde{a})(\Delta\tilde{u} - \nabla\tilde{\Pi})    \\
\quad\quad\quad\quad - (1+\bar{a}+\tilde{a})(\tilde{B}\cdot\nabla\tilde{B})
- (1+\bar{a}+\tilde{a})(\tilde{B}\cdot\nabla\bar{B} + \bar{B}\cdot\nabla\tilde{B})   \\
\quad\quad\quad\quad\quad\quad\quad\quad\quad\quad\quad\quad\quad\quad\quad\,\,
= \tilde{a}(\Delta\bar{u} - \nabla\bar{\Pi}) + \tilde{a}\,(\bar{B}\cdot\nabla\bar{B}), \\
\partial_{t}\tilde{B} - \Delta\tilde{B} + \tilde{u}\cdot\nabla\tilde{B} - \tilde{B}\cdot\nabla\tilde{u} = \tilde{B}\cdot\nabla\bar{u}
+ \bar{B}\nabla\tilde{u} - \tilde{u}\nabla\bar{B} - \bar{u}\cdot\nabla\tilde{B}, \\
\mathrm{div}\,\tilde{u} = \mathrm{div}\,\tilde{B} = 0,  \\
(\tilde{a}, \tilde{u}, \tilde{B})|_{t=0} = (\tilde{a}_{0}, \tilde{u}_{0}, \tilde{B}_{0}).
\end{cases}
\end{split}
\end{align}

Then the proof of Theorem \ref{stability main theorem} is equivalent to the proof of the global well posedness of (\ref{perturbed solution})
with small enough initial data $(\tilde{a}_{0}, \tilde{u}_{0})$. Indeed, according to the coupled parabolic hyperbolic theory \cite{R. Danchin 2004},
it is standard
to prove that there exists a positive time $\tilde{T}^{*}$ such that (\ref{mhd_a}) with initial data
$(\bar{a}_{0}+\tilde{a}_{0}, \bar{u}_{0}+\tilde{u}_{0}, \bar{B}_{0}+\tilde{B}_{0})$ has a unique solution $(a,u,B)$ with
\begin{align*}
& a \in C([0, \tilde{T}^{*}); B_{2,1}^{7/2}(\mathbb{R}^{3})), \\
& u \in C([0, \tilde{T}^{*}); B_{2,1}^{2}) \cap L_{\mathrm{loc}}^{1}((0,\tilde{T}^{*}); \dot{B}_{2,1}^{4}(\mathbb{R}^{3})),  \\
& B \in C([0, \tilde{T}^{*}); B_{2,1}^{2}) \cap L_{\mathrm{loc}}^{1}((0,\tilde{T}^{*}); \dot{B}_{2,1}^{4}(\mathbb{R}^{3})).
\end{align*}
Then $(\tilde{a}, \tilde{u}, \tilde{B})$ with
\begin{align*}
& \tilde{a} \in C([0, \tilde{T}^{*}); B_{2,1}^{7/2}(\mathbb{R}^{3})), \\
& \tilde{u} \in C([0, \tilde{T}^{*}); B_{2,1}^{2}) \cap L_{\mathrm{loc}}^{1}((0,\tilde{T}^{*}); \dot{B}_{2,1}^{4}(\mathbb{R}^{3})),  \\
& \tilde{B} \in C([0, \tilde{T}^{*}); B_{2,1}^{2}) \cap L_{\mathrm{loc}}^{1}((0,\tilde{T}^{*}); \dot{B}_{2,1}^{4}(\mathbb{R}^{3})).
\end{align*}
solves (\ref{perturbed solution}) on $[0, \tilde{T}^{*})$. Without loss of generality, we may assume that $\tilde{T}^{*}$ is the maximal time of
the existence to this solution. The aim of what follows is to prove that $\tilde{T}^{*} = \infty$ and $(\tilde{a}, \tilde{u}, \tilde{B})$ remains small
for all $t > 0$.
\begin{lemma}\label{lemma 7.1}
Let
\begin{align*}
V(t) = 2 \int_{0}^{t}\left( \|\nabla\bar{u}(\tau)\|_{L^{\infty}} + \|\nabla\bar{B}(\tau)\|_{L^{\infty}} \right)\,d\tau.
\end{align*}
Then under the assumption of Theorem \ref{stability main theorem}, we have
\begin{align}
\label{7.2}
\begin{split}
&\, \frac{d}{dt}\left[ e^{-V(t)}\left( \|(\sqrt{\rho}\tilde{u}(t)\|_{L^{2}}^{2} + \|\tilde{B}(t)\|_{L^{2}}^{2} \right) \right] \\
&\, \quad\quad\quad\quad\quad\quad\quad\quad
+ c_{0}g^{2}(t)e^{-V(t)}\left( \|\sqrt{\rho}\tilde{u}(t)\|_{L^{2}}^{2} + \|\tilde{B}(t)\|_{L^{2}}^{2} \right)   \\
\leq &\, Ce^{-V(t)}\bigg\{  g^{2}(t)\int_{S_{1}(t)}e^{-2t|\xi|^{2}}|\hat{\tilde{u}}_{0}(\xi)|^{2}\,d\xi
+ g^{2}(t)\int_{S_{2}(t)}e^{-2t|\xi|^{2}}|\hat{\tilde{B}}_{0}(\xi)|^{2}\,d\xi   \\
&\, + \|\tilde{\rho}\|_{L^{3}}^{2}\|\Delta\bar{u} - \nabla\bar{\Pi}\|_{L^{2}}^{2} + \|\tilde{\rho}\|_{L^{3}}^{2}\|\bar{B}\cdot\nabla\bar{B}\|_{L^{2}}^{2} \\
&\, + g(t)^{7}\left( \|\tilde{u}\|_{L_{t}^{2}(L^{2})}^{2} + \|\tilde{B}\|_{L_{t}^{2}(L^{2})}^{2} \right)
\left( \|\bar{u}\|_{L_{t}^{2}(L^{2})}^{2} + \|\bar{B}\|_{L_{t}^{2}(L^{2})}^{2} \right)  \\
&\, + g(t)^{7}\left( \|\tilde{B}\|_{L_{t}^{2}(L^{2})}^{2} + \|\tilde{u}\|_{L_{t}^{2}(L^{2})}^{2} \right)^{2}    \\
&\, + g(t)^{5}\|\tilde{B}\|_{L_{t}^{2}(L^{2})}^{2}\left( \|\nabla\bar{B}\|_{L_{t}^{2}(L^{2})}^{2} + \|\nabla\tilde{B}\|_{L_{t}^{2}(L^{2})}^{2} \right)  \\
&\, + g(t)^{5}\left( \|\Delta\tilde{u}\|_{L_{t}^{1}(L^{2})}^{2} + \|\nabla\tilde{\Pi}\|_{L_{t}^{1}(L^{2})}^{2}
+ \|\tilde{\rho}\|_{L_{t}^{\infty}(L^{2})}^{2}\|\Delta\bar{u} - \nabla\bar{\Pi}\|_{L_{t}^{1}(L^{2})}^{2} \right)  \\
&\, + g(t)^{5} \|\tilde{\rho}\|_{L_{t}^{\infty}(L^{2})}^{2} \|\bar{B}\cdot\nabla\bar{B}\|_{L_{t}^{1}(L^{2})}^{2} \bigg\},
\end{split}
\end{align}
for $t < \tilde{T}^{*}$, where the time dependent phase space region $S_{1}(t)$, $S_{2}(t)$ is given as in the proof of
Proposition \ref{decay pro 3}. Here and in what follows, we shall always denote
\begin{align*}
\rho := \frac{1}{1+\bar{a}+\tilde{a}}, \quad \bar{\rho} := \frac{1}{1+\bar{a}}, \quad \tilde{\rho}:=\rho - \bar{\rho}.
\end{align*}
\end{lemma}
\begin{proof}
Thanks to (\ref{perturbed solution}), $(\rho, \tilde{u}, \tilde{B})$ solves
\begin{align}
\label{perturbed}
\begin{cases}
\partial_{t}\rho + \mathrm{div} \, (\rho u) = 0, \\
\rho \partial_{t}\tilde{u} + \rho u\cdot\nabla\tilde{u} - B\cdot\nabla\tilde{B} + \rho \tilde{u}\cdot\nabla\bar{u} - \tilde{B}\cdot\nabla\bar{B}
-\Delta \tilde{u} + \nabla\tilde{\Pi}   \\
\quad\quad\quad\quad\quad\quad\quad\quad\quad\quad\quad\quad\quad\quad\quad\quad\quad
= -\frac{\tilde{\rho}}{\bar{\rho}}(\Delta\bar{u} - \nabla\bar{\Pi}) - \frac{\tilde{\rho}}{\bar{\rho}}\bar{B}\cdot\nabla\bar{B}, \\
\partial_{t}\tilde{B} - \Delta\tilde{B} + u\cdot\nabla\tilde{B} - B\cdot\nabla\tilde{u} = \tilde{B}\cdot\nabla\bar{u} - \tilde{u}\cdot\nabla\bar{B},  \\
\mathrm{div}\,u = \mathrm{div}\,B = 0,
\end{cases}
\end{align}
from which we get by a standard energy estimate that
\begin{align*}
&\, \frac{1}{2}\frac{d}{dt}\left( \|\sqrt{\rho}\tilde{u}(t)\|_{L^{2}}^{2} + \|\tilde{B}(t)\|_{L^{2}}^{2} \right)
+ \left( \|\nabla\tilde{u}(t)\|_{L^{2}}^{2} + \|\nabla\tilde{B}(t)\|_{L^{2}}^{2} \right)    \\
= &\, -\int_{\mathbb{R}^{3}}\rho \,\tilde{u}\cdot\nabla\bar{u}\,\tilde{u}\,dx + \int_{\mathbb{R}^{3}}\tilde{B}\cdot\nabla\bar{u}\,\tilde{B}\,dx
- \int_{\mathbb{R}^{3}}\frac{\tilde{\rho}}{\bar{\rho}}(\Delta\bar{u} - \nabla\bar{\Pi})\,\tilde{u}\,dx  \\
&\, - \int_{\mathbb{R}^{3}}\frac{\tilde{\rho}}{\bar{\rho}}\,\bar{B}\cdot\nabla\bar{B}\,\tilde{u}\,dx    \\
\leq &\, \|\nabla\bar{u}\|_{L^{\infty}}\int_{\mathbb{R}^{3}}\rho |\tilde{u}|^{2} + |\tilde{B}|^{2} \,dx
+ C \|\tilde{\rho}\|_{L^{3}} \|\Delta\bar{u}-\nabla\bar{\Pi}\|_{L^{2}}\|\nabla\tilde{u}\|_{L^{2}}   \\
&\, + C \|\tilde{\rho}\|_{L^{3}}\|\bar{B}\cdot\nabla\bar{B}\|_{L^{2}}\|\nabla\tilde{u}\|_{L^{2}}.
\end{align*}
This gives
\begin{align}
\label{7.4}
\begin{split}
&\, \frac{d}{dt}\left(e^{-V(t)}\left( \|\sqrt{\rho}\tilde{u}(t)\|_{L^{2}}^{2} + \|\tilde{B}(t)\|_{L^{2}}^{2} \right)\right)
+ e^{-V(t)}\left( \|\nabla\tilde{u}(t)\|_{L^{2}}^{2} + \|\nabla\tilde{B}(t)\|_{L^{2}}^{2} \right)   \\
\leq &\, Ce^{-V(t)}\|\tilde{\rho}\|_{L^{3}}^{2}\|\Delta\bar{u}-\nabla\bar{\Pi}\|_{L^{2}}^{2}
+ C e^{-V(t)}\|\tilde{\rho}\|_{L^{3}}^{2}\|\bar{B}\cdot\nabla\bar{B}\|_{L^{2}}^{2},
\end{split}
\end{align}
which along with a similar derivation of (\ref{decay 3 step 1 main}) ensures
\begin{align}
\begin{split}
&\, \frac{d}{dt}\left(e^{-V(t)}\left( \|\sqrt{\rho}\tilde{u}(t)\|_{L^{2}}^{2} + \|\tilde{B}(t)\|_{L^{2}}^{2} \right)\right) \\
&\, \quad\quad\quad\quad\quad\quad\quad\quad\quad
+ c_{0}g^{2}(t)e^{-V(t)}\left( \|\sqrt{\rho}\tilde{u}(t)\|_{L^{2}}^{2} + \|\tilde{B}(t)\|_{L^{2}}^{2} \right)   \\
\leq &\, Ce^{-V(t)}\|\tilde{\rho}\|_{L^{3}}^{2}\left( \|\Delta\bar{u} - \nabla\bar{\Pi}\|_{L^{2}}^{2} + \|\bar{B}\cdot\nabla\bar{B}\|_{L^{2}}^{2} \right)   \\
&\, \quad\quad\quad\quad\quad
+ C e^{-V(t)}g^{2}(t)\left( \int_{S_{1}(t)} |\hat{\tilde{u}}(\xi)|^{2}\,d\xi + \int_{S_{2}(t)}|\hat{\tilde{B}}(\xi)|^{2}\,d\xi \right),
\end{split}
\end{align}
with $c_{0} \leq \min{(\frac{1}{\rho}, 1)}$ and the time dependent phase space region $S_{1}(t)$, $S_{2}(t)$
being the same as the one in (\ref{decay 3 step 1 main}).

Now, we need to give the estimate of $\int_{S_{1}(t)}|\hat{\tilde{u}}(\xi)|^{2}\,d\xi$ and $\int_{S_{2}(t)}|\hat{\tilde{B}}(\xi)|^{2}\,d\xi$.
Rewrite the second equation in (\ref{perturbed}) as
\begin{align*}
\tilde{u}(t) = &\, e^{t\Delta}\tilde{u}_{0} + \int_{0}^{t}e^{(t-\tau)\Delta}\mathbb{P}
\bigg( -\nabla\cdot(\tilde{u}\otimes\tilde{u} + \bar{u}\otimes\tilde{u} + \tilde{u}\otimes\bar{u})  \\
&\, +\left( \frac{1}{\rho}-1 \right)\Delta\tilde{u} - \left( \frac{1}{\rho}-1 \right)\nabla\tilde{\Pi}
- \frac{\tilde{\rho}}{\rho\bar{\rho}} (\Delta\bar{u}-\nabla\bar{\Pi})   \\
&\, + \frac{1}{\rho}B\cdot\nabla\tilde{B} + \frac{1}{\rho}\tilde{B}\cdot\nabla\bar{B} - \frac{\tilde{\rho}}{\rho\bar{\rho}}\, \bar{B}\cdot\nabla\bar{B}
\bigg)\,d\tau.
\end{align*}
Taking the Fourier transformation with respect to the $x$ variables and integrating the resulting equation over $S_{1}(t)$, we obtain
\begin{align}
\label{7.6}
\begin{split}
&\, \int_{S_{1}(t)}|\hat{\tilde{u}}(\xi)|^{2}\,d\xi     \\
\lesssim &\, \int_{S_{1}(t)} e^{-2t|\xi|^{2}}|\hat{\tilde{u}}_{0}(\xi)|^{2}\,d\xi
+ g(t)^{5}\bigg( \int_{0}^{t}\|\mathcal{F}(\tilde{u}\otimes\tilde{u})\|_{L_{\xi}^{\infty}}  \\
&\, + \|\mathcal{F}(\bar{u}\otimes\tilde{u})\|_{L_{\xi}^{\infty}}\,d\tau \bigg)^{2}
+ g(t)^{3}\bigg( \int_{0}^{t} \left\|\mathcal{F}\left(\left(\frac{1}{\rho}-1\right)\Delta\tilde{u}\right)\right\|_{L_{\xi}^{\infty}} \\
&\, + \left\|\mathcal{F}\left(\left(\frac{1}{\rho}-1\right)\Delta\tilde{\Pi}\right)\right\|_{L_{\xi}^{\infty}}
+ \left\|\mathcal{F}\left(\frac{\tilde{\rho}}{\rho\bar{\rho}}(\Delta\bar{u}-\nabla\bar{\Pi})\right)\right\|_{L_{\xi}^{\infty}}\,d\tau \bigg)^{2}    \\
&\, + g(t)^{3}\bigg( \int_{0}^{t}\left\|\mathcal{F}\left(\frac{1}{\rho}\nabla\cdot\left(\tilde{B}\otimes\tilde{B}\right)\right)\right\|_{L_{\xi}^{\infty}}
+ \left\|\mathcal{F}\left(\frac{1}{\rho}\,\nabla\cdot\left(\bar{B}\otimes\tilde{B}\right)\right)\right\|_{L_{\xi}^{\infty}} \\
&\, + \left\|\mathcal{F}\left( \frac{\tilde{\rho}}{\rho\bar{\rho}}\,\bar{B}\cdot\nabla\bar{B} \right)\right\|_{L_{\xi}^{\infty}}
\,d\tau \bigg)^{2}.
\end{split}
\end{align}
Notice that
\begin{align*}
\int_{0}^{t}\|\mathcal{F}(\tilde{u}\otimes\tilde{u})(\tau)\|_{L_{\xi}^{\infty}}\,d\tau \leq C\,\int_{0}^{t}\|\tilde{u}(\tau)\|_{L^{2}}^{2}\,d\tau,
\end{align*}
\begin{align*}
\int_{0}^{t}\|\mathcal{F}(\tilde{u}\otimes\bar{u})(\tau)\|_{L_{\xi}^{\infty}}\,d\tau \leq C\,
\left( \int_{0}^{t}\|\tilde{u}(\tau)\|_{L^{2}}^{2} \right)^{\frac{1}{2}}\left( \int_{0}^{t}\|\bar{u}(\tau)\|_{L^{2}}^{2} \right)^{\frac{1}{2}},
\end{align*}
\begin{align*}
\int_{0}^{t}\left\|\mathcal{F}\left(\left(\frac{1}{\rho}-1\right)\Delta\tilde{u}\right)(\tau)\right\|_{L_{\xi}^{\infty}}\,d\tau
\leq C \int_{0}^{t}\|\Delta\tilde{u}\|_{L^{2}}\,d\tau,
\end{align*}
\begin{align*}
\int_{0}^{t}\left\|\mathcal{F}\left(\left(\frac{1}{\rho}-1\right)\nabla\tilde{\Pi}\right)\right\|_{L_{\xi}^{\infty}}
\leq C \int_{0}^{t}\|\nabla\tilde{\Pi}\|_{L^{2}}\,d\tau,
\end{align*}
\begin{align*}
\int_{0}^{t}\left\|\mathcal{F}\left(\frac{\tilde{\rho}}{\bar{\rho}\rho}\left(\Delta\bar{u}-\nabla\bar{\Pi}\right)\right)\right\|_{L_{\xi}^{\infty}}\,d\tau
\leq C \|\tilde{\rho}\|_{L_{t}^{\infty}(L^{2})}\int_{0}^{t}\|\Delta\bar{u}-\nabla\bar{\Pi}\|_{L^{2}}\,d\tau,
\end{align*}
\begin{align*}
\int_{0}^{t}\left\|\mathcal{F}\left(\frac{1}{\rho}\, \nabla\left(\tilde{B}\otimes\tilde{B}\right)\right)\right\|_{L_{\xi}^{\infty}}\,d\tau
\leq &\, C \int_{0}^{t}\|\frac{1}{\rho}\,\nabla\cdot\left(\tilde{B}\otimes\tilde{B}\right)\|_{L^{1}}\,d\tau \\
\leq &\, C \left(\int_{0}^{t}\|\tilde{B}(\tau)\|_{L^{2}}^{2}\,d\tau\right)^{\frac{1}{2}}
\left(\int_{0}^{t}\|\nabla\tilde{B}(\tau)\|_{L^{2}}^{2}\,d\tau\right)^{\frac{1}{2}},
\end{align*}
\begin{align*}
\int_{0}^{t}\left\|\mathcal{F}\left(\frac{1}{\rho}\, \nabla\left(\tilde{B}\otimes\bar{B}\right)\right)\right\|_{L_{\xi}^{\infty}}\,d\tau
\leq C \left(\int_{0}^{t}\|\tilde{B}(\tau)\|_{L^{2}}^{2}\,d\tau\right)^{\frac{1}{2}}
\left(\int_{0}^{t}\|\nabla\bar{B}(\tau)\|_{L^{2}}^{2}\,d\tau\right)^{\frac{1}{2}},
\end{align*}
\begin{align*}
\int_{0}^{t}\left\|\frac{\tilde{\rho}}{\rho\bar{\rho}}\,\bar{B}\,\nabla\bar{B}\right\|_{L^{1}}\,d\tau
\leq C\,\|\tilde{\rho}\|_{L_{t}^{\infty}(L^{2})}\int_{0}^{t}\|\bar{B}\cdot\nabla\bar{B}\|_{L^{2}}\,d\tau.
\end{align*}
Then, we rewrite the third equation in (\ref{perturbed}) as
\begin{align*}
\tilde{B}(t) = &\, e^{t\Delta}\tilde{B}_{0} + \int_{0}^{t}e^{(t-\tau)\Delta} \bigg[\nabla\cdot\Big(\tilde{B}\otimes\bar{u}
+ \bar{B}\otimes\tilde{u} - \tilde{u}\otimes\bar{B}     \\
&\, - \bar{u}\otimes\tilde{B} + \tilde{B}\otimes\tilde{u} - \tilde{u}\otimes\tilde{B}
\Big)\bigg]\,d\tau.
\end{align*}
Taking the Fourier transformation with respect to the $x$ variables and integrating the resulting equation over $S(t)$, we obtain
\begin{align}
\label{7.6'}
\begin{split}
\int_{S_{2}(t)}|\hat{\tilde{B}}(\xi)|^{2}\,d\xi \lesssim &\,\int_{S_{2}}e^{-2t|\xi|^{2}}|\hat{\tilde{B}}_{0}|^{2}\,d\xi \\
&\, + g(t)^{5}\bigg( \int_{0}^{t}\|\mathcal{F}(\tilde{B}\otimes\bar{u})\|_{L_{\xi}^{\infty}}
+ \|\mathcal{F}(\bar{B}\otimes\tilde{u})\|_{L_{\xi}^{\infty}}   \\
&\, + \|\mathcal{F}(\tilde{u}\otimes\bar{B})\|_{L_{\xi}^{\infty}} + \|\mathcal{F}(\bar{u}\otimes\tilde{B})\|_{L_{\xi}^{\infty}} \\
&\, + \|\mathcal{F}(\tilde{B}\otimes\tilde{u})\|_{L_{\xi}^{\infty}} + \|\mathcal{F}(\tilde{u}\otimes\tilde{B})\|_{L_{\xi}^{\infty}}
\bigg)^{2}.
\end{split}
\end{align}
Similar to the estimates used above, we have
\begin{align}
\begin{split}
\int_{S_{2}(t)}|\hat{\tilde{B}}(\xi)|^{2}\,d\xi & \lesssim \, \int_{S_{2}(t)}e^{-2t|\xi|^{2}}|\hat{\tilde{B}}_{0}(\xi)|^{2}\,d\xi   \\
&\, + g(t)^{5}\left(\|\tilde{u}\|_{L_{t}^{2}(L^{2})}^{2} + \|\tilde{B}\|_{L_{t}^{2}(L^{2})}^{2}\right)
\left(\|\bar{u}\|_{L_{t}^{2}(L^{2})}^{2} + \|\bar{B}\|_{L_{t}^{2}(L^{2})}^{2}\right)    \\
&\, + g(t)^{5}\left(\|\tilde{B}\|_{L_{t}^{2}(L^{2})}^{2} + \|\tilde{u}\|_{L_{t}^{2}(L^{2})}^{2}\right)^{2}
\end{split}
\end{align}
Combining all the above estimates, we complete the proof.
\end{proof}

\begin{lemma}\label{lemma 7.2}
Let
\begin{align*}
U(t) := C \int_{0}^{t} \Big( \|\Delta\bar{u}(\tau)\|_{L^{2}}\|\nabla\bar{u}(\tau)\|_{L^{2}}
& + \|\bar{u}(\tau)\|_{L^{\infty}}^{2}    \\
& + \|\Delta\bar{B}(\tau)\|_{L^{2}}\|\nabla\bar{B}(\tau)\|_{L^{2}}
+ \|\bar{B}(\tau)\|_{L^{\infty}}^{2} \Big)\,d\tau.
\end{align*}
If
\begin{align}
\label{7.7}
\sup_{t\in [0,\bar{T})} \left[ \|\tilde{u}(t)\|_{L^{2}}\|\nabla\tilde{u}(t)\|_{L^{2}} + \|\tilde{B}(t)\|_{L^{2}}\|\nabla\tilde{B}(t)\|_{L^{2}} \right]
\leq \nu
\end{align}
for some $\bar{T} \leq \tilde{T}^{*}$ and some sufficiently small positive constant $\nu$, then under the assumptions of Theorem \ref{stability main theorem},
we have
\begin{align}
\label{7.8}
\begin{split}
&\, \frac{d}{dt}\left(e^{-U(t)}\|(\nabla\tilde{u}, \nabla\tilde{B})\|_{L^{2}}^{2}\right)
+ e^{-U(t)}\|(\sqrt{\rho}\partial_{t}\tilde{u}, \partial_{t}\tilde{B})\|_{L^{2}}^{2}    \\
&\quad\quad\quad\quad\quad\quad\quad\quad\quad\quad\quad\quad\quad\quad\quad\quad\quad
+ c_{0}e^{-U(t)}\|(\Delta\tilde{u}, \Delta\tilde{B})\|_{L^{2}}^{2}     \\
\leq &\,C\,e^{-U(t)}\left(\|\tilde{\rho}\|_{L^{\infty}}^{2}\|\Delta\bar{u}-\nabla\bar{\Pi}\|_{L^{2}}^{2}
+ \|\tilde{\rho}\|_{L^{\infty}}^{2}\|\bar{B}\cdot\nabla\bar{B}\|_{L^{2}}^{2}\right) \quad\quad \text{for }  t < \bar{T}.
\end{split}
\end{align}
\end{lemma}
\begin{proof}
Taking the $L^{2}$ inner product between the second equation of (\ref{perturbed}) and $\partial_{t}\tilde{u}$, we obtain
\begin{align*}
&\, \frac{1}{2}\frac{d}{dt}\|\nabla\tilde{u}\|_{L^{2}}^{2} + \|\sqrt{\rho}\partial_{t}\tilde{u}\|_{L^{2}}^{2}   \\
\lesssim &\, \|\sqrt{\rho}\partial_{t}\tilde{u}\|_{L^{2}}\|\sqrt{\rho}\|_{L^{\infty}}
\left(\|\tilde{u}\cdot\nabla\tilde{u}\|_{L^{2}} + \|\tilde{u}\cdot\nabla\bar{u}\|_{L^{2}} + \|\bar{u}\cdot\nabla\tilde{u}\|_{L^{2}}\right)  \\
&\, + \|\sqrt{\rho}\partial_{t}\tilde{u}\|_{L^{2}}\left\|\frac{\tilde{\rho}}{\bar{\rho}\sqrt{\rho}}\right\|_{L^{\infty}}\|\Delta\bar{u}-\nabla\bar{\Pi}\|_{L^{2}}
+ \left\|\frac{\tilde{\rho}}{\bar{\rho}}\right\|_{L^{\infty}}^{2}\|\bar{B}\cdot\nabla\bar{B}\|_{L^{2}}^{2}  \\
&\, + \|\tilde{B}\cdot\nabla\tilde{B}\|_{L^{2}}^{2} + \|\tilde{B}\cdot\nabla\bar{B}\|_{L^{2}}^{2} + \|\bar{B}\cdot\nabla\tilde{B}\|_{L^{2}}^{2}.
\end{align*}
Notice that
\begin{align}
\label{7.9}
\begin{split}
&\, \|\tilde{u}\cdot\nabla\tilde{u}\|_{L^{2}} + \|\tilde{u}\cdot\nabla\bar{u}\|_{L^{2}} + \|\bar{u}\cdot\nabla\tilde{u}\|_{L^{2}}   \\
\lesssim &\, \|\tilde{u}\|_{L^{2}}^{\frac{1}{2}}\|\nabla\tilde{u}\|_{L^{2}}^{\frac{1}{2}}\|\Delta\tilde{u}\|_{L^{2}}
+ \|\nabla\tilde{u}\|_{L^{2}}\left(\|\nabla\bar{u}\|_{L^{2}}^{\frac{1}{2}}\|\Delta\bar{u}\|_{L^{2}}^{\frac{1}{2}}+\|\bar{u}\|_{L^{\infty}}\right), \\
&\, \|\tilde{B}\cdot\nabla\tilde{B}\|_{L^{2}} + \|\tilde{B}\cdot\nabla\bar{B}\|_{L^{2}} + \|\bar{B}\cdot\nabla\tilde{B}\|_{L^{2}}   \\
\lesssim &\, \|\tilde{B}\|_{L^{2}}^{\frac{1}{2}}\|\nabla\tilde{B}\|_{L^{2}}^{\frac{1}{2}}\|\Delta\tilde{B}\|_{L^{2}}
+ \|\nabla\tilde{B}\|_{L^{2}}\left(\|\nabla\bar{B}\|_{L^{2}}^{\frac{1}{2}}\|\Delta\bar{B}\|_{L^{2}}^{\frac{1}{2}}+\|\bar{B}\|_{L^{\infty}}\right),
\end{split}
\end{align}
which ensures
\begin{align}
\label{7.10}
\begin{split}
&\, \frac{d}{dt}\|\nabla\tilde{u}\|_{L^{2}}^{2} + \|\sqrt{\rho}\partial_{t}\tilde{u}\|_{L^{2}}^{2}  \\
\lesssim &\, \|\tilde{u}\|_{L^{2}}\|\nabla\tilde{u}\|_{L^{2}}\|\Delta\tilde{u}\|_{L^{2}}^{2}
+ \|\nabla\tilde{u}\|_{L^{2}}^{2} \Big(\|\Delta\bar{u}\|_{L^{2}}\|\nabla\bar{u}\|_{L^{2}}+\|\bar{u}\|_{L^{\infty}}^{2}\Big) \\
&\, \|\tilde{B}\|_{L^{2}}\|\nabla\tilde{B}\|_{L^{2}}\|\Delta\tilde{B}\|_{L^{2}}^{2}
+ \|\nabla\tilde{B}\|_{L^{2}}^{2} \Big(\|\Delta\bar{B}\|_{L^{2}}\|\nabla\bar{B}\|_{L^{2}}+\|\bar{B}\|_{L^{\infty}}^{2}\Big) \\
&\, + \|\tilde{\rho}\|_{L^{\infty}}^{2}\|\Delta\bar{u}-\nabla\bar{\Pi}\|_{L^{2}}^{2}
+ \|\tilde{\rho}\|_{L^{\infty}}^{2}\|\bar{B}\cdot\nabla\bar{B}\|_{L^{2}}^{2}.
\end{split}
\end{align}
By taking the $L^{2}$ inner product between (\ref{perturbed}) and $\Delta\tilde{u}$, we get
\begin{align*}
&\, \frac{1}{2}\frac{d}{dt}\|\nabla\tilde{u}\|_{L^{2}}^{2} + \|\rho\|_{L^{\infty}}^{-1}\|\Delta\tilde{u}\|_{L^{2}}^{2} \\
\leq &\, \|\Delta\tilde{u}\|_{L^{2}}\Big( \|\tilde{u}\cdot\nabla\tilde{u}\|_{L^{2}} + \|\tilde{u}\cdot\nabla\bar{u}\|_{L^{2}}
+ \|\bar{u}\cdot\nabla\tilde{u}\|_{L^{2}}      \\
&\, + \|\rho^{-1}\|_{L^{\infty}}(\|\tilde{B}\cdot\nabla\tilde{B}\|_{L^{2}} + \|\bar{B}\cdot\nabla\tilde{B}\|_{L^{2}}
+ \|\tilde{B}\cdot\nabla\bar{B}\|_{L^{2}} ) \\
&\, \left\|\frac{\tilde{\rho}}{\rho\bar{\rho}}\right\|_{L^{\infty}}\|\Delta\bar{u}-\nabla\bar{\Pi}\|_{L^{2}}
+ \|\rho^{-1}\|_{L^{\infty}}\|\nabla\bar{\Pi}\|_{L^{2}} + \left\|\frac{\tilde{\rho}}{\rho\bar{\rho}}\right\|_{L^{\infty}}
\|\bar{B}\cdot\nabla\bar{B}\|_{L^{2}} \Big).
\end{align*}
Thanks to the momentum equation of (\ref{perturbed}) and $\mathrm{div}\,\tilde{u} = 0$, we have
\begin{align}
\label{7.11}
\begin{split}
\|\nabla\tilde{\Pi}\|_{L^{2}} + \|\Delta\tilde{u}\|_{L^{2}}
\lesssim &\,
\|\tilde{u}\cdot\nabla\tilde{u}\|_{L^{2}} + \|\tilde{u}\cdot\nabla\bar{u}\|_{L^{2}} + \|\bar{u}\cdot\nabla\tilde{u}\|_{L^{2}}   \\
&\, + \|\tilde{B}\cdot\nabla\tilde{B}\|_{L^{2}} + \|\tilde{B}\cdot\nabla\bar{B}\|_{L^{2}} + \|\bar{B}\cdot\nabla\tilde{B}\|_{L^{2}}   \\
&\, + \|\tilde{\rho}\|_{L^{\infty}}\|\Delta\bar{u}-\nabla\bar{\Pi}\|_{L^{2}} + \|\tilde{\rho}\|_{L^{\infty}}\|\bar{B}\cdot\nabla\bar{B}\|_{L^{2}} \\
&\, + \|\sqrt{\rho}\partial_{t}\tilde{u}\|_{L^{2}}.
\end{split}
\end{align}
This gives
\begin{align}
\label{7.12}
\begin{split}
&\, \frac{d}{dt}\|\nabla\tilde{u}\|_{L^{2}}^{2} + c_{0}\|\Delta\tilde{u}\|_{L^{2}}^{2}  \\
\lesssim &\, \|\tilde{u}\|_{L^{2}}\|\nabla\tilde{u}\|_{L^{2}}\|\Delta\tilde{u}\|_{L^{2}}^{2}
+ \|\nabla\tilde{u}\|_{L^{2}}^{2}\Big( \|\nabla\bar{u}\|_{L^{2}}\|\Delta\bar{u}\|_{L^{2}}+\|\bar{u}\|_{L^{\infty}}^{2} \Big)   \\
&\, \|\tilde{B}\|_{L^{2}}\|\nabla\tilde{B}\|_{L^{2}}\|\Delta\tilde{B}\|_{L^{2}}^{2}
+ \|\nabla\tilde{B}\|_{L^{2}}^{2}\Big( \|\nabla\bar{B}\|_{L^{2}}\|\Delta\bar{B}\|_{L^{2}}+\|\bar{B}\|_{L^{\infty}}^{2} \Big)   \\
&\, \|\tilde{\rho}\|_{L^{\infty}}^{2}\|\Delta\bar{u} - \nabla\bar{\Pi}\|_{L^{2}}^{2}
+ \|\sqrt{\rho}\partial_{t}\tilde{u}\|_{L^{2}}^{2} + \|\tilde{\rho}\|_{L^{\infty}}^{2}\|\bar{B}\cdot\nabla\bar{B}\|_{L^{2}}^{2}.
\end{split}
\end{align}
Combining (\ref{7.10}) and (\ref{7.12}), we arrive at
\begin{align}
\label{A}
\begin{split}
&\, \frac{d}{dt}\|\nabla\tilde{u}\|_{L^{2}}^{2} + \|\sqrt{\rho}\partial_{t}\tilde{u}\|_{L^{2}}^{2}
+ \left( c_{0}-C\|\tilde{u}\|_{L^{2}}\|\nabla\tilde{u}\|_{L^{2}} \right)\|\Delta\tilde{u}\|_{L^{2}}^{2} \\
\lesssim &\, \|\nabla\tilde{u}\|_{L^{2}}^{2}(\|\Delta\bar{u}\|_{L^{2}}\|\nabla\bar{u}\|_{L^{2}}+\|\bar{u}\|_{L^{\infty}}^{2})
+ \|\tilde{\rho}\|_{L^{\infty}}^{2}\|\Delta\bar{u}-\nabla\bar{\Pi}\|_{L^{2}}^{2}    \\
&\, + \|\nabla\tilde{B}\|_{L^{2}}^{2}(\|\Delta\bar{B}\|_{L^{2}}\|\nabla\bar{B}\|_{L^{2}}+\|\bar{B}\|_{L^{\infty}}^{2})
+ \|\tilde{\rho}\|_{L^{\infty}}^{2}\|\bar{B}\cdot\nabla\bar{B}\|_{L^{2}}^{2}     \\
&\, + \|\tilde{B}\|_{L^{2}}\|\nabla\tilde{B}\|_{L^{2}}\|\Delta\tilde{B}\|_{L^{2}}^{2}
\end{split}
\end{align}
for $t < \tilde{T}^{*}$.

Taking the $L^{2}$ inner product between the third equation of (\ref{perturbed}) and $\partial_{t}\tilde{B}$, we have
\begin{align*}
&\, \frac{1}{2}\frac{d}{dt}\|\nabla\tilde{B}\|_{L^{2}}^{2} + \|\partial_{t}\tilde{B}\|_{L^{2}}^{2} \\
\lesssim &\, \|\tilde{u}\cdot\nabla\tilde{B}\|_{L^{2}}^{2} + \|\bar{u}\cdot\nabla\tilde{B}\|_{L^{2}}^{2}
+ \|\bar{B}\cdot\nabla\tilde{u}\|_{L^{2}}^{2} + \|\tilde{B}\cdot\nabla\tilde{u}\|_{L^{2}}^{2}   \\
&\, + \|\tilde{B}\cdot\nabla\bar{u}\|_{L^{2}}^{2} + \|\tilde{u}\cdot\nabla\bar{B}\|_{L^{2}}^{2}.
\end{align*}

Noticing that
\begin{align*}
&\, \|\tilde{u}\cdot\nabla\tilde{B}\|_{L^{2}} + \|\bar{u}\cdot\nabla\tilde{B}\|_{L^{2}} + \|\tilde{u}\cdot\nabla\bar{B}\|_{L^{2}}   \\
&\, \quad\quad\quad\quad\quad\quad\quad\quad
+ \|\bar{B}\cdot\nabla\tilde{u}\|_{L^{2}} + \|\tilde{B}\cdot\nabla\tilde{u}\|_{L^{2}} + \|\tilde{B}\cdot\nabla\bar{u}\|_{L^{2}}   \\
\lesssim &\, \|\tilde{u}\|_{L^{2}}^{\frac{1}{2}}\|\nabla\tilde{u}\|_{L^{2}}^{\frac{1}{2}}\|\Delta\tilde{B}\|_{L^{2}}
+ \|\nabla\tilde{u}\|_{L^{2}}\|\nabla\bar{B}\|_{L^{2}}^{\frac{1}{2}}\|\Delta\bar{B}\|_{L^{2}}^{\frac{1}{2}} \\
&\, + \|\tilde{B}\|_{L^{2}}^{\frac{1}{2}}\|\nabla\tilde{B}\|_{L^{2}}^{\frac{1}{2}}\|\Delta\tilde{u}\|_{L^{2}}
+ \|\nabla\tilde{B}\|_{L^{2}}\|\nabla\bar{u}\|_{L^{2}}^{\frac{1}{2}}\|\Delta\bar{u}\|_{L^{2}}^{\frac{1}{2}} \\
&\, + \|\nabla\tilde{B}\|_{L^{2}}\|\bar{u}\|_{L^{\infty}} + \|\nabla\tilde{u}\|_{L^{2}}\|\bar{B}\|_{L^{\infty}}
\end{align*}
which ensures
\begin{align}
\label{7.10'}
\begin{split}
&\, \frac{d}{dt}\|\nabla\tilde{B}\|_{L^{2}}^{2} + \|\partial_{t}\tilde{B}\|_{L^{2}}^{2}   \\
\lesssim &\, \|\tilde{u}\|_{L^{2}}\|\nabla\tilde{u}\|_{L^{2}}\|\Delta\tilde{B}\|_{L^{2}}^{2}
+ \|\nabla\tilde{u}\|_{L^{2}}^{2}\|\nabla\bar{B}\|_{L^{2}}\|\Delta\bar{B}\|_{L^{2}}  \\
&\, \|\tilde{B}\|_{L^{2}}\|\nabla\tilde{B}\|_{L^{2}}\|\Delta\tilde{u}\|_{L^{2}}^{2}
+ \|\nabla\tilde{B}\|_{L^{2}}^{2}\|\nabla\bar{u}\|_{L^{2}}\|\Delta\bar{u}\|_{L^{2}}   \\
&\, + \|\nabla\tilde{B}\|_{L^{2}}^{2}\|\bar{u}\|_{L^{\infty}}^{2} + \|\nabla\tilde{u}\|_{L^{2}}^{2}\|\bar{B}\|_{L^{\infty}}^{2}.
\end{split}
\end{align}
By taking the $L^{2}$ inner product between (\ref{perturbed}) and $\Delta \tilde{B}$, we get
\begin{align*}
\frac{1}{2}\frac{d}{dt}\|\nabla\tilde{B}\|_{L^{2}}^{2} + \|\Delta\tilde{B}\|_{L^{2}}^{2}
\leq &\, \|\Delta\tilde{B}\|_{L^{2}}\Big( \|\tilde{u}\cdot\nabla\tilde{B}\|_{L^{2}} + \|\bar{u}\cdot\nabla\tilde{B}\|_{L^{2}}   \\
&\, + \|\tilde{B}\cdot\nabla\tilde{u}\|_{L^{2}} + \|\bar{B}\cdot\nabla\tilde{u}\|_{L^{2}} \\
&\, + \|\tilde{B}\cdot\nabla\bar{u}\|_{L^{2}} + \|\tilde{u}\cdot\nabla\bar{B}\|_{L^{2}}  \Big).
\end{align*}
Combining this and (\ref{7.10'}), we arrive at
\begin{align}
\label{B}
\begin{split}
&\, \frac{d}{dt}\|\nabla\tilde{B}\|_{L^{2}}^{2} + \|\partial_{t}\tilde{B}\|_{L^{2}}^{2}
+ (1-C\|\tilde{u}\|_{L^{2}}\|\nabla\tilde{u}\|_{L^{2}})\|\Delta\tilde{B}\|_{L^{2}}^{2}  \\
\lesssim &\, \|\nabla\tilde{u}\|_{L^{2}}^{2}\|\nabla\bar{B}\|_{L^{2}}\|\Delta\bar{B}\|_{L^{2}}
+ \|\nabla\tilde{B}\|_{L^{2}}^{2}\|\bar{u}\|_{L^{\infty}}^{2} + \|\nabla\tilde{u}\|_{L^{2}}^{2}\|\bar{B}\|_{L^{\infty}}^{2} \\
&\, + \|\tilde{B}\|_{L^{2}}\|\nabla\tilde{B}\|_{L^{2}}\|\Delta\tilde{u}\|_{L^{2}}^{2}
+ \|\nabla\tilde{B}\|_{L^{2}}^{2}\|\nabla\bar{u}\|_{L^{2}}\|\Delta\bar{u}\|_{L^{2}}.
\end{split}
\end{align}
Hence, by (\ref{A}) and (\ref{B}), we obtain that
\begin{align*}
&\, \frac{d}{dt}\Big( \|\nabla\tilde{u}\|_{L^{2}}^{2} + \|\nabla\tilde{B}\|_{L^{2}}^{2} \Big)
+ \|\sqrt{\rho}\partial_{t}\tilde{u}\|_{L^{2}}^{2} + \|\partial_{t}\tilde{B}\|_{L^{2}}^{2}   \\
&\,\quad\quad + (c_{0} - C\|\tilde{u}\|_{L^{2}}\|\nabla\tilde{u}\|_{L^{2}} - C\|\tilde{B}\|_{L^{2}}\|\nabla\tilde{B}\|_{L^{2}})\|\Delta\tilde{u}\|_{L^{2}}^{2} \\
&\,\quad\quad + (1 - C\|\tilde{u}\|_{L^{2}}\|\nabla\tilde{u}\|_{L^{2}} - C\|\tilde{B}\|_{L^{2}}\|\nabla\tilde{B}\|_{L^{2}})\|\Delta\tilde{B}\|_{L^{2}}^{2}   \\
\lesssim &\, \Big( \|\nabla\tilde{u}\|_{L^{2}}^{2} + \|\nabla\tilde{B}\|_{L^{2}}^{2}\Big)
\Big( \|\Delta\bar{u}\|_{L^{2}}\|\nabla\bar{u}\|_{L^{2}} + \|\Delta\bar{B}\|_{L^{2}}\|\nabla\bar{B}\|_{L^{2}}  \\
&\, + \|\bar{u}\|_{L^{\infty}}^{2} + \|\bar{B}\|_{L^{\infty}}^{2} \Big)
+ \|\tilde{\rho}\|_{L^{\infty}}^{2}\|\Delta\bar{u}-\nabla\bar{\Pi}\|_{L^{2}}^{2}
+ \|\tilde{\rho}\|_{L^{\infty}}^{2}\|\bar{B}\cdot\nabla\bar{B}\|_{L^{2}}^{2}
\end{align*}
for $t < \tilde{T}^{*}$. For small enough $\nu$, we obtain
\begin{align*}
&\, \frac{d}{dt}\Big( \|\nabla\tilde{u}\|_{L^{2}}^{2} + \|\nabla\tilde{B}\|_{L^{2}}^{2} \Big) + \|\sqrt{\rho}\partial_{t}\tilde{u}\|_{L^{2}}^{2}
+ \|\partial_{t}\tilde{B}\|_{L^{2}}^{2} \\
&\, \quad\quad\quad\quad\quad\quad\quad\quad\quad\quad\quad\quad\quad\quad\quad
+ c_{0}\|\Delta\tilde{u}\|_{L^{2}}^{2} + \|\Delta\tilde{B}\|_{L^{2}}^{2}    \\
\lesssim &\, \Big( \|\nabla\tilde{u}\|_{L^{2}}^{2} + \|\nabla\tilde{B}\|_{L^{2}}^{2} \Big)
\Big( \|\Delta\bar{u}\|_{L^{2}}\|\nabla\bar{u}\|_{L^{2}} + \|\Delta\bar{B}\|_{L^{2}}\|\nabla\bar{B}\|_{L^{2}}   \\
&\, + \|\bar{u}\|_{L^{\infty}}^{2} + \|\bar{B}\|_{L^{\infty}}^{2} \Big)
+ \|\tilde{\rho}\|_{L^{\infty}}^{2} \|\Delta\bar{u}-\nabla\bar{\Pi}\|_{L^{2}}^{2}
+ \|\tilde{\rho}\|_{L^{\infty}}^{2} \|\bar{B}\cdot\nabla\bar{B}\|_{L^{2}}^{2}.
\end{align*}
Then simple calculations leads to
\begin{align*}
&\, \frac{d}{dt}\left( e^{-U(t)}\left( \|\nabla\tilde{u}\|_{L^{2}}^{2} + \|\nabla\tilde{B}\|_{L^{2}}^{2} \right)\right)
+ e^{-U(t)}\left( \|\sqrt{\rho}\partial_{t}\tilde{u}\|_{L^{2}}^{2} + \|\partial_{t}\tilde{B}\|_{L^{2}}^{2} \right)  \\
&\, \quad\quad\quad\quad\quad\quad\quad\quad\quad\quad\quad\quad\quad\quad\,\,\,\,\,
+ e^{-U(t)}\left( c_{0}\|\Delta\tilde{u}\|_{L^{2}}^{2} + \|\Delta\tilde{B}\|_{L^{2}}^{2} \right)    \\
\leq &\,C\,e^{-U(t)}\|\tilde{\rho}\|_{L^{\infty}}^{2}\Big( \|\Delta\bar{u} - \nabla\bar{\Pi}\|_{L^{2}}^{2} + \|\bar{B}\cdot\nabla\bar{B}\|_{L^{2}}^{2} \Big)
\end{align*}
for $t < \bar{T}$.
\end{proof}

\begin{proposition}
\label{decay bar}
Under the assumptions in Theorem \ref{stability main theorem}, there exist constants $C$ and $c$ so that if
\begin{align*}
\delta_{0} := \|\tilde{u}_{0}\|_{H^{1}} + \|\tilde{u}_{0}\|_{L^{p}} + \|\tilde{B}_{0}\|_{H^{1}} + \|\tilde{B}_{0}\|_{L^{p}}
+ \|\tilde{\rho}_{0}\|_{L^{2}} + \|\tilde{\rho}_{0}\|_{L^{\infty}} < c,
\end{align*}
then there hold
\begin{align}
\label{7.14}
\begin{split}
&\, \|\tilde{u}(t)\|_{L^{2}} + \|\tilde{B}(t)\|_{L^{2}} \leq C \delta_{0} \langle t \rangle^{-\beta(p)},  \\
&\, \|\nabla\tilde{u}\|_{L^{2}} + \|\nabla\tilde{B}\|_{L^{2}} \leq C \delta_{0} \langle t \rangle^{-\frac{1}{2} - \beta(p)}\quad\text{for }t<\tilde{T}^{*}, \\
&\, \int_{0}^{\tilde{T}^{*}}\|\Delta\tilde{u}\|_{L^{2}} + \|\Delta\tilde{B}\|_{L^{2}} + \|\nabla\tilde{\Pi}\|_{L^{2}}\,dt \leq C\delta_{0},    \\
&\, \int_{0}^{\tilde{T}^{*}}\|\tilde{u}\|_{L^{\infty}}+\|\tilde{B}\|_{L^{\infty}}\,dt \leq C\delta_{0},
\end{split}
\end{align}
and
\begin{align}
\label{7.14'}
\begin{split}
\int_{0}^{\tilde{T}^{*}} \left(\|(\partial_{t}\tilde{u}, \partial_{t}\tilde{B})\|_{L^{2}}^{2} + \|(\Delta\tilde{u}, \Delta\tilde{B})\|_{L^{2}}^{2}
+ \|\nabla\tilde{\Pi}\|_{L^{2}}^{2}\right) \langle t \rangle^{(1+2\beta(p))^{-}} \, d\tau \leq C \delta_{0}
\end{split}
\end{align}
for $\beta(p) = \frac{3}{4}\left( \frac{2}{p} - 1 \right)$.
\end{proposition}
\begin{proof}
Let $\xi(t) := \sup_{t'\in [0,t]} \left( \|\tilde{\rho}(t)\|_{L^{2}} + \|\tilde{\rho}(t)\|_{L^{\infty}} \right)$.
Integrating (\ref{7.4}) and (\ref{7.8}) over $[0,t]$ for $t \leq \bar{T}$ that
\begin{align*}
e^{-V(t)}\left(\|\tilde{u}\|_{L^{2}}^{2} + \|\tilde{B}\|_{L^{2}}^{2}\right)\lesssim &\,\|\tilde{u}_{0}\|_{L^{2}}^{2} + \|\tilde{B}_{0}\|_{L^{2}}^{2}  \\
&\, + \int_{0}^{t}e^{-V(t')}\|\tilde{\rho}\|_{L^{3}}^{2}\|\Delta\bar{u} - \nabla\bar{\Pi}\|_{L^{2}}^{2}\,dt'    \\
&\, + \int_{0}^{t}e^{-V(t')}\|\tilde{\rho}\|_{L^{3}}^{2}\|\bar{B}\cdot\nabla\bar{B}\|_{L^{2}}^{2}\,dt'.
\end{align*}
and
\begin{align*}
&\, e^{-U(t)} \|(\nabla\tilde{u}, \nabla\tilde{B})\|_{L^{2}}^{2} + \int_{0}^{t} e^{-U(t')} \|(\sqrt{\rho}\partial_{t}\tilde{u}, \partial_{t}\tilde{B})\|_{L^{2}}^{2} + c_{0}e^{-U(t')} \|(\Delta\tilde{u}, \Delta\tilde{B})\|_{L^{2}}^{2}\,dt'     \\
\lesssim &\, \|(\nabla\tilde{u}_{0}, \nabla\tilde{B}_{0})\|_{L^{2}}^{2}
+ \int_{0}^{t} e^{-V(t')} \|\Delta\bar{u}-\nabla\bar{\Pi}\|_{L^{2}}^{2}\,dt' + \int_{0}^{t}e^{-V(t')}\|\bar{B}\cdot\nabla\bar{B}\|_{L^{2}}^{2}\,dt',
\end{align*}
where $V(t)$ and $U(t)$ defined as in Lemma \ref{lemma 7.1} and \ref{lemma 7.2} respectively.
From the decay properties of the reference solution, we have
\begin{align}
\label{7.15}
\begin{split}
&\, \|\tilde{u}(t)\|_{L^{2}} + \|\tilde{B}(t)\|_{L^{2}} + \|\nabla\tilde{u}\|_{L_{t}^{2}(L^{2})} + \|\nabla\tilde{B}\|_{L_{t}^{2}(L^{2})}     \\
\lesssim &\, \|(\tilde{u}_{0}, \tilde{B}_{0})\|_{L^{2}} + \sup_{t'\in [0,t]}\|\rho(t')\|_{L^{3}}
\leq \, C \, \left( \delta_{0} + \xi(t) \right),  \\
&\, \|\nabla\tilde{u}(t)\|_{L^{2}} + \|\nabla\tilde{B}(t)\|_{L^{2}} + \|\sqrt{\rho}\partial_{t}\tilde{u}\|_{L_{t}^{2}(L^{2})}
+ \|\partial_{t}\tilde{B}\|_{L_{t}^{2}(L^{2})}  \\
&\,\quad
+ \|\Delta\tilde{u}\|_{L_{t}^{2}(L^{2})} + \|\Delta\tilde{B}\|_{L_{t}^{2}(L^{2})}
\leq \,C \,\left( \delta_{0} + \xi(t) \right).
\end{split}
\end{align}
Now, we use Schonbek's strategy in \cite{schonbek} to prove (\ref{7.14}) and (\ref{7.14'}).

Step 1: Rough decay estimates of $\|\tilde{u}(t)\|_{L^{2}}$, $\|\nabla\tilde{u}(t)\|_{L^{2}}$
and $\|\tilde{B}(t)\|_{L^{2}}$, $\|\nabla\tilde{B}(t)\|_{L^{2}}$.
Let $S_{1}(t) := \left\{ \xi : |\xi|\leq \sqrt{\frac{2}{c_{0}}} g(t) \right\}$,
$S_{2}(t) := \left\{ \xi : |\xi|\leq \sqrt{2} g(t) \right\}$
with $g(t)$ satisfying $g(t) \lesssim \langle t \rangle^{-\frac{1}{2}}$, which will be chosen later on.
Then thanks to Lemma \ref{lemma 7.1}, (\ref{7.15}), (\ref{u decay beta}) and (\ref{b decay beta}), we obtain for $t \leq \bar{T}$
\begin{align*}
&\, \frac{d}{dt}\left[ e^{-V(t)}\left( \|\sqrt{\rho}\tilde{u}\|_{L^{2}}^{2} + \|\tilde{B}\|_{L^{2}}^{2} \right)\right]
+ c_{0}g^{2}(t)e^{-V(t)}\left( \|\sqrt{\rho}\tilde{u}\|_{L^{2}}^{2} + \|\tilde{B}\|_{L^{2}}^{2} \right)     \\
\leq &\, C\, \bigg[ g^{2}(t)\delta_{0}^{2}\langle t \rangle^{-2\beta(p)} + \xi(t)^{2}\|\Delta\bar{u} - \nabla\bar{\Pi}\|_{L^{2}}^{2}
+ \xi(t)^{2} \|\bar{B}\cdot\nabla\bar{B}\|_{L^{2}}^{2}  \\
&\, + g(t)^{5} \left( \delta_{0} + \xi(t) \right)^{2}\langle t \rangle  \bigg]
\end{align*}
provided that $\xi(t) \leq 1$ for $t \leq \bar{T}$, which will be justified later on. In what follows, we shall always assume that
$t \leq \bar{T}$. Then, we have
\begin{align*}
&\, e^{c_{0}\int_{0}^{t}g(t')^{2}\,dt'}e^{-V(t)}\Big( \|\sqrt{\rho}\tilde{u}\|_{L^{2}}^{2} + \|\tilde{B}\|_{L^{2}}^{2} \Big) \\
\lesssim &\, \|(\sqrt{\rho_{0}}\tilde{u}_{0}, \tilde{B}_{0})\|_{L^{2}}^{2}
+ \delta_{0}^{2}\int_{0}^{t}e^{c_{0}\int_{0}^{t'}g(\tau)^{2}\,d\tau}\langle t \rangle^{-1-2\beta(p)} \,dt'  \\
&\, + \xi(t)^{2}\int_{0}^{t}e^{c_{0}\int_{0}^{t'}g(\tau)^{2}\,d\tau}
\left( \|\Delta\bar{u} - \nabla\bar{\Pi}\|_{L^{2}}^{2} + \|\bar{B}\cdot\nabla\bar{B}\|_{L^{2}}^{2} \right)\,dt' \\
&\, + (\delta_{0}+\xi(t))^{2}\int_{0}^{t}e^{c_{0}\int_{0}^{t'}g(\tau)^{2}\,d\tau}\langle t \rangle^{-\frac{3}{2}}\,dt'.
\end{align*}
Taking $g^{2}(t) = \frac{\alpha}{c_{0}(1+t)}$ with $\frac{1}{2} < \alpha < 1+ 2\beta(p)$ in the above inequality gives
\begin{align}
\label{7.16}
\|\tilde{u}(t)\|_{L^{2}} + \|\tilde{B}(t)\|_{L^{2}} \leq C (\delta_{0} + \xi(t))\langle t \rangle^{-\frac{1}{4}},
\end{align}
Applying a similar procedure to (\ref{7.8}) ensures that
\begin{align*}
&\, \frac{d}{dt}\left[ e^{-U(t)}\left( \|\nabla\tilde{u}\|_{L^{2}}^{2} + \|\nabla\tilde{B}\|_{L^{2}}^{2} \right) \right]
+ 2 g^{2}(t)e^{-U(t)} \left( \|\nabla\tilde{u}\|_{L^{2}}^{2} + \|\nabla\tilde{B}\|_{L^{2}}^{2} \right)  \\
&\, \quad\quad\quad\quad\quad\quad\quad\quad
+ e^{-U(t)}\left( \|\partial_{t}\tilde{u}\|_{L^{2}}^{2} + \|\partial_{t}\tilde{B}\|_{L^{2}}^{2} \right)     \\
\lesssim &\, g^{4}(t)e^{-U(t)}\int_{S_{1}(t)}|\hat{\tilde{u}}(\xi)|^{2}\,d\xi
+ g^{4}(t)e^{-U(t)}\int_{S_{2}(t)}|\hat{\tilde{B}}(\xi)|^{2}\,d\xi  \\
&\, + e^{-U(t)}\|\tilde{\rho}\|_{L^{\infty}}^{2}\|\Delta\bar{u} - \nabla\bar{\Pi}\|_{L^{2}}^{2}
+ e^{-U(t)}\|\tilde{\rho}\|_{L^{\infty}}^{2}\|\bar{B}\cdot\nabla\bar{B}\|_{L^{2}}^{2}
\end{align*}
which together with (\ref{7.6}), (\ref{7.6'}) implies that
\begin{align}
\label{7.17}
\begin{split}
&\, \frac{d}{dt}\left[e^{-U(t)}\left( \|\nabla\tilde{u}\|_{L^{2}}^{2} + \|\nabla\tilde{B}\|_{L^{2}}^{2} \right)\right]
+ 2g^{2}(t)e^{-U(t)}\left(\|\nabla\tilde{u}\|_{L^{2}}^{2} + \|\nabla\tilde{B}\|_{L^{2}}^{2}\right)  \\
&\, \quad\quad\quad\quad\quad\quad\quad\quad
+ e^{-U(t)}\left( \|\partial_{t}\tilde{u}\|_{L^{2}}^{2} + \|\partial_{t}\tilde{B}\|_{L^{2}}^{2} \right) \\
\lesssim &\, e^{-U(t)}\bigg\{  g^{4}(t)\int_{S_{1}(t)}e^{-2t|\xi|^{2}}|\hat{\tilde{u}}_{0}(\xi)|^{2}\,d\xi
+ g^{4}(t)\int_{S_{2}(t)}e^{-2t|\xi|^{2}}|\hat{\tilde{B}}_{0}|^{2}\,d\xi    \\
&\, + g(t)^{9}\Big( \|\tilde{u}\|_{L_{t}^{2}(L^{2})}^{2} + \|\tilde{B}\|_{L_{t}^{2}(L^{2})}^{2} \Big)
\Big( \|\bar{u}\|_{L_{t}^{2}(L^{2})}^{2} + \|\bar{B}\|_{L_{t}^{2}(L^{2})}^{2} \Big)         \\
&\, + g(t)^{9} \Big( \|\tilde{B}\|_{L_{t}^{2}(L^{2})}^{2} + \|\tilde{u}\|_{L_{t}^{2}(L^{2})}^{2} \Big)^{2} \\
&\, + g(t)^{7} \|\tilde{B}\|_{L_{t}^{2}(L^{2})}^{2} \Big( \|\nabla\bar{B}\|_{L_{t}^{2}(L^{2})}^{2}
+ \|\nabla\tilde{B}\|_{L_{t}^{2}(L^{2})}^{2} \Big)  \\
&\, + g(t)^{7} \Big( \|\Delta\tilde{u}\|_{L_{t}^{1}(L^{2})}^{2} + \|\nabla\tilde{\Pi}\|_{L_{t}^{1}(L^{2})}^{2}
+ \|\tilde{\rho}\|_{L_{t}^{\infty}(L^{2})}^{2}\|\Delta\bar{u} - \nabla\bar{\Pi}\|_{L_{t}^{1}(L^{2})}^{2} \Big)  \\
&\, + g(t)^{7}\|\tilde{\rho}\|_{L_{t}^{\infty}(L^{2})}^{2}\|\bar{B}\cdot\nabla\bar{B}\|_{L_{t}^{1}(L^{2})}^{2}
+ \|\tilde{\rho}\|_{L^{\infty}}^{2}\|\Delta\bar{u}-\nabla\bar{\Pi}\|_{L^{2}}^{2}    \\
&\, + \|\tilde{\rho}\|_{L^{\infty}}^{2}\|\bar{B}\cdot\nabla\bar{B}\|_{L^{2}}^{2}
\bigg\}
\end{split}
\end{align}
Substituting (\ref{7.16}) into (7.17), we arrive at
\begin{align*}
&\, \frac{d}{dt}\left( e^{-U(t)}\left( \|\nabla\tilde{u}\|_{L^{2}}^{2} + \|\nabla\tilde{B}\|_{L^{2}}^{2} \right) \right)
+ 2 e^{-U(t)} g^{2}(t)\left( \|\nabla\tilde{u}\|_{L^{2}}^{2} + \|\nabla\tilde{B}\|_{L^{2}}^{2} \right)  \\
&\, \quad\quad\quad\quad\quad\quad\quad\quad
+ e^{-U(t)}\left( \|\partial_{t}\tilde{u}\|_{L^{2}}^{2} + \|\partial_{t}\tilde{B}\|_{L^{2}}^{2} \right) \\
\lesssim &\, \left( \delta_{0} + \xi(t) \right)^{2}\langle t \rangle^{-\frac{5}{2}}
+ \xi(t)^{2}\|\Delta\bar{u} - \nabla\bar{\Pi}\|_{L^{2}}^{2} + \xi(t)^{2}\|\bar{B}\cdot\nabla\bar{B}\|_{L^{2}}^{2},
\end{align*}
from which we deduce that
\begin{align*}
&\, e^{2\int_{0}^{t}g(\tau)^{2}\,d\tau}e^{-U(t)}\left( \|\nabla\tilde{u}\|_{L^{2}}^{2} + \|\nabla\tilde{B}\|_{L^{2}}^{2} \right)    \\
&\quad\quad\quad\quad\quad\quad\quad\quad\quad\quad
+ \int_{0}^{t}e^{2\int_{0}^{t'}g(\tau)^{2}\,d\tau}e^{-U(t')}\left(\|\partial_{t}\tilde{u}\|_{L^{2}}^{2} + \|\partial_{t}\tilde{B}\|_{L^{2}}^{2}\right)\,dt' \\
\lesssim &\, \|\nabla\tilde{u}_{0}\|_{L^{2}}^{2} + \|\nabla\tilde{B}_{0}\|_{L^{2}}^{2}
+ \xi(t)^{2}\int_{0}^{t}e^{2\int_{0}^{t'}g(\tau)^{2}\,d\tau}\|\Delta\bar{u}-\nabla\bar{\Pi}\|_{L^{2}}^{2}\,dt'  \\
&\, + \xi(t)^{2}\int_{0}^{t}e^{2\int_{0}^{t'}g(\tau)^{2}\,d\tau} \|\bar{B}\cdot\nabla\bar{B}\|_{L^{2}}^{2}\,dt' \\
&\, + \left( \delta_{0} + \xi(t) \right)^{2} \int_{0}^{t}e^{2\int_{0}^{t'}g(\tau)^{2}\,d\tau}\langle t \rangle^{-\frac{5}{2}} \, dt'.
\end{align*}
Taking $g^{2}(t) = \frac{\alpha}{2(1+t)}$ with $\alpha > 1$ in the above inequality yields
\begin{align}
\label{7.18}
\begin{split}
&\, \langle t \rangle^{\alpha}\left( \|\nabla\tilde{u}\|_{L^{2}}^{2} + \|\nabla\tilde{B}\|_{L^{2}}^{2} \right)
+ \int_{0}^{t} \langle t \rangle^{\alpha}\left( \|\partial_{t}\tilde{u}\|_{L^{2}}^{2} + \|\partial_{t}\tilde{B}\|_{L^{2}}^{2} \right)\,dt'  \\
\lesssim &\, \delta_{0}^{2} + \xi(t)^{2}\int_{0}^{t}\langle t' \rangle^{\alpha}\|\Delta\bar{u} - \nabla\bar{\Pi}\|_{L^{2}}^{2}\,dt'
+ \xi(t)^{2}\int_{0}^{t}\langle t' \rangle^{\alpha}\|\bar{B}\cdot\nabla\bar{B}\|_{L^{2}}^{2}\,dt'   \\
&\, + \Big( \delta_{0}+\xi(t) \Big)^{2}\int_{0}^{t}\langle t' \rangle^{\alpha-\frac{5}{2}}\,dt'.
\end{split}
\end{align}
Taking $\alpha \in (\frac{3}{2}, 1+2\beta(p))$ in (\ref{7.18}) yields
\begin{align}
\|\nabla\tilde{u}(t)\|_{L^{2}} + \|\nabla\tilde{B}(t)\|_{L^{2}} \lesssim \left( \delta_{0}+\xi(t) \right)\langle t \rangle^{-\frac{3}{4}}.
\end{align}

Step 2 : Improved decay estimates of $\|\tilde{u}(t)\|_{L^{2}}$, $\|\nabla\tilde{u}(t)\|_{L^{2}}$ and $\|\tilde{B}(t)\|_{L^{2}}$, $\|\nabla\tilde{B}(t)\|_{L^{2}}$.
On the other hand, taking $\alpha \in [1, \frac{3}{2})$ in (\ref{7.18}), we obtain
\begin{align}
\int_{0}^{t}\langle t' \rangle^{\left( \frac{3}{2} \right)^{-}}\left( \|\partial_{t}\tilde{u}(t')\|_{L^{2}}^{2}
+ \|\partial_{t}\tilde{B}(t')\|_{L^{2}}^{2} \right)\,dt' \lesssim \left( \delta_{0}+\xi(t) \right)^{2}
\end{align}
which implies
\begin{align}
\label{7.20}
\int_{0}^{t}\|\partial_{t}\tilde{u}(t')\|_{L^{2}} + \|\partial_{t}\tilde{B}(t')\|_{L^{2}} \,dt' \lesssim \delta_{0} + \xi(t).
\end{align}
Next, we give the estimates about $\|\Delta\tilde{u}\|_{L_{t}^{1}(L^{2})}^{2} + \|\nabla\tilde{\Pi}\|_{L_{t}^{1}(L^{2})}^{2}$.
Notice that
\begin{align*}
&\, \int_{0}^{t} \|\tilde{u}\cdot\nabla\tilde{u}\|_{L^{2}} + \|\tilde{u}\cdot\nabla\bar{u}\|_{L^{2}}
+ \|\bar{u}\cdot\nabla\tilde{u}\|_{L^{2}}\, dt' \\
\lesssim &\, \int_{0}^{t} \|\tilde{u}\|_{L^{2}}^{\frac{1}{2}}\|\nabla\tilde{u}\|_{L^{2}}^{\frac{1}{2}}\|\Delta\tilde{u}\|_{L^{2}}
+ \|\tilde{u}\|_{L^{2}}^{\frac{1}{2}}\|\nabla\tilde{u}\|_{L^{2}}^{\frac{1}{2}}\|\Delta\bar{u}\|_{L^{2}} \\
&\, \quad\quad\quad\quad\quad\quad\quad\quad\quad\quad\quad\quad\quad\quad
+ \|\bar{u}\|_{L^{2}}^{\frac{1}{2}}\|\nabla\bar{u}\|_{L^{2}}^{\frac{1}{2}}\|\Delta\tilde{u}\|_{L^{2}}\, dt' \\
\lesssim &\, \Big( \delta_{0}+\xi(t) \Big)\ln \langle t \rangle,
\end{align*}
and
\begin{align*}
&\, \int_{0}^{t} \|\tilde{B}\cdot\nabla\tilde{B}\|_{L^{2}} + \|\tilde{B}\cdot\nabla\bar{B}\|_{L^{2}}
+ \|\bar{B}\cdot\nabla\tilde{B}\|_{L^{2}}\, dt' \\
\lesssim &\, \int_{0}^{t} \|\tilde{B}\|_{L^{2}}^{\frac{1}{2}}\|\nabla\tilde{B}\|_{L^{2}}^{\frac{1}{2}}\|\Delta\tilde{B}\|_{L^{2}}
+ \|\tilde{B}\|_{L^{2}}^{\frac{1}{2}}\|\nabla\tilde{B}\|_{L^{2}}^{\frac{1}{2}}\|\Delta\bar{u}\|_{L^{2}} \\
&\, \quad\quad\quad\quad\quad\quad\quad\quad\quad\quad\quad\quad\quad\quad
+ \|\bar{B}\|_{L^{2}}^{\frac{1}{2}}\|\nabla\bar{B}\|_{L^{2}}^{\frac{1}{2}}\|\Delta\tilde{B}\|_{L^{2}}\, dt' \\
\lesssim &\, \Big( \delta_{0}+\xi(t) \Big)\ln \langle t \rangle.
\end{align*}
Substituting the above inequality and (\ref{7.20}) into (\ref{7.11}) results in
\begin{align}
\label{7.21}
\left( \int_{0}^{t}\|\Delta\tilde{u}\|_{L^{2}} + \|\nabla\tilde{\Pi}\|_{L^{2}} \,dt' \right)^{2}
\leq \,C\,\Big( \delta_{0}+\xi(t) \Big)^{2}\ln^{2} \langle t \rangle.
\end{align}
Applying (\ref{7.16}) gives
\begin{align}
\label{7.22}
\int_{0}^{t}\|\tilde{u}\|_{L^{2}}^{2} + \|\tilde{B}\|_{L^{2}}^{2}\,dt' \lesssim \Big( \delta_{0}+\xi(t) \Big)^{2}\langle t \rangle^{\frac{1}{2}}.
\end{align}
Plugging (\ref{7.21}) and (\ref{7.22}) into (\ref{7.17}), we obtain
\begin{align*}
&\, \frac{d}{dt}\left[ e^{-U(t)}\left( \|\nabla\tilde{u}\|_{L^{2}}^{2} + \|\nabla\tilde{B}\|_{L^{2}}^{2} \right)\right]
+ 2 g^{2}(t)e^{-U(t)}\left[ \|\nabla\tilde{u}\|_{L^{2}}^{2} + \|\nabla\tilde{B}\|_{L^{2}}^{2} \right]   \\
&\,\quad\quad\quad\quad\quad\quad
+ e^{-U(t)}\left[ \|\partial_{t}\tilde{u}\|_{L^{2}}^{2} + \|\partial_{t}\tilde{B}\|_{L^{2}}^{2} \right]     \\
\lesssim &\, g^{4}(t)\delta_{0}^{2}\langle t \rangle^{-2\beta(p)} + \Big( \delta_{0}+\xi(t) \Big)^{2}
\langle t \rangle^{-\frac{7}{2}}\ln^{2}\langle t \rangle + \xi^{2}(t)\|\Delta\bar{u} - \nabla\bar{\Pi}\|_{L^{2}}^{2}    \\
&\, + \xi^{2}(t)\|\bar{B}\cdot\nabla\bar{B}\|_{L^{2}}^{2}   \\
\lesssim &\, \Big( \delta_{0}+\xi(t) \Big)^{2}\langle t \rangle^{-2-2\beta(p)} + \xi^{2}(t)\|\Delta\bar{u}-\nabla\bar{\Pi}\|_{L^{2}}^{2}
+ \xi^{2}(t)\|\bar{B}\cdot\nabla\bar{B}\|_{L^{2}}^{2}.
\end{align*}
Taking $g^{2}(t) = \frac{\alpha}{c_{0}(1+t)}$ with $1+2\beta(p) < \alpha < 2 + 6\beta(p)$ in the above inequality and integrating
the resulting inequality over $[0, t]$, we arrive at
\begin{align}
\label{7.23}
\|\nabla\tilde{u}(t)\|_{L^{2}}+\|\nabla\tilde{B}(t)\|_{L^{2}}\lesssim \Big( \delta_{0}+\xi(t) \Big)\langle t \rangle^{-\frac{1}{2}-\beta(p)}
\end{align}
for $t < \bar{T}$.
Now, taking $g^{2}(t) = \frac{\alpha}{c_{0}(1+t)}$ with $\alpha < 1+2\beta(p)$, we deduce from Theorem \ref{decay_main_theorem} that
\begin{align}
\label{7.24}
\int_{0}^{t}\langle t' \rangle^{(1+2\beta(p))^{-}} \left( \|\partial_{t}\tilde{u}\|_{L^{2}}^{2} + \|\partial_{t}\tilde{B}\|_{L^{2}}^{2} \right)\,dt'
\lesssim \left( \delta_{0}+\xi(t) \right)^{2} \quad \text{for } t\leq \bar{T}.
\end{align}
Plugging (\ref{7.21}) and (\ref{7.22}) into (\ref{7.2}), we arrive at
\begin{align*}
&\, \frac{d}{dt}\left[ e^{-V(t)}\left( \|\sqrt{\rho}\tilde{u}\|_{L^{2}}^{2} + \|\tilde{B}\|_{L^{2}}^{2} \right)\right]
+ c_{0}g^{2}(t)e^{-V(t)}\left( \|\sqrt{\rho}\tilde{u}\|_{L^{2}}^{2} + \|\tilde{B}\|_{L^{2}}^{2} \right) \\
\lesssim &\, \Big( \delta_{0}^{2}+\xi(t) \Big)^{2}\langle t \rangle^{-1-2\beta(p)} + \xi^{2}(t)\|\Delta\bar{u}-\nabla\bar{\Pi}\|_{L^{2}}^{2}
+ \xi^{2}(t)\|\bar{B}\cdot\nabla\bar{B}\|_{L^{2}}^{2}
\end{align*}
Then, taking $g^{2}(t) = \frac{\alpha}{c_{0}(1+t)}$ with $1+2\beta(p) < \alpha < 2+ 6\beta(p)$ in the above inequality and
integrating the resulting inequality over $[0, t]$, we reach
\begin{align}
\|\tilde{u}(t)\|_{L^{2}} + \|\tilde{B}(t)\|_{L^{2}}\lesssim \Big( \delta_{0} + \xi(t) \Big)\langle t \rangle^{-\beta(p)}.
\end{align}

Step 3 : Time integral estimates of $\|\nabla\bar{\Pi}\|_{L^{2}}$, $\|\Delta\tilde{u}(t)\|_{L^{2}}$ and $\|\Delta\tilde{B}(t)\|_{L^{2}}$.
It follows from (\ref{7.9}) and (\ref{7.11}) that
\begin{align*}
\|\Delta\tilde{\Pi}\|_{L^{2}} + \|\Delta\tilde{u}\|_{L^{2}} \lesssim &\, \|\partial_{t}\tilde{u}\|_{L^{2}}
+ \|\tilde{u}\|_{L^{2}}^{\frac{1}{2}}\|\nabla\tilde{u}\|_{L^{2}}^{\frac{1}{2}}\|\Delta\tilde{u}\|_{L^{2}}
+ \|\tilde{B}\|_{L^{2}}^{\frac{1}{2}}\|\nabla\tilde{B}\|_{L^{2}}^{\frac{1}{2}}\|\Delta\tilde{B}\|_{L^{2}}   \\
&\, + \|\tilde{\rho}\|_{L^{\infty}}\|\Delta\bar{u} - \nabla\bar{\Pi}\|_{L^{2}} + \|\tilde{\rho}\|_{L^{\infty}}\|\bar{B}\cdot\nabla\bar{B}\|_{L^{2}}   \\
&\, + \|\nabla\tilde{u}\|_{L^{2}}\left( \|\nabla\bar{u}\|_{L^{2}}^{\frac{1}{2}}\|\Delta\bar{u}\|_{L^{2}}^{\frac{1}{2}}+\|\bar{u}\|_{L^{\infty}} \right) \\
&\, + \|\nabla\tilde{B}\|_{L^{2}}\left( \|\nabla\bar{B}\|_{L^{2}}^{\frac{1}{2}}\|\Delta\bar{B}\|_{L^{2}}^{\frac{1}{2}}+\|\bar{B}\|_{L^{\infty}} \right),
\end{align*}
and from the equation of $\tilde{B}$, we have
\begin{align*}
\|\Delta\tilde{B}\|_{L^{2}} \lesssim &\,\|\partial_{t}\tilde{B}\|_{L^{2}} + \|\nabla\tilde{B}\|_{L^{2}}\|\bar{u}\|_{L^{\infty}}
+ \|\nabla\tilde{u}\|_{L^{2}}\|\bar{B}\|_{L^{\infty}}   \\
&\, + \|\tilde{u}\|_{L^{2}}^{\frac{1}{2}}\|\nabla\tilde{u}\|_{L^{2}}^{\frac{1}{2}}\|\Delta\tilde{B}\|_{L^{2}}
+ \|\tilde{B}\|_{L^{2}}^{\frac{1}{2}}\|\nabla\tilde{B}\|_{L^{2}}^{\frac{1}{2}}\|\Delta\tilde{u}\|_{L^{2}}   \\
&\, + \|\nabla\tilde{u}\|_{L^{2}}\|\nabla\bar{B}\|_{L^{2}}^{\frac{1}{2}}\|\Delta\bar{B}\|_{L^{2}}^{\frac{1}{2}}
+ \|\nabla\tilde{B}\|_{L^{2}}\|\nabla\bar{u}\|_{L^{2}}^{\frac{1}{2}}\|\Delta\bar{u}\|_{L^{2}}^{\frac{1}{2}}.
\end{align*}
Inequality (\ref{7.7}) along with
\begin{align*}
& \|\bar{u}\|_{L^{\infty}}\lesssim \|\bar{u}\|_{L^{6}}^{\frac{1}{2}}\|\nabla\bar{u}\|_{L^{6}}^{\frac{1}{2}}
\lesssim \|\nabla\bar{u}\|_{L^{2}}^{\frac{1}{2}}\|\Delta\bar{u}\|_{L^{2}}^{\frac{1}{2}},   \\
& \|\bar{B}\|_{L^{\infty}}\lesssim \|\bar{B}\|_{L^{6}}^{\frac{1}{2}}\|\nabla\bar{B}\|_{L^{6}}^{\frac{1}{2}}
\lesssim \|\nabla\bar{B}\|_{L^{2}}^{\frac{1}{2}}\|\Delta\bar{B}\|_{L^{2}}^{\frac{1}{2}},
\end{align*}
implies for $t\leq \bar{T}$
\begin{align*}
&\, \|\nabla\tilde{\Pi}\|_{L^{2}} + \|\Delta\tilde{u}\|_{L^{2}} + \|\Delta\tilde{B}\|_{L^{2}}   \\
\lesssim &\, \|\partial_{t}\tilde{u}\|_{L^{2}} + \|\partial_{t}\tilde{B}\|_{L^{2}}
+ \|\tilde{\rho}\|_{L^{\infty}}\|\bar{B}\cdot\nabla\bar{B}\|_{L^{2}}
+ \|\tilde{\rho}\|_{L^{\infty}}\|\Delta\bar{u} - \nabla\bar{\Pi}\|_{L^{2}}  \\
&\, + \left( \|\nabla\tilde{u}\|_{L^{2}} + \|\nabla\tilde{B}\|_{L^{2}} \right)
\left( \|\nabla\bar{u}\|_{L^{2}}^{\frac{1}{2}}\|\Delta\bar{u}\|_{L^{2}}^{\frac{1}{2}}
+ \|\nabla\bar{B}\|_{L^{2}}^{\frac{1}{2}}\|\Delta\bar{B}\|_{L^{2}}^{\frac{1}{2}} \right)
\end{align*}
Therefore, thanks to Theorem (\ref{decay_main_theorem}) and (\ref{7.24}), we obtain
\begin{align*}
\begin{split}
&\, \int_{0}^{t}\langle t' \rangle^{(1+2\beta(p))^{-}}\left( \|\Delta\tilde{u}\|_{L^{2}}^{2} + \|\Delta\tilde{B}\|_{L^{2}}^{2}
+ \|\nabla\tilde{\Pi}\|_{L^{2}}^{2} \right)\,dt'    \\
\lesssim &\, \int_{0}^{t}\langle t' \rangle^{(1+2\beta(p))^{-}}\left( \|\partial_{t}\tilde{u}\|_{L^{2}}^{2} + \|\partial_{t}\tilde{B}\|_{L^{2}}^{2}
\right)\,dt'    \\
&\, + \left( \delta_{0}+\xi(t) \right)^{2}\int_{0}^{t}\|\nabla\bar{u}\|_{L^{2}}\|\Delta\bar{u}\|_{L^{2}}
+ \|\nabla\bar{B}\|_{L^{2}}\|\Delta\bar{B}\|_{L^{2}} \,dt'  \\
&\, + \left( \delta_{0}+\xi(t) \right)^{2}\int_{0}^{t} \langle t' \rangle^{(1+2\beta(p))^{-}}\|\Delta\bar{u} - \nabla\bar{\Pi}\|_{L^{2}}^{2}\,dt'  \\
&\, + \left( \delta_{0}+\xi(t) \right)^{2}\int_{0}^{t} \langle t' \rangle^{(1+2\beta(p))^{-}}\|\bar{B}\cdot\nabla\bar{B}\|_{L^{2}}^{2}\,dt' \\
\leq &\, C \left( \delta_{0}+\xi(t) \right)^{2}
\end{split}
\end{align*}
for $t \leq \bar{T}$. So finally, we obtain that
\begin{align}
\label{7.25}
\begin{split}
\int_{0}^{t} \langle t' \rangle^{(1+2\beta(p))^{-}} \left( \|\Delta\tilde{u}\|_{L^{2}}^{2} + \|\Delta\tilde{B}\|_{L^{2}}^{2}
+ \|\nabla\tilde{\Pi}\|_{L^{2}}^{2} \right)\,dt'\leq C \left( \delta_{0}+\xi(t) \right)^{2}
\end{split}
\end{align}
which leads to
\begin{align}
\int_{0}^{t} \|\Delta\tilde{u}\|_{L^{2}} + \|\Delta\tilde{B}\|_{L^{2}} + \|\nabla\tilde{\Pi}\|_{L^{2}} \,dt'
\leq C \left( \delta_{0}+\xi(t) \right)
\end{align}
for $t \leq \bar{T}$.

Step 4 : Estimate of $\int_{0}^{t}\|\tilde{u}(t')\|\,dt'$ and $\int_{0}^{t}\|\tilde{B}(t')\|_{L^{\infty}}\,dt'$.
Thanks to (\ref{7.23}), we have
\begin{align}
\label{7.26}
\begin{split}
\|\tilde{u}(t)\|_{L^{\infty}} \lesssim &\, \|\tilde{u}(t)\|_{L^{6}}^{\frac{1}{2}}\|\nabla\tilde{u}(t)\|_{L^{6}}^{\frac{1}{2}}   \\
\lesssim &\, \|\nabla\tilde{u}(t)\|_{L^{2}}^{\frac{1}{2}}\|\Delta\tilde{u}(t)\|_{L^{2}}^{\frac{1}{2}}   \\
\leq &\, C_{\eta} \left( \delta_{0}+\xi(t) \right)\langle t \rangle^{-\frac{1}{2}-\beta(p)} + \eta \|\Delta\tilde{u}(t)\|_{L^{2}},
\end{split}
\end{align}
and
\begin{align}
\label{7.26'}
\begin{split}
\|\tilde{B}(t)\|_{L^{\infty}} \leq C_{\eta}\left( \delta_{0}+\xi(t) \right)\langle t \rangle^{-\frac{1}{2}-\beta(p)}
+ \eta \|\Delta\tilde{B}(t)\|_{L^{2}},
\end{split}
\end{align}
for any $\eta > 0$.
It is easy to observe from the transport equation that
\begin{align*}
\partial_{t}\tilde{\rho} + u\cdot\nabla\tilde{\rho} = -\tilde{u}\cdot\nabla\bar{\rho}.
\end{align*}
From this and $\mathrm{div}\,u = 0$, we deduce that
\begin{align*}
\|\tilde{\rho}(t)\|_{L_{t}^{\infty}(L^{q})} \leq \|\tilde{\rho}_{0}\|_{L^{q}}
+ \|\nabla\bar{\rho}\|_{L_{t}^{\infty}(L^{q})}\|\tilde{u}\|_{L_{t}^{1}(L^{\infty})}
\end{align*}
for any $q \in [1, \infty]$. From the above and (\ref{estimate of a in 5 2}), we get
\begin{align}
\label{7.27}
\xi(t) \leq \delta_{0} + C \|\tilde{u}\|_{L_{t}^{1}(L^{\infty})}.
\end{align}
From (\ref{7.25}), we know that
\begin{align*}
&\, \int_{0}^{t}\|\Delta\tilde{u}\|_{L^{2}} + \|\Delta\tilde{B}\|_{L^{2}}\,dt' \\
\lesssim &\, \left( \int_{0}^{t} \left(\|\Delta\tilde{u}\|_{L^{2}}^{2}+\|\Delta\tilde{B}\|_{L^{2}}^{2}\right)
\langle t' \rangle^{(1+2\beta(p))^{-}}\,dt'\right)^{\frac{1}{2}}    \\
\leq &\, C \left( \delta_{0}+\xi(t) \right).
\end{align*}
Plugging (\ref{7.27}) into (\ref{7.26}) and integrating the resulting equation over $[0, t]$, we obtain
\begin{align}
\begin{split}
&\, \|\tilde{u}\|_{L_{t}^{1}(L^{\infty})} + \|\tilde{B}\|_{L_{t}^{1}(L^{\infty})} \\
\leq &\, C_{\eta} \delta_{0}
+ C_{\eta} \int_{0}^{t}\langle t' \rangle^{-\frac{1}{2}-\beta(p)}\int_{0}^{t'}\|\tilde{u}\|_{L^{\infty}}+\|\tilde{B}\|_{L^{\infty}}\,d\tau\,dt' \\
&\, + \eta \left( \|\Delta\tilde{u}\|_{L_{t}^{1}(L^{2})} + \|\Delta\tilde{B}\|_{L_{t}^{1}(L^{2})} \right)   \\
\leq &\, C_{\eta} \delta_{0}
+ C_{\eta} \int_{0}^{t}\langle t' \rangle^{-\frac{1}{2}-\beta(p)}\int_{0}^{t'}\|\tilde{u}\|_{L^{\infty}}+\|\tilde{B}\|_{L^{\infty}}\,d\tau\,dt' \\
&\, + \eta \left( \delta_{0}+\|\tilde{u}\|_{L_{t}^{1}(L^{\infty})} + \|\tilde{B}\|_{L_{t}^{1}(L^{\infty})} \right).
\end{split}
\end{align}
Taking $\eta = \frac{1}{2C}$. Let $m(t) = \|\tilde{u}\|_{L_{t}^{1}(L^{\infty})} + \|\tilde{B}\|_{L_{t}^{1}(L^{\infty})}$, we obtain
\begin{align*}
m(t) \leq C \left( \delta_{0} + \int_{0}^{t} \langle t' \rangle^{-\frac{1}{2}-\beta(p)} m(t') \,dt' \right)
\end{align*}
for $t \leq \bar{T}$. Applying the Gronwall's inequality gives
\begin{align*}
m(t) = \|\tilde{u}\|_{L_{t}^{1}(L^{\infty})} + \|\tilde{B}\|_{L_{t}^{1}(L^{\infty})} \leq C\delta_{0}
\end{align*}
for $t \leq \bar{T}$.
From this and (\ref{7.27}), we deduce that
\begin{align}
\label{7.29}
\xi(t) \leq C \delta_{0} \quad \text{for } t\leq \bar{T}.
\end{align}
Thanks to (\ref{7.23}), (\ref{7.24}) and (\ref{7.25}), we arrive at (\ref{7.14}) and (\ref{7.14'}) for $t \leq \bar{T}$.
Moreover, (\ref{7.15}) and (\ref{7.29}) ensures that
\begin{align*}
& \|\tilde{u}(t)\|_{L^{2}}+ \|\tilde{B}(t)\|_{L^{2}} + \|\nabla\tilde{u}(t)\|_{L^{2}} + \|\nabla\tilde{B}(t)\|_{L^{2}} \\
&\quad\quad\quad\,\,\,\,   + \|\sqrt{\rho}\partial_{t}\tilde{u}\|_{L_{t}^{2}(L^{2})} + \|\partial_{t}\tilde{B}\|_{L_{t}^{2}(L^{2})}
+ \|\Delta\tilde{u}\|_{L_{t}^{2}(L^{2})} + \|\Delta\tilde{B}\|_{L_{t}^{2}(L^{2})}\leq C \delta_{0}
\end{align*}
for $t \leq \bar{T}$.
Now, let $\nu$ be the small constant determined in (\ref{7.7}). We take $\delta_{0}$ sufficiently small such that
\begin{align*}
\|\tilde{u}(t)\|_{L^{2}}\|\nabla\tilde{u}(t)\|_{L^{2}}
+ \|\tilde{B}(t)\|_{L^{2}}\|\nabla\tilde{B}(t)\|_{L^{2}} \leq C \delta_{0}^{2} \leq \frac{\nu}{2}
\end{align*}
for $t \leq \bar{T}$.
This shows that $\bar{T}$ can be any time smaller than $\tilde{T}^{*}$. This completes the proof of the proposition.
\end{proof}

\begin{proposition}
Under the assumptions of Theorem \ref{stability main theorem}, there exist constants $C$ and $c$ so that if
\begin{align*}
A_{0} := \|\tilde{u}_{0}\|_{H^{1}} + \|\tilde{u}_{0}\|_{L^{p}} + \|\tilde{B}_{0}\|_{H^{1}} + \|\tilde{B}_{0}\|_{L^{p}}
+ \|\tilde{a}_{0}\|_{B_{2,1}^{3/2}} \leq c,
\end{align*}
then we have
\begin{align}
\begin{split}
\|\tilde{a}\|_{\tilde{L}_{t}^{\infty}(B_{2,1}^{3/2})} & + \|\tilde{u}\|_{\tilde{L}_{t}^{\infty}(L^{p})}
+ \|\tilde{u}\|_{\tilde{L}_{t}^{\infty}(\dot{B}_{2,1}^{1/2})} + \|\tilde{u}\|_{L_{t}^{1}(\dot{B}_{2,1}^{5/2})}  \\
& + \|\tilde{B}\|_{\tilde{L}_{t}^{\infty}(L^{p})}
+ \|\tilde{B}\|_{\tilde{L}_{t}^{\infty}(\dot{B}_{2,1}^{1/2})} + \|\tilde{B}\|_{L_{t}^{1}(\dot{B}_{2,1}^{5/2})} \leq C A_{0}
\end{split}
\end{align}
for all $t < \tilde{T}^{*}$.
\end{proposition}
\begin{proof}
Note that $\delta_{0} \lesssim A_{0}$ and $\tilde{a} = a - \bar{a}$. Thanks to Proposition \ref{global reference 5/2},
Proposition \ref{global reference 2} and Proposition \ref{decay bar}, we get by applying (\ref{transport 1}) to
the transport equations in (\ref{mhd_a}) and using Gronwall's inequality that
\begin{align}
\label{7.31}
\begin{split}
\|a\|_{\tilde{L}_{t}^{\infty}(\dot{H}^{2})} \leq C \exp{\left\{ C \|\tilde{u}\|_{L_{t}^{1}(\dot{B}_{2,1}^{5/2})} \right\}} \|a_{0}\|_{\dot{H}^{2}}.
\end{split}
\end{align}
Similarly applying (\ref{transport 1}) to the transport equations in (\ref{perturbed solution}) and using Gronwall's inequality that
\begin{align}
\label{7.32}
\begin{split}
\|\tilde{a}\|_{\tilde{L}_{t}^{\infty}(B_{2,1}^{3/2})} \leq C A_{0}\exp{\left\{ C \|\tilde{u}\|_{L_{t}^{1}(\dot{B}_{2,1}^{5/2})} \right\}}
\end{split}
\end{align}
for any $t < \tilde{T}^{*}$.
Applying Lemma \ref{linear estimate momentum} to the second equation in (\ref{perturbed solution}) yields
\begin{align}
\label{u b per}
\begin{split}
&\, \|\tilde{u}\|_{\tilde{L}_{t}^{\infty}(\dot{B}_{2,1}^{\frac{1}{2}})} + \|\tilde{u}\|_{\tilde{L}_{t}^{1}(\dot{B}_{2,1}^{5/2})}
+ \|\tilde{B}\|_{\tilde{L}_{t}^{\infty}(\dot{B}_{2,1}^{1/2})} + \|\tilde{B}\|_{\tilde{L}_{t}^{1}(\dot{B}_{2,1}^{5/2})} \\
\leq &\, C\exp{\left( C\|\bar{u}\|_{L_{t}^{1}(\dot{B}_{2,1}^{5/2})} + C\|\bar{B}\|_{L_{t}^{1}(\dot{B}_{2,1}^{5/2})} \right)}
\bigg\{ \|\tilde{u}_{0}\|_{\dot{B}_{2,1}^{\frac{1}{2}}} + \|\tilde{B}_{0}\|_{\dot{B}_{2,1}^{1/2}}  \\
&\, + \|\tilde{u}\cdot\nabla\bar{u}\|_{\tilde{L}_{t}^{1}(\dot{B}_{2,1}^{1/2})} + \|\tilde{u}\cdot\nabla\tilde{u}\|_{L_{t}^{1}(\dot{B}_{2,1}^{1/2})}
+ \|a\, \bar{B}\cdot\nabla\tilde{B}\|_{L_{t}^{1}(\dot{B}_{2,1}^{1/2})}  \\
&\, + \|(1+a)(\tilde{B}\cdot\nabla\bar{B} + \tilde{B}\cdot\nabla\tilde{B})\|_{L_{t}^{1}(\dot{B}_{2,1}^{1/2})}
+ \|\tilde{a}(\Delta\bar{u} - \nabla\bar{\Pi})\|_{L_{t}^{1}(\dot{B}_{2,1}^{1/2})}   \\
&\, + \|\tilde{a}\,\bar{B}\cdot\nabla\bar{B}\|_{L_{t}^{1}(\dot{B}_{2,1}^{1/2})}
+ \|\tilde{B}\cdot\nabla\bar{u}\|_{L_{t}^{1}(\dot{B}_{2,1}^{1/2})} + \|\tilde{B}\cdot\nabla\tilde{u}\|_{L_{t}^{1}(\dot{B}_{2,1}^{1/2})}  \\
&\, + \|\tilde{u}\cdot\nabla\bar{B}\|_{L_{t}^{1}(\dot{B}_{2,1}^{1/2})} + \|\tilde{u}\cdot\nabla\tilde{B}\|_{L_{t}^{1}(\dot{B}_{2,1}^{1/2})}
+ \|a\|_{L_{t}^{\infty}(\dot{H}^{2})}\|\tilde{u}\|_{L_{t}^{1}(\dot{B}_{2,1}^{2})}   \\
&\, + \|a\|_{L_{t}^{\infty}(\dot{H}^{2})}\|\nabla\tilde{\Pi}\|_{L_{t}^{1}(L^{2})}  \bigg\}.
\end{split}
\end{align}
Next, we list some useful estimates
\begin{align*}
& \|\tilde{u}\cdot\nabla\tilde{u}\|_{L_{t}^{1}(\dot{B}_{2,1}^{1/2})}
\lesssim \|\tilde{u}\|_{L_{t}^{\infty}(\dot{B}_{2,1}^{1/2})}\|\tilde{u}\|_{L_{t}^{1}(\dot{B}_{2,1}^{5/2})},  \\
& \|(1+a)\tilde{B}\cdot\nabla\tilde{B}\|_{L_{t}^{1}(\dot{B}_{2,1}^{1/2})}
\lesssim (1+\|a\|_{L_{t}^{\infty}(\dot{B}_{2,1}^{3/2})}) \|\tilde{B}\|_{L_{t}^{\infty}(\dot{B}_{2,1}^{1/2})}\|\tilde{B}\|_{L_{t}^{1}(\dot{B}_{2,1}^{5/2})},\\
& \|\tilde{B}\cdot\nabla\tilde{u}\|_{L_{t}^{1}(\dot{B}_{2,1}^{1/2})}
\lesssim \|\tilde{B}\|_{L_{t}^{\infty}(\dot{B}_{2,1}^{1/2})}\|\tilde{B}\|_{L_{t}^{1}(\dot{B}_{2,1}^{5/2})}, \\
& \|\tilde{u}\cdot\nabla\tilde{B}\|_{L_{t}^{1}(\dot{B}_{2,1}^{1/2})}
\lesssim \|\tilde{u}\|_{L_{t}^{\infty}(\dot{B}_{2,1}^{1/2})}\|\tilde{B}\|_{L_{t}^{1}(\dot{B}_{2,1}^{5/2})}, \\
& \|\tilde{u}\cdot\nabla\bar{u}\|_{L_{t}^{1}(\dot{B}_{2,1}^{1/2})}
\lesssim \|\tilde{u}\|_{L_{t}^{\infty}(\dot{B}_{2,1}^{1/2})}\|\nabla\bar{u}\|_{L_{t}^{1}(\dot{B}_{2,1}^{3/2})},  \\
& \|a\,\bar{B}\cdot\nabla\tilde{B}\|_{L_{t}^{1}(\dot{B}_{2,1}^{1/2})}
\lesssim \|a\|_{L_{t}^{\infty}(\dot{B}_{2,1}^{3/2})}\|\bar{B}\|_{L_{t}^{\infty}(\dot{B}_{2,1}^{1})}\|\tilde{B}\|_{L_{t}^{1}(\dot{B}_{2,1}^{2})},    \\
& \|(1+a)\tilde{B}\cdot\nabla\bar{B}\|_{L_{t}^{1}(\dot{B}_{2,1}^{1/2})}
\lesssim (1+\|a\|_{L_{t}^{\infty}(\dot{B}_{2,1}^{3/2})})\|\tilde{B}\|_{L_{t}^{\infty}(\dot{B}_{2,1}^{1/2})}\|\bar{B}\|_{L_{t}^{1}(\dot{B}_{2,1}^{5/2})}, \\
& \|\tilde{a}\,\bar{B}\cdot\nabla\bar{B}\|_{L_{t}^{1}(\dot{B}_{2,1}^{1/2})}
\lesssim \|\tilde{a}\|_{L_{t}^{1}(\dot{B}_{2,1}^{3/2})}\|\bar{B}\|_{L_{t}^{\infty}(\dot{B}_{2,1}^{1/2})}\|\bar{B}\|_{L_{t}^{1}(\dot{B}_{2,1}^{5/2})}, \\
& \|\tilde{u}\cdot\nabla\bar{B}\|_{L_{t}^{1}(\dot{B}_{2,1}^{1/2})}
\lesssim \|\tilde{u}\|_{L_{t}^{\infty}(\dot{B}_{2,1}^{1/2})}\|\bar{B}\|_{L_{t}^{1}(\dot{B}_{2,1}^{5/2})},   \\
& \|\tilde{a}\,(\Delta\bar{u}-\nabla\bar{\Pi})\|_{L_{t}^{1}(\dot{B}_{2,1}^{1/2})}
\lesssim \|\tilde{a}\|_{L_{t}^{\infty}(\dot{B}_{2,1}^{3/2})}\|\Delta\bar{u}-\nabla\bar{\Pi}\|_{L_{t}^{1}(\dot{B}_{2,1}^{1/2})}  \\
&\quad\quad\quad\quad\quad\quad\quad\quad\quad\quad
\leq C A_{0} \exp{\left\{ C \|\tilde{u}\|_{L_{t}^{1}(\dot{B}_{2,1}^{5/2})} \right\}},
\end{align*}
which along with (\ref{7.14}), (\ref{7.31}), (\ref{u b per}) and the following inequalities
\begin{align*}
& \|\tilde{u}\|_{L_{t}^{\infty}(\dot{B}_{2,1}^{1/2})}
\lesssim \|\tilde{u}\|_{L_{t}^{\infty}(L^{2})}^{\frac{1}{2}}\|\nabla\tilde{u}\|_{L_{t}^{\infty}(L^{2})}^{\frac{1}{2}},  \\
& \|\tilde{u}\|_{L_{t}^{1}(\dot{B}_{2,1}^{2})}
\lesssim \|\tilde{u}\|_{L_{t}^{1}(\dot{B}_{2,1}^{5/2})}^{\frac{2}{3}}\|\tilde{u}\|_{L_{t}^{1}(\dot{H}^{1})}^{\frac{1}{3}},  \\
& \|\tilde{B}\|_{L_{t}^{\infty}(\dot{B}_{2,1}^{1/2})}
\lesssim \|\tilde{B}\|_{L_{t}^{\infty}(L^{2})}^{\frac{1}{2}}\|\nabla\tilde{B}\|_{L_{t}^{\infty}(L^{2})}^{\frac{1}{2}},  \\
& \|\tilde{B}\|_{L_{t}^{1}(\dot{B}_{2,1}^{2})}
\lesssim \|\tilde{B}\|_{L_{t}^{1}(\dot{B}_{2,1}^{5/2})}^{\frac{2}{3}}\|\tilde{B}\|_{L_{t}^{1}(\dot{H}^{1})}^{\frac{1}{3}},
\end{align*}
imply that
\begin{align}\label{fanbo 2}
\begin{split}
& \|\tilde{u}\|_{\tilde{L}_{t}^{\infty}(\dot{B}_{2,1}^{1/2})}+\|\tilde{u}\|_{L_{t}^{1}(\dot{B}_{2,1}^{5/2})}
+\|\tilde{B}\|_{\tilde{L}_{t}^{\infty}(\dot{B}_{2,1}^{1/2})} + \|\tilde{B}\|_{L_{t}^{1}(\dot{B}_{2,1}^{5/2})}   \\
\leq &\, C \bigg[ \|\tilde{u}_{0}\|_{\dot{B}_{2,1}^{1/2}} + \|\tilde{B}_{0}\|_{\dot{B}_{2,1}^{1/2}} + \delta_{0}\|\tilde{u}\|_{L_{t}^{1}(\dot{B}_{2,1}^{5/2})}
+ \delta_{0}\|\tilde{B}\|_{L_{t}^{1}(\dot{B}_{2,1}^{5/2})} + A_{0}    \\
&\, + A_{0}\exp{\left( C \|\tilde{u}\|_{L_{t}^{1}(\dot{B}_{2,1}^{5/2})} \right)} \bigg]
+ \frac{1}{2}\|\tilde{B}\|_{L_{t}^{1}(\dot{B}_{2,1}^{5/2})} + \frac{1}{2}\|\tilde{u}\|_{L_{t}^{1}(\dot{B}_{2,1}^{5/2})}.
\end{split}
\end{align}
This gives
\begin{align*}
& \|\tilde{u}\|_{\tilde{L}_{t}^{\infty}(\dot{B}_{2,1}^{\frac{1}{2}})} + \|\tilde{B}\|_{\tilde{L}_{t}^{\infty}(\dot{B}_{2,1}^{\frac{1}{2}})}
+ \left( \frac{1}{2} - C \delta_{0} \right)\left( \|\tilde{u}\|_{L_{t}^{1}(\dot{B}_{2,1}^{5/2})} + \|\tilde{B}\|_{L_{t}^{1}(\dot{B}_{2,1}^{5/2})} \right) \\
\lesssim &\, \|\tilde{u}_{0}\|_{\dot{B}_{2,1}^{1/2}} + \|\tilde{B}_{0}\|_{\dot{B}_{2,1}^{1/2}} + A_{0}
+ A_{0}\exp{\left( C \|\tilde{u}\|_{L_{t}^{1}(\dot{B}_{2,1}^{5/2})} \right)}.
\end{align*}
Taking $A_{0} < c$ sufficiently small, we deduce
\begin{align}
\label{a u b}
\begin{split}
\|\tilde{a}\|_{\tilde{L}_{t}^{\infty}(B_{2,1}^{3/2})} + \|\tilde{u}\|_{\tilde{L}_{t}^{\infty}(\dot{B}_{2,1}^{1/2})}
& + \|\tilde{u}\|_{L_{t}^{1}(\dot{B}_{2,1}^{5/2})}    \\
& + \|\tilde{B}\|_{\tilde{L}_{t}^{\infty}(\dot{B}_{2,1}^{1/2})}
+ \|\tilde{B}\|_{L_{t}^{1}(\dot{B}_{2,1}^{5/2})} \leq C A_{0}
\end{split}
\end{align}
for any $t < \tilde{T}^{*}$.

On the other hand, by multiplying by $|\tilde{u}^{i}|^{p-1}\mathrm{sgn}(\tilde{u}^{i})$ the $\tilde{u}^{i}$ equation in (\ref{perturbed solution})
for $i = 1, 2, 3$ and integrating the resulting equation over $\mathbb{R}^{3}$, we obtain
\begin{align}
\label{u p}
\begin{split}
\frac{d}{dt}\|\tilde{u}\|_{L^{p}}^{p} \lesssim &\, \|\tilde{u}\|_{L^{p}}^{p-1}\bigg\{
\|\tilde{u}\cdot\nabla\tilde{u}\|_{L^{p}} + \|\tilde{u}\cdot\nabla\bar{u}\|_{L^{p}} + \|\bar{u}\cdot\nabla\tilde{u}\|_{L^{p}}   \\
&\, + \|(\bar{a} + \tilde{a})(\Delta\tilde{u} - \nabla\tilde{\Pi})\|_{L^{p}} + \|\nabla\tilde{\Pi}\|_{L^{p}}    \\
&\, + \|\tilde{a}\,(\Delta\bar{u} - \nabla\bar{\Pi})\|_{L^{p}} + \|\tilde{a}\,\bar{B}\cdot\nabla\bar{B}\|_{L^{p}}   \\
&\, + \|(1+\bar{a}+\tilde{a})\tilde{B}\cdot\nabla\tilde{B}\|_{L^{p}} + \|(1+\bar{a}+\tilde{a})\tilde{B}\cdot\nabla\bar{B}\|_{L^{p}} \\
&\, + \|(1+\bar{a}+\tilde{a})\bar{B}\cdot\nabla\tilde{B}\|_{L^{p}} \bigg\}.
\end{split}
\end{align}
Similarly, for $\tilde{B}$, we have
\begin{align}
\label{b p}
\begin{split}
\frac{d}{dt}\|\tilde{B}\|_{L^{p}}^{p} \lesssim &\, \|\tilde{B}\|_{L^{p}}^{p-1} \bigg\{
\|\tilde{u}\cdot\nabla\tilde{B}\|_{L^{p}} + \|\tilde{B}\cdot\nabla\tilde{u}\|_{L^{p}} + \|\tilde{B}\cdot\nabla\bar{u}\|_{L^{p}} \\
&\, + \|\bar{B}\cdot\nabla\tilde{u}\|_{L^{p}} + \|\tilde{u}\cdot\nabla\bar{B}\|_{L^{p}} + \|\bar{u}\cdot\nabla\tilde{B}\|_{L^{p}} \bigg\}
\end{split}
\end{align}
Notice that applying the operator $\mathrm{div}$ to the second equation in (\ref{perturbed solution}) results in
\begin{align*}
\Delta\tilde{\Pi} = & \mathrm{div}\,\bigg\{
- \tilde{u}\cdot\nabla\tilde{u} - \tilde{u}\cdot\nabla\bar{u} - \bar{u}\cdot\nabla\tilde{u} + (\bar{a}+\tilde{a})(\Delta\tilde{u}-\nabla\tilde{\Pi}) \\
& + (1+\bar{a}+\tilde{a})\tilde{B}\cdot\nabla\tilde{B} + (1+\bar{a}+\tilde{a})\tilde{B}\cdot\nabla\bar{B} \\
& + (1+\bar{a}+\tilde{a})\bar{B}\cdot\nabla\tilde{B} + \tilde{a}\,(\Delta\bar{u}-\nabla\bar{\Pi}) + \tilde{a}\,\bar{B}\cdot\nabla\bar{B}
\bigg\}
\end{align*}
which together with the elliptic estimates implies that
\begin{align*}
\|\nabla\tilde{\Pi}\|_{L^{p}} \lesssim &\, \|\tilde{u}\cdot\nabla\tilde{u}\|_{L^{p}} + \|\tilde{u}\dot\nabla\bar{u}\|_{L^{p}}
+ \|\bar{u}\cdot\nabla\tilde{u}\|_{L^{p}}   \\
&\, + \|(\bar{a}+\tilde{a})(\Delta\tilde{u}-\nabla\tilde{\Pi})\|_{L^{p}} + \|(1+\bar{a}+\tilde{a})\tilde{B}\cdot\nabla\tilde{B}\|_{L^{p}} \\
&\, + \|(1+\bar{a}+\tilde{a})\tilde{B}\cdot\nabla\bar{B}\|_{L^{p}} + \|(1+\bar{a}+\tilde{a})\bar{B}\cdot\nabla\tilde{B}\|_{L^{p}}   \\
&\, + \|\tilde{a}\,(\Delta\bar{u}-\nabla\bar{\Pi})\|_{L^{p}} + \|\tilde{a}\,\bar{B}\cdot\nabla\bar{B}\|_{L^{p}}.
\end{align*}
As a consequence, we deduce from (\ref{u p}) that
\begin{align}
\label{u p final}
\begin{split}
\|\tilde{u}\|_{L_{t}^{\infty}(L^{p})} \lesssim &\, \|\tilde{u}\cdot\nabla\tilde{u}\|_{L_{t}^{1}(L^{p})} +
\|\tilde{u}\cdot\nabla\bar{u}\|_{L_{t}^{1}(L^{p})} + \|\bar{u}\cdot\nabla\tilde{u}\|_{L_{t}^{1}(L^{p})} \\
&\, + \|(\bar{a}+\tilde{a})(\Delta\tilde{u}-\nabla\tilde{\Pi})\|_{L_{t}^{1}(L^{p})} + \|\tilde{a}\,(\Delta\bar{u}-\nabla\bar{\Pi})\|_{L_{t}^{1}(L^{p})} \\
&\, + \|\tilde{a}\,\bar{B}\cdot\nabla\bar{B}\|_{L_{t}^{1}(L^{p})} + \|(1+\bar{a}+\tilde{a})\tilde{B}\cdot\nabla\tilde{B}\|_{L_{t}^{1}(L^{p})} \\
&\, + \|(1+\bar{a}+\tilde{a})\tilde{B}\cdot\nabla\bar{B}\|_{L_{t}^{1}(L^{p})}
+ \|(1+\bar{a}+\tilde{a})\bar{B}\cdot\nabla\tilde{B}\|_{L_{t}^{1}(L^{p})} \\
&\, + \|\tilde{u}_{0}\|_{L^{p}}.
\end{split}
\end{align}
However, since $p \in (1, \frac{6}{5})$, it follows from Proposition \ref{decay bar} and Theorem \ref{decay_main_theorem} that
\begin{align*}
&\, \|\tilde{u}\cdot\nabla\tilde{u}\|_{L_{t}^{1}(L^{p})} + \|\tilde{u}\cdot\nabla\bar{u}\|_{L_{t}^{1}(L^{p})} \\
\lesssim &\, \|\tilde{u}\|_{L_{t}^{\infty}(L^{\frac{2p}{2-p}})} \left( \|\nabla\tilde{u}\|_{L_{t}^{1}(L^{2})} + \|\nabla\bar{u}\|_{L_{t}^{1}(L^{2})} \right) \\
\lesssim &\, \|\tilde{u}\|_{L_{t}^{\infty}(H^{1/2})} \lesssim \delta_{0} \lesssim A_{0},
\end{align*}
where we used the Sobolev embedding inequality $\|u\|_{L^{\frac{2p}{2-p}}} \lesssim \|u\|_{H^{1/2}}$.
Similarly, we have
\begin{align*}
& \|\tilde{B}\cdot\nabla\tilde{B}\|_{L_{t}^{1}(L^{p})} + \|\tilde{B}\cdot\nabla\bar{B}\|_{L_{t}^{1}(L^{p})} \lesssim A_{0}, \\
& \|\tilde{u}\cdot\nabla\tilde{B}\|_{L_{t}^{1}(L^{p})} + \|\tilde{u}\cdot\nabla\bar{B}\|_{L_{t}^{1}(L^{p})} \lesssim A_{0}, \\
& \|\tilde{B}\cdot\nabla\tilde{u}\|_{L_{t}^{1}(L^{p})} + \|\tilde{B}\cdot\nabla\bar{u}\|_{L_{t}^{1}(L^{p})} \lesssim A_{0},
\end{align*}
and
\begin{align*}
&\, \|(1+\bar{a}+\tilde{a})\tilde{B}\cdot\nabla\tilde{B}\|_{L_{t}^{1}(L^{p})} + \|(1+\bar{a}+\tilde{a})\tilde{B}\cdot\nabla\bar{B}\|_{L_{t}^{1}(L^{p})} \\
\lesssim &\, \left( 1+\|\bar{a}\|_{L_{t}^{\infty}(L^{\infty})}+\|\tilde{a}\|_{L_{t}^{\infty}(L^{\infty})} \right)
\left( \|\tilde{B}\cdot\nabla\tilde{B}\|_{L_{t}^{1}(L^{p})} + \|\tilde{B}\cdot\nabla\bar{B}\|_{L_{t}^{1}(L^{p})} \right) \lesssim A_{0}.
\end{align*}
Also, we can obtain
\begin{align*}
\|\bar{u}\cdot\nabla\tilde{u}\|_{L_{t}^{1}(L^{p})} \lesssim &\, \|\bar{u}\|_{L_{t}^{\infty}(L^{\frac{2p}{2-p}})}\|\nabla\tilde{u}\|_{L_{t}^{1}(L^{2})}  \\
\lesssim &\, \|\bar{u}\|_{L_{t}^{\infty}(H^{1/2})}\|\nabla\tilde{u}\|_{L_{t}^{1}(L^{2})} \lesssim A_{0}.
\end{align*}
Similarly, we have
\begin{align*}
&\, \|\bar{B}\cdot\nabla\tilde{B}\|_{L_{t}^{1}(L^{p})} + \|\bar{B}\cdot\nabla\tilde{u}\|_{L_{t}^{1}(L^{p})}
+ \|\bar{u}\cdot\nabla\tilde{B}\|_{L_{t}^{1}(L^{p})} \lesssim A_{0}, \\
&\, \|(1+\bar{a}+\tilde{a})\bar{B}\cdot\nabla\tilde{B}\|_{L_{t}^{1}(L^{p})} \lesssim
(1 + \|\bar{a}\|_{L_{t}^{\infty}(L^{\infty})}   \\
&\, \quad\quad\quad\quad\quad\quad\quad\quad\quad\quad\quad\quad
+ \|\tilde{a}\|_{L_{t}^{\infty}(L^{\infty})})\|\bar{B}\|_{L_{t}^{\infty}(H^{1/2})}
\|\nabla\tilde{B}\|_{L_{t}^{1}(L^{2})} \lesssim A_{0},
\end{align*}
and
\begin{align*}
\|(\bar{a}+\tilde{a})(\Delta\tilde{u}-\nabla\tilde{\Pi})\|_{L_{t}^{1}(L^{p})} + \|\tilde{a}\,(\Delta\bar{u}-\nabla\bar{\Pi})\|_{L_{t}^{1}(L^{p})}
+ \|\tilde{a}\,\bar{B}\cdot\nabla\bar{B}\|_{L_{t}^{1}(L^{p})} \lesssim A_{0}.
\end{align*}
Plugging the above estimates into (\ref{b p}) and (\ref{u p final}), we arrive at
\begin{align*}
\|\tilde{u}\|_{L_{t}^{\infty}(L^{p})} + \|\tilde{B}\|_{L_{t}^{\infty}(L^{p})} \lesssim A_{0}.
\end{align*}
This along with (\ref{a u b}) complete the proof.
\end{proof}

\begin{proposition}
Under the assumption of Theorem \ref{stability main theorem}, there holds
\begin{align}
\|\tilde{a}(t)\|_{L^{\infty}} \leq C A_{0}
\end{align}
and
\begin{align}
\|\nabla\tilde{a}(t)\|_{L^{\infty}} \leq C \left( \|\nabla\tilde{a}_{0}\|_{L^{\infty}} + \|\bar{a}_{0}\|_{B_{2,1}^{7/2}} \right),
\end{align}
where $A_{0}$ defined as before.
\end{proposition}
\begin{proof}
Firstly, let us give the equation of $\tilde{a}$
\begin{align*}
\partial_{t}\tilde{a} + u\cdot\nabla\tilde{a} = -\tilde{u}\cdot\nabla\bar{a}.
\end{align*}
Using basic estimates about transport equations, we have
\begin{align*}
\|\tilde{a}(t)\|_{L^{\infty}} \leq C e^{C \int_{0}^{t}\|\nabla u\|_{L^{\infty}}dt'}\left( \|\tilde{a}_{0}\|_{L^{\infty}} +
\int_{0}^{t} \|\tilde{u}\cdot\nabla\bar{a}\|_{L^{\infty}}\,dt' \right)
\end{align*}
Moreover, we obtain
\begin{align*}
\|\tilde{a}(t)\|_{L^{\infty}} \lesssim \|\tilde{a}_{0}\|_{L^{\infty}} + \|\nabla\bar{a}\|_{L_{t}^{\infty}(\dot{B}_{2,1}^{3/2})} A_{0}.
\end{align*}
By the definition of $A_{0}$ and Proposition \ref{global reference 2}, we get
\begin{align*}
\|\tilde{a}(t)\|_{L^{\infty}} \leq C A_{0}.
\end{align*}
Similarly, we give the equation of $\nabla\tilde{a}$
\begin{align*}
\partial_{t}\nabla\tilde{a} + u\cdot\nabla(\nabla\tilde{a}) = -\tilde{u}\cdot\nabla(\nabla\bar{a}).
\end{align*}
Then classical estimates about transport equations, Proposition \ref{decay pro 4} and Proposition \ref{decay bar} gives
\begin{align*}
\|\nabla\tilde{a}(t)\|_{L^{\infty}} \lesssim \|\nabla\tilde{a}_{0}\|_{L^{\infty}} + \|\bar{a}\|_{L_{t}^{\infty}(\dot{B}_{2,1}^{7/2})}.
\end{align*}
Due to  Lemma \ref{transport_estimate}, we obtain
\begin{align*}
\|\bar{a}(t)\|_{B_{2,1}^{7/2}} \lesssim \|\bar{a}_{0}\|_{B_{2,1}^{7/2}},
\end{align*}
where we used (\ref{estimate in paper}).
From this, we easily obtain
\begin{align*}
\|\nabla\tilde{a}(t)\|_{L^{\infty}} \leq C \left( \|\nabla\tilde{a}_{0}\|_{L^{\infty}} + \|\bar{a}_{0}\|_{B_{2,1}^{7/2}} \right).
\end{align*}
\end{proof}

\begin{proposition}
\label{high per}
Under the assumptions of Theorem \ref{stability main theorem}, there holds
\begin{align*}
\sup_{t\geq t_{0}}\left( \|\nabla^{2}\tilde{u}(t)\|_{L^{2}}^{2} + \|\nabla^{2}\tilde{B}(t)\|_{L^{2}}^{2} \right) & +
\int_{t_{0}}^{\infty}\|\partial_{t}\nabla\tilde{u}(t)\|_{L^{2}}^{2} + \|\partial_{t}\nabla\tilde{B}(t)\|_{L^{2}}^{2}\,dt    \\
& + \int_{t_{0}}^{\infty}\|\nabla^{3}\tilde{u}(t)\|_{L^{2}}^{2} + \|\nabla^{3}\tilde{B}(t)\|_{L^{2}}^{2}\,dt \leq C,
\end{align*}
where $C$ depends on the initial data.
\end{proposition}
\begin{proof}
Taking derivative to the second and third equation of (\ref{perturbed}), we obtain
\begin{align}
\label{dup}
\begin{split}
\rho \partial_{t}\partial_{j}\tilde{u}^{i} & - \Delta \partial_{j}\tilde{u}^{i} + \partial_{j}\partial_{i}\tilde{\Pi}
= - \partial_{j}\rho\,\partial_{t}\tilde{u}^{i} - \partial_{j}\rho\,u\cdot\nabla\tilde{u}^{i} - \rho \, \partial_{j}u\cdot\nabla\tilde{u}^{i}   \\
& -\rho\,u\cdot\nabla\partial_{j}\tilde{u}^{i} + \partial_{j}B\cdot\nabla\tilde{B}^{i} + B\cdot\nabla\partial_{j}\tilde{B}^{i} -
\partial_{j}\rho \,\tilde{u}\cdot\nabla\bar{u}^{i} - \rho \, \partial_{j}\tilde{u}\cdot\nabla\bar{u}^{i}    \\
& - \rho\,\tilde{u}\cdot\nabla\partial_{j}\bar{u}^{i} + \partial_{j}\tilde{B}\cdot\nabla\bar{B}^{i} + \tilde{B}\cdot\nabla\partial_{j}\bar{B}^{i}
- \frac{\tilde{\rho}}{\bar{\rho}}(\Delta\partial_{j}\bar{u}^{i}-\partial_{j}\partial_{i}\tilde{\Pi})    \\
& - \partial_{j}\left( \frac{\tilde{\rho}}{\bar{\rho}} \right)(\Delta\bar{u}^{i}-\partial_{i}\tilde{\Pi})
- \frac{\tilde{\rho}}{\bar{\rho}}\,\partial_{j}\bar{B}\cdot\nabla\bar{B}^{i} - \frac{\tilde{\rho}}{\bar{\rho}} \, \bar{B}\cdot\nabla\partial_{j}\bar{B}^{i} \\
& - \partial_{j}\left( \frac{\tilde{\rho}}{\bar{\rho}} \right)\bar{B}\cdot\nabla\bar{B}^{i}
\end{split}
\end{align}
and
\begin{align}
\label{dbp}
\begin{split}
\partial_{t}\partial_{j}\tilde{B}^{i} - \Delta\partial_{i}\tilde{B}^{i} & = - \partial_{j}u\cdot\nabla\tilde{B}^{i} - u\cdot\nabla\partial_{j}\tilde{B}^{i}
+ \partial_{j}B\cdot\nabla\tilde{u}^{i} + B\cdot\nabla\partial_{j}\tilde{u}^{i} \\
& - \partial_{j}\tilde{B}\cdot\nabla\bar{u}^{i} - \tilde{B}\cdot\nabla\partial_{j}\bar{u}^{i} + \partial_{j}\tilde{u}\cdot\nabla\bar{B}^{i}
+ \tilde{u}\cdot\nabla\partial_{j}\bar{B}^{i}.
\end{split}
\end{align}
Multiplying (\ref{dup}) with $\partial_{t}\partial_{j}\tilde{u}^{i}$ and integrating over $\mathbb{R}^{3}$, we obtain
\begin{align}
\label{dup e}
\begin{split}
& \|\sqrt{\rho}\partial_{t}\nabla\tilde{u}(t)\|_{L^{2}}^{2} + \frac{1}{2}\frac{d}{dt}\|\nabla^{2}\tilde{u}(t)\|_{L^{2}}^{2}
= -\int_{\mathbb{R}^{3}} \nabla\rho\, \partial_{t}\tilde{u}\,\partial_{t}\nabla\tilde{u}\,dx    \\
&\quad - \int_{\mathbb{R}^{3}}\nabla\rho\,\tilde{u}\cdot\nabla\tilde{u}\,\partial_{t}\nabla\tilde{u}\,dx
 - \int_{\mathbb{R}^{3}}\nabla\rho\,\bar{u}\cdot\nabla\tilde{u}\,\partial_{t}\nabla\tilde{u}\,dx   \\
&\quad  - \int_{\mathbb{R}^{3}}\rho\,\nabla\tilde{u}\cdot\nabla\tilde{u}\,\partial_{t}\nabla\tilde{u}\,dx
- \int_{\mathbb{R}^{3}} \rho\,\nabla\bar{u}\cdot\nabla\tilde{u}\,\partial_{t}\nabla\tilde{u}\,dx    \\
&\quad - \int{\mathbb{R}^{3}} \rho\,u\cdot\nabla\nabla\tilde{u}\,\partial_{t}\nabla\tilde{u}\,dx
+ \int_{\mathbb{R}^{3}}\nabla\tilde{B}\cdot\nabla\tilde{B}\,\partial_{t}\nabla\tilde{u}\,dx    \\
&\quad + \int_{\mathbb{R}^{3}} \nabla\bar{B}\cdot\nabla\tilde{B}\,\partial_{t}\nabla\tilde{u}\,dx
+ \int_{\mathbb{R}^{3}} \tilde{B}\cdot\nabla\nabla\tilde{B}\,\partial_{t}\nabla\tilde{u}\,dx  \\
&\quad  + \int_{\mathbb{R}^{3}} \bar{B}\cdot\nabla\nabla\tilde{B}\,\partial_{t}\nabla\tilde{u}\,dx
- \int_{\mathbb{R}^{3}}\nabla\rho\,\tilde{u}\cdot\nabla\bar{u}\,\partial_{t}\nabla\tilde{u}\,dx \\
&\quad - \int_{\mathbb{R}^{3}} \rho\,\nabla\tilde{u}\cdot\nabla\bar{u}\,\partial_{t}\nabla\tilde{u}\,dx
- \int_{\mathbb{R}^{3}} \rho\,\tilde{u}\cdot\nabla\nabla\bar{u}\,dx \\
&\quad + \int_{\mathbb{R}^{3}} \nabla\tilde{B}\cdot\nabla\bar{B}\,\partial_{t}\nabla\tilde{u}\,dx
+ \int_{\mathbb{R}^{3}} \tilde{B}\cdot\nabla\nabla\bar{B}\,\partial_{t}\nabla\tilde{u}\,dx      \\
&\quad - \int_{\mathbb{R}^{3}} \frac{\tilde{\rho}}{\bar{\rho}}(\Delta\partial_{j}\bar{u}^{i}-\partial_{j}\partial_{i}\tilde{\Pi})\,\partial_{t}\nabla\tilde{u}\,dx
- \int_{\mathbb{R}^{3}} \partial_{j}\left(\frac{\tilde{\rho}}{\bar{\rho}}\right)(\Delta\bar{u}^{i}-\partial_{i}\tilde{\Pi})\,\partial_{t}\nabla\tilde{u}\,dx \\
&\quad - \int_{\mathbb{R}^{3}} \frac{\tilde{\rho}}{\bar{\rho}}\,\nabla\bar{B}\cdot\nabla\bar{B}\,dx
- \int_{\mathbb{R}^{3}} \frac{\tilde{\rho}}{\bar{\rho}} \,\bar{B}\cdot\nabla\nabla\bar{B}\,\partial_{t}\nabla\tilde{u}\,dx  \\
&\quad - \int_{\mathbb{R}^{3}} \partial_{j}\left( \frac{\tilde{\rho}}{\bar{\rho}} \right)\,\bar{B}\cdot\nabla\bar{B}\,dx.
\end{split}
\end{align}
For $\tilde{B}$, using same process, we have
\begin{align}
\label{dbp e}
\begin{split}
\|\partial_{t}\nabla\tilde{B}(t)\|_{L^{2}}^{2} & + \frac{1}{2}\frac{d}{dt} \|\nabla^{2}\tilde{B}(t)\|_{L^{2}}^{2}
= - \int_{\mathbb{R}^{3}} \nabla u \cdot\nabla\tilde{B}\,\partial_{t}\tilde{B}\,dx  \\
&\quad -\int_{\mathbb{R}^{3}}u\cdot\nabla\nabla\tilde{B}\,\partial_{t}\nabla\tilde{B}\,dx
+ \int_{\mathbb{R}^{3}} \nabla B\cdot\nabla\tilde{u}\,\partial_{t}\nabla\tilde{B}\,dx   \\
&\quad + \int_{\mathbb{R}^{3}} B\cdot\nabla\nabla\tilde{u}\,\partial_{t}\nabla\tilde{B}\,dx
- \int_{\mathbb{R}^{3}} \nabla\tilde{B}\cdot\nabla\bar{u}\,\partial_{t}\nabla\tilde{B}\,dx  \\
&\quad - \int_{\mathbb{R}^{3}} \tilde{B}\cdot\nabla\nabla\bar{u}\,\partial_{t}\nabla\tilde{B}\,dx
+ \int_{\mathbb{R}^{3}}\nabla\tilde{u}\cdot\nabla\bar{B}\,\partial_{t}\nabla\tilde{B}\,dx \\
&\quad + \int_{\mathbb{R}^{3}} \tilde{u}\cdot\nabla\nabla\bar{B}\,\partial_{t}\nabla\tilde{B}\,dx.
\end{split}
\end{align}
Combining (\ref{dup e}) and (\ref{dbp e}), we obtain
{ \allowdisplaybreaks
\begin{align*}
%\label{start1}
%\begin{split}
&\, \frac{1}{2}\frac{d}{dt}\left( \|\nabla^{2}\tilde{u}(t)\|_{L^{2}}^{2} + \|\nabla^{2}\tilde{B}(t)\|_{L^{2}}^{2} \right)
+ \|\sqrt{\rho}\,\partial_{t}\nabla\tilde{u}\|_{L^{2}}^{2} + \|\partial_{t}\nabla\tilde{B}\|_{L^{2}}^{2}    \\
\leq &\, \|\partial_{t}\tilde{u}\|_{L^{2}}^{2} + \left( \|\tilde{u}\cdot\nabla\tilde{u}\|_{L^{2}}
+ \|\bar{u}\cdot\nabla\tilde{u}\|_{L^{2}} + \|\tilde{u}\cdot\nabla\bar{u}\|_{L^{2}} \right)^{2}     \\
&\, + \|\nabla\tilde{u}\|_{L^{2}}^{2}\|\nabla^{2}\tilde{u}\|_{L^{2}}\|\nabla^{3}\tilde{u}\|_{L^{2}}
+ \|\nabla\bar{u}\|_{L^{2}}^{2}\|\nabla^{2}\tilde{u}\|_{L^{2}}\|\nabla^{3}\tilde{u}\|_{L^{2}}   \\
&\, + \|\tilde{u}\|_{L^{2}}\|\nabla\tilde{u}\|_{L^{2}}\|\nabla^{3}\tilde{u}\|_{L^{2}}^{2}
+ \|\nabla\bar{u}\|_{L^{2}}\|\nabla^{2}\bar{u}\|_{L^{2}}\|\nabla^{2}\tilde{u}\|_{L^{2}}^{2} \\
&\, + \|\nabla\tilde{B}\|_{L^{2}}^{2}\|\nabla^{2}\tilde{B}\|_{L^{2}}\|\nabla^{3}\tilde{B}\|_{L^{2}}
+ \|\nabla\bar{B}\|_{L^{2}}^{2}\|\nabla^{2}\tilde{B}\|_{L^{2}}\|\nabla^{3}\tilde{B}\|_{L^{2}} \\
&\, + \|\tilde{B}\|_{L^{2}}\|\nabla\tilde{B}\|_{L^{2}}\|\nabla^{3}\tilde{B}\|_{L^{2}}^{2}
+ \|\nabla\bar{B}\|_{L^{2}}\|\nabla^{2}\bar{B}\|_{L^{2}}\|\nabla^{2}\tilde{B}\|_{L^{2}}^{2} \\
&\, + \|\nabla\bar{u}\|_{L^{2}}^{2}\|\nabla^{2}\tilde{u}\|_{L^{2}}\|\nabla^{3}\tilde{u}\|_{L^{2}}
+ \|\tilde{u}\|_{L^{2}}\|\nabla\tilde{u}\|_{L^{2}}\|\nabla^{3}\bar{u}\|_{L^{2}}^{2}   \\
&\, \|\nabla\bar{B}\|_{L^{2}}^{2}\|\nabla^{2}\tilde{B}\|_{L^{2}}\|\nabla^{3}\tilde{B}\|_{L^{2}}
+ \|\tilde{B}\|_{L^{2}}\|\nabla\tilde{B}\|_{L^{2}}\|\nabla^{3}\bar{B}\|_{L^{2}}^{2} \\
&\, + \|\tilde{\rho}\|_{L^{\infty}}^{2}\|\nabla(\Delta\bar{u}-\nabla\bar{\Pi})\|_{L^{2}}^{2}
+ \|\nabla\tilde{a}\|_{L^{\infty}}^{2}\|\Delta\bar{u}-\nabla\bar{\Pi}\|_{L^{2}}^{2} \\
&\, + \|\tilde{\rho}\|_{L^{\infty}}^{2}\|\nabla\bar{B}\cdot\nabla\bar{B}\|_{L^{2}}^{2}
+ \|\tilde{\rho}\|_{L^{\infty}}^{2}\|\bar{B}\cdot\nabla\nabla\bar{B}\|_{L^{2}}^{2}  \\
&\, + \|\nabla\tilde{a}\|_{L^{\infty}}^{2}\|\bar{B}\cdot\nabla\bar{B}\|_{L^{2}}^{2}
+\|\nabla^{2}\tilde{u}\|_{L^{2}}\|\nabla^{3}\tilde{u}\|_{L^{2}}\|\nabla\tilde{B}\|_{L^{2}}^{2} \\
&\, + \|\nabla\bar{u}\|_{L^{2}}^{2}\|\nabla^{2}\tilde{B}\|_{L^{2}}\|\nabla^{3}\tilde{B}\|_{L^{2}}
+ \|\tilde{u}\|_{L^{2}}\|\nabla\tilde{u}\|_{L^{2}}\|\nabla^{3}\tilde{B}\|_{L^{2}}^{2}   \\
&\, + \|\nabla\bar{u}\|_{L^{2}}\|\nabla^{2}\bar{u}\|_{L^{2}}\|\nabla^{2}\tilde{B}\|_{L^{2}}^{2}
+ \|\nabla^{2}\tilde{B}\|_{L^{2}}\|\nabla^{3}\tilde{B}\|_{L^{2}}\|\nabla\tilde{u}\|_{L^{2}}^{2} \\
&\, + \|\nabla\bar{B}\|_{L^{2}}^{2}\|\nabla^{2}\tilde{u}\|_{L^{2}}\|\nabla^{3}\tilde{u}\|_{L^{2}}
+ \|\tilde{B}\|_{L^{2}}\|\nabla\tilde{B}\|_{L^{2}}\|\nabla^{3}\tilde{u}\|_{L^{2}}^{2}  \\
&\, + \|\nabla\bar{B}\|_{L^{2}}\|\nabla^{2}\bar{B}\|_{L^{2}}\|\nabla^{2}\tilde{u}\|_{L^{2}}^{2}
+ \|\nabla\bar{u}\|_{L^{2}}^{2}\|\nabla^{2}\tilde{B}\|_{L^{2}}\|\nabla^{3}\tilde{B}\|_{L^{2}}   \\
&\, + \|\tilde{B}\|_{L^{2}}\|\nabla\tilde{B}\|_{L^{2}}\|\nabla^{3}\bar{u}\|_{L^{2}}^{2}
+ \|\nabla\bar{B}\|_{L^{2}}^{2}\|\nabla^{2}\tilde{u}\|_{L^{2}}\|\nabla^{3}\tilde{u}\|_{L^{2}}   \\
&\, + \|\tilde{u}\|_{L^{2}}\|\nabla\tilde{B}\|_{L^{2}}\|\nabla^{3}\bar{B}\|_{L^{2}}^{2}.
%\end{split}
\end{align*}
Multiplying $\frac{1}{\rho}\,\Delta\tilde{u}$ and $\Delta\tilde{u}$ to (\ref{dup}) and (\ref{dbp}) separately and
doing some basic energy estimates, we obtain
\begin{align*}
%\label{start2}
%\begin{split}
&\, \frac{1}{2}\frac{d}{dt}\left( \|\nabla^{2}\tilde{u}\|_{L^{2}}^{2} + \|\nabla^{2}\tilde{B}\|_{L^{2}}^{2} \right)
+ \|\rho\|_{L^{\infty}}^{-1}\|\nabla^{3}\tilde{u}\|_{L^{2}}^{2} + \|\nabla^{3}\tilde{B}\|_{L^{2}}^{2}   \\
\leq &\, \|\partial_{t}\nabla\tilde{u}\|_{L^{2}}^{2} +
\left( \|\tilde{u}\cdot\nabla\tilde{u}\|_{L^{2}} + \|\bar{u}\cdot\nabla\tilde{u}\|_{L^{2}} + \|\tilde{u}\cdot\nabla\bar{u}\|_{L^{2}} \right)^{2} \\
&\, + \|\nabla^{2}\tilde{u}\|_{L^{2}}\|\nabla^{3}\tilde{u}\|_{L^{2}}\|\nabla\tilde{u}\|_{L^{2}}^{2}
+ \|\nabla^{2}\tilde{u}\|_{L^{2}}\|\nabla^{3}\tilde{u}\|_{L^{2}}\|\nabla\bar{u}\|_{L^{2}}^{2}   \\
&\, + \|\tilde{u}\|_{L^{2}}\|\nabla\tilde{u}\|_{L^{2}}\|\nabla^{3}\tilde{u}\|_{L^{2}}^{2}
+ \|\nabla\bar{u}\|_{L^{2}}\|\nabla^{2}\bar{u}\|_{L^{2}}\|\nabla^{2}\tilde{u}\|_{L^{2}}^{2}    \\
&\, + \|\nabla^{2}\tilde{B}\|_{L^{2}}\|\nabla^{3}\tilde{B}\|_{L^{2}}\|\nabla\tilde{B}\|_{L^{2}}^{2}
+ \|\nabla^{2}\tilde{B}\|_{L^{2}}\|\nabla^{3}\tilde{B}\|_{L^{2}}\|\nabla\bar{B}\|_{L^{2}}^{2}   \\
&\, + \|\tilde{B}\|_{L^{2}}\|\nabla\tilde{B}\|_{L^{2}}\|\nabla^{3}\tilde{B}\|_{L^{2}}^{2}
+ \|\nabla\bar{B}\|_{L^{2}}\|\nabla^{2}\bar{B}\|_{L^{2}}\|\nabla^{2}\tilde{B}\|_{L^{2}}^{2} \\
&\, + \|\nabla^{2}\tilde{u}\|_{L^{2}}\|\nabla^{3}\tilde{u}\|_{L^{3}}\|\nabla\bar{u}\|_{L^{2}}^{2}
+ \|\tilde{u}\|_{L^{2}}\|\nabla\tilde{u}\|_{L^{2}}\|\nabla^{3}\bar{u}\|_{L^{2}}^{2} \\
&\, + \|\nabla^{2}\tilde{B}\|_{L^{2}}\|\nabla^{3}\tilde{B}\|_{L^{2}}\|\nabla\bar{B}\|_{L^{2}}^{2}
+ \|\tilde{B}\|_{L^{2}}\|\nabla\tilde{B}\|_{L^{2}}\|\nabla^{3}\bar{B}\|_{L^{2}}^{2} \\
&\, + \|\tilde{\rho}\|_{L^{\infty}}^{2}\|\nabla(\Delta\bar{u}-\nabla\bar{\Pi})\|_{L^{2}}^{2}
+ \|\nabla\tilde{a}\|_{L^{\infty}}^{2}\|\Delta\bar{u}-\nabla\bar{\Pi}\|_{L^{2}}^{2} \\
&\, + \|\tilde{\rho}\|_{L^{\infty}}^{2}\|\nabla\bar{B}\cdot\nabla\bar{B}\|_{L^{2}}^{2}
+ \|\tilde{\rho}\|_{L^{\infty}}^{2}\|\bar{B}\cdot\nabla\nabla\bar{B}\|_{L^{2}}^{2}  \\
&\, + \|\nabla\tilde{a}\|_{L^{\infty}}^{2}\|\bar{B}\cdot\nabla\bar{B}\|_{L^{2}}^{2}
+ \|\nabla^{2}\tilde{u}\|_{L^{2}}\|\nabla^{3}\tilde{u}\|_{L^{2}}\|\nabla\tilde{B}\|_{L^{2}}^{2}  \\
&\, + \|\nabla^{2}\tilde{B}\|_{L^{2}}\|\nabla^{3}\tilde{B}\|_{L^{2}}\|\nabla\bar{u}\|_{L^{2}}^{2}
+ \|\tilde{u}\|_{L^{2}}\|\nabla\tilde{u}\|_{L^{2}}\|\nabla^{3}\tilde{B}\|_{L^{2}}^{2}  \\
&\, + \|\nabla\bar{u}\|_{L^{2}}\|\nabla^{2}\bar{u}\|_{L^{2}}\|\nabla^{2}\tilde{B}\|_{L^{2}}^{2}
+ \|\nabla^{2}\tilde{B}\|_{L^{2}}\|\nabla^{3}\tilde{B}\|_{L^{2}}\|\nabla\tilde{u}\|_{L^{2}}^{2} \\
&\, + \|\nabla^{2}\tilde{u}\|_{L^{2}}\|\nabla^{3}\tilde{u}\|_{L^{2}}\|\nabla\bar{B}\|_{L^{2}}^{2}
+ \|\tilde{B}\|_{L^{2}}\|\nabla\tilde{B}\|_{L^{2}}\|\nabla^{3}\tilde{u}\|_{L^{2}}^{2}   \\
&\, + \|\nabla\bar{B}\|_{L^{2}}\|\nabla^{2}\bar{B}\|_{L^{2}}\|\nabla^{2}\tilde{B}\|_{L^{2}}^{2}
+ \|\nabla^{2}\tilde{B}\|_{L^{2}}\|\nabla^{3}\tilde{B}\|_{L^{2}}\|\nabla\bar{u}\|_{L^{2}}^{2} \\
&\, + \|\tilde{B}\|_{L^{2}}\|\nabla\tilde{B}\|_{L^{2}}\|\nabla^{3}\bar{u}\|_{L^{2}}^{2}
+ \|\nabla^{2}\tilde{u}\|_{L^{2}}\|\nabla^{3}\tilde{u}\|_{L^{2}}\|\nabla\bar{B}\|_{L^{2}}^{2}   \\
&\, + \|\tilde{u}\|_{L^{2}}\|\nabla\tilde{u}\|_{L^{2}}\|\nabla^{3}\bar{B}\|_{L^{2}}^{2}.
%\end{split}
\end{align*}
}
At this point, we got two completed inequalities. The second inequality times a small number then plus the first inequality yieds
{ \allowdisplaybreaks
\begin{align*}
&\, \frac{d}{dt}\left( \|\nabla^{2}\tilde{u}(t)\|_{L^{2}}^{2} + \|\nabla^{2}\tilde{B}(t)\|_{L^{2}}^{2} \right)
+ \|\sqrt{\rho}\partial_{t}\nabla\tilde{u}\|_{L^{2}}^{2} + \|\partial_{t}\nabla\tilde{B}\|_{L^{2}}^{2}  \\
&\, + \left( \frac{c_{0}}{2} - C \|\tilde{u}\|_{L^{2}}\|\nabla\tilde{u}\|_{L^{2}} - C \|\tilde{B}\|_{L^{2}}\|\nabla\tilde{B}\|_{L^{2}} \right)
\left( \|\nabla^{3}\tilde{u}\|_{L^{2}}^{2} + \|\nabla^{3}\tilde{B}\|_{L^{2}}^{2} \right) \\
\leq &\, \left( \|\nabla\bar{u}\|_{L^{2}}\|\nabla^{2}\bar{u}\|_{L^{2}} + \|\nabla\bar{B}\|_{L^{2}}\|\nabla^{2}\bar{B}\|_{L^{2}} \right)
\left( \|\nabla^{2}\tilde{u}\|_{L^{2}}^{2} + \|\nabla^{2}\tilde{B}\|_{L^{2}}^{2} \right)    \\
&\, + \Big( \|\nabla\tilde{u}\|_{L^{2}}^{4} + \|\nabla\tilde{B}\|_{L^{2}}^{4} + \|\nabla\bar{u}\|_{L^{2}}^{4}
+ \|\nabla\bar{B}\|_{L^{2}}^{4} \Big)\Big( \|\nabla^{2}\tilde{u}\|_{L^{2}}^{2} + \|\nabla^{2}\tilde{B}\|_{L^{2}}^{2} \\
&\, + \|\nabla^{2}\bar{u}\|_{L^{2}}^{2} + \|\nabla^{2}\bar{B}\|_{L^{2}}^{2} \Big) + \|\partial_{t}\tilde{u}\|_{L^{2}}^{2}
+ \|\tilde{u}\|_{L^{2}}\|\nabla\tilde{u}\|_{L^{2}}\|\nabla^{2}\tilde{u}\|_{L^{2}}^{2}   \\
&\, +\|\nabla\tilde{u}\|_{L^{2}}^{2}\|\nabla\bar{u}\|_{L^{2}}\|\nabla^{2}\bar{u}\|_{L^{2}}
+ \Big( \|\tilde{u}\|_{L^{2}}\|\nabla\tilde{u}\|_{L^{2}} + \|\tilde{B}\|_{L^{2}}\|\nabla\tilde{B}\|_{L^{2}} \Big)
\Big( \|\nabla^{3}\bar{u}\|_{L^{2}}^{2} \\
&\, + \|\nabla^{3}\bar{B}\|_{L^{2}}^{2} \Big) + \|\tilde{\rho}\|_{L^{\infty}}^{2}\|\nabla(\Delta\bar{u}-\nabla\bar{\Pi})\|_{L^{2}}^{2}
+ \|\nabla\tilde{a}\|_{L^{\infty}}^{2}\|\Delta\bar{u}-\nabla\bar{\Pi}\|_{L^{2}}^{2} \\
&\, + \|\tilde{\rho}\|_{L^{\infty}}^{2}\|\nabla\bar{B}\cdot\nabla\bar{B}\|_{L^{2}}^{2}
+ \|\tilde{\rho}\|_{L^{\infty}}^{2}\|\bar{B}\cdot\nabla\nabla\bar{B}\|_{L^{2}}^{2}
+ \|\nabla\tilde{a}\|_{L^{\infty}}^{2}\|\bar{B}\cdot\nabla\bar{B}\|_{L^{2}}^{2}.
\end{align*}
}
Taking $c$ in Theorem \ref{stability main theorem} small enough, we have
{ \allowdisplaybreaks
\begin{align*}
&\, \frac{d}{dt}\left( \|\nabla^{2}\tilde{u}(t)\|_{L^{2}}^{2} + \|\nabla^{2}\tilde{B}(t)\|_{L^{2}}^{2} \right)
+ \|\sqrt{\rho}\partial_{t}\nabla\tilde{u}\|_{L^{2}}^{2} + \|\partial_{t}\nabla\tilde{B}\|_{L^{2}}^{2}  \\
&\,\quad + \|\nabla^{3}\tilde{u}\|_{L^{2}}^{2} + \|\nabla^{3}\tilde{B}\|_{L^{2}}^{2}  \\
\leq &\, \left( \|\nabla\bar{u}\|_{L^{2}}\|\nabla^{2}\bar{u}\|_{L^{2}} + \|\nabla\bar{B}\|_{L^{2}}\|\nabla^{2}\bar{B}\|_{L^{2}} \right)
\left( \|\nabla^{2}\tilde{u}\|_{L^{2}}^{2} + \|\nabla^{2}\tilde{B}\|_{L^{2}}^{2} \right)    \\
&\, + \Big( \|\nabla\tilde{u}\|_{L^{2}}^{4} + \|\nabla\tilde{B}\|_{L^{2}}^{4} + \|\nabla\bar{u}\|_{L^{2}}^{4}
+ \|\nabla\bar{B}\|_{L^{2}}^{4} \Big)\Big( \|\nabla^{2}\tilde{u}\|_{L^{2}}^{2} + \|\nabla^{2}\tilde{B}\|_{L^{2}}^{2} \\
&\, + \|\nabla^{2}\bar{u}\|_{L^{2}}^{2} + \|\nabla^{2}\bar{B}\|_{L^{2}}^{2} \Big) + \|\partial_{t}\tilde{u}\|_{L^{2}}^{2}
+ \|\tilde{u}\|_{L^{2}}\|\nabla\tilde{u}\|_{L^{2}}\|\nabla^{2}\tilde{u}\|_{L^{2}}^{2}   \\
&\, +\|\nabla\tilde{u}\|_{L^{2}}^{2}\|\nabla\bar{u}\|_{L^{2}}\|\nabla^{2}\bar{u}\|_{L^{2}}
+ \Big( \|\tilde{u}\|_{L^{2}}\|\nabla\tilde{u}\|_{L^{2}} + \|\tilde{B}\|_{L^{2}}\|\nabla\tilde{B}\|_{L^{2}} \Big)
\Big( \|\nabla^{3}\bar{u}\|_{L^{2}}^{2} \\
&\, + \|\nabla^{3}\bar{B}\|_{L^{2}}^{2} \Big) + \|\tilde{\rho}\|_{L^{\infty}}^{2}\|\nabla(\Delta\bar{u}-\nabla\bar{\Pi})\|_{L^{2}}^{2}
+ \|\nabla\tilde{a}\|_{L^{\infty}}^{2}\|\Delta\bar{u}-\nabla\bar{\Pi}\|_{L^{2}}^{2} \\
&\, + \|\tilde{\rho}\|_{L^{\infty}}^{2}\|\nabla\bar{B}\cdot\nabla\bar{B}\|_{L^{2}}^{2}
+ \|\tilde{\rho}\|_{L^{\infty}}^{2}\|\bar{B}\cdot\nabla\nabla\bar{B}\|_{L^{2}}^{2}
+ \|\nabla\tilde{a}\|_{L^{\infty}}^{2}\|\bar{B}\cdot\nabla\bar{B}\|_{L^{2}}^{2}.
\end{align*}
}
Integrating the above inequality, using decay estimates about reference solution and perturbed solution, we obtain
\begin{align*}
&\, \sup_{t\geq t_{0}}\Big( \|\nabla^{2}\tilde{u}(t)\|_{L^{2}}^{2} + \|\nabla^{2}\tilde{B}(t)\|_{L^{2}}^{2} \Big)
+ \int_{t_{0}}^{\infty} \|\sqrt{\rho}\,\partial_{t}\nabla\tilde{u}\|_{L^{2}}^{2} + \|\partial_{t}\nabla\tilde{B}\|_{L^{2}}^{2}\,dt  \\
&\,\quad\quad\quad\quad\quad\quad\quad\quad\quad\quad\quad
+ \int_{t_{0}}^{\infty}\|\nabla^{3}\tilde{u}\|_{L^{2}}^{2} + \|\nabla^{3}\tilde{B}\|_{L^{2}}^{2} \, dt \leq C.
\end{align*}
Hence, the proof of Proposition \ref{high per} is completed.
\end{proof}

Now, we can complete the proof of Theorem \ref{stability main theorem} as following.

\begin{proof}
According to the statement at the beginning of this section, given initial data $(\bar{a}_{0}+\tilde{a}_{0}, \bar{u}_{0}+\tilde{u}_{0}, \bar{B}_{0}+\tilde{B}_{0})$, (\ref{mhd_a}) has a unique solution $(a, u, B)$ on $[0, \tilde{T}^{*})$ such that
\begin{align*}
& a \in C([0, \tilde{T}^{*}); B_{2,1}^{7/2}(\mathbb{R}^{3})), \\
& u \in C([0, \tilde{T}^{*}); B_{2,1}^{2}) \cap L_{\mathrm{loc}}^{1}((0,\tilde{T}^{*}); \dot{B}_{2,1}^{4}(\mathbb{R}^{3})),  \\
& B \in C([0, \tilde{T}^{*}); B_{2,1}^{2}) \cap L_{\mathrm{loc}}^{1}((0,\tilde{T}^{*}); \dot{B}_{2,1}^{4}(\mathbb{R}^{3})).
\end{align*}
We need only prove the maximal existence time $\tilde{T}^{*} = \infty$. Indeed, according to all the decay estimates for reference solution and
perturbed solution, we repeat the argument used in the proof of Proposition \ref{global reference 5/2} and Proposition \ref{global reference 2}
to prove that $\tilde{T}^{*} = \infty$. Then a standard interpolation argument gives (\ref{1 3 1}) and (\ref{1 3 2}).
This completes the proof of Theorem \ref{stability main theorem}.
\end{proof}

%%%%%%%%%%%%%%%%%%%%%%%%%%%%%%%%%%%%%%%%%%%%%%%%%%%%%%%%%%%%%%%%%%%%%%%%%%%%%%%%%%%%%%%%%%%%%%%%%%%%%%%%%%%%%%%%%%%%%%%%%%%%%%%%%%%%%%%%%%%%%%%%%%%%%%%%%%%

\section{Appendix}

For completeness, in this section, we give the proof of some Lemmas and Propositions.
Firstly, we give the proof of Remark \ref{s=1 linear couple}.
\begin{proof}
Using the same method as in the proof of Lemma \ref{linear estimate momentum}, we can get the result.
The only difference is that we allow the parameter $s=1$, so we only give the estimates related to $s<1$ which is required in
the proof of Lemma \ref{linear estimate momentum}.
Notice that
\begin{align*}
\mathrm{div}[\Delta_{q}, a]\nabla u = &\, \Delta_{q}\mathrm{div}\,\mathcal{R}(a, \nabla u) + \Delta_{q}\mathrm{div}\, T_{\nabla u}a \\
&\, -\mathrm{div}R(a,\Delta_{q}\nabla u) - \mathrm{div}[T_{a},\Delta_{q}]\nabla u.
\end{align*}
Applying Lemma \ref{bernsteininequality}, we have
\begin{align*}
&\, \|\Delta_{q}\mathrm{div}\,\mathcal{R}(a,\nabla u)(t)\|_{L^{2}} \\
\lesssim &\, 2^{\frac{5}{2}q}\sum_{k \geq q-5} \|\Delta_{k}a(t)\|_{L^{2}}\|\tilde{\Delta}_{k}\nabla u(t)\|_{L^{2}}  \\
\lesssim &\, 2^{\frac{5}{2}q}\sum_{k \geq q-5} 2^{-2k}2^{-k-\frac{1}{2}k} 2^{2k}\|\Delta_{k}a(t)\|_{L^{2}}2^{k+\frac{3}{2}k}\|\tilde{\Delta}_{k}u(t)\|_{L^{2}} \\
\lesssim &\, 2^{\frac{5}{2}q}\sum_{k \geq q-5} d_{k}(t) 2^{-k(1+\frac{5}{2})}\|a(t)\|_{\dot{B}_{2,1}^{2}}\|u(t)\|_{\dot{B}_{2,1}^{5/2}}   \\
\lesssim &\, d_{q}(t) 2^{-q} \|a(t)\|_{\dot{B}_{2,1}^{2}}\|u(t)\|_{\dot{B}_{2,1}^{5/2}},
\end{align*}
and
\begin{align*}
&\, \|\mathrm{div}\,R(a,\Delta_{q}\nabla u)(t)\|_{L^{2}}    \\
\lesssim &\, \sum_{|k-q|\leq 1}\|\nabla\Delta_{k}a(t)\|_{L^{2}}\|\tilde{\Delta}_{k}\Delta_{q}\nabla u(t)\|_{L^{\infty}}  \\
&\, \quad\quad\quad\quad\quad\quad
+ \sum_{|k-q|\leq 1} \|\Delta_{k}a(t)\|_{L^{2}}\|\tilde{\Delta}_{k}\Delta_{q}\Delta u(t)\|_{L^{\infty}}  \\
\lesssim &\, \sum_{|k-q|\leq 1}2^{k}\|\Delta_{k}a(t)\|_{L^{2}}2^{q\frac{5}{2}}\|\tilde{\Delta}_{k}u(t)\|_{L^{2}}   \\
\lesssim &\, d_{q}(t)2^{-q}\|a(t)\|_{\dot{B}_{2,1}^{2}}\|u(t)\|_{\dot{B}_{2,1}^{5/2}}.
\end{align*}
Note that
\begin{align*}
\|S_{k-1}\nabla u(t)\|_{L^{\infty}} \lesssim &\, \sum_{j\leq k-2}\|\Delta_{j}\nabla u(t)\|_{L^{\infty}} \\
\lesssim &\, \sum_{j\leq k-2}2^{\frac{5}{2}j}\|\Delta_{j}u(t)\|_{L^{2}}  \\
\lesssim &\, \|u(t)\|_{\dot{B}_{2,1}^{5/2}},
\end{align*}
this along with Lemma \ref{bernsteininequality} leads to
\begin{align*}
\|\Delta_{q}\mathrm{div}\,T_{\nabla u}a(t)\|_{L^{2}} \lesssim &\, 2^{q}\sum_{|k-q|\leq 4}\|\Delta_{k}a(t)\|_{L^{2}}\|S_{k-1}\nabla u(t)\|_{L^{\infty}}  \\
\lesssim &\, 2^{q}\sum_{|q-k|\leq 4}d_{q}(t)2^{-2q}\|u(t)\|_{\dot{B}_{2,1}^{5/2}}\|a(t)\|_{\dot{B}_{2,1}^{2}}  \\
\lesssim &\, d_{q}(t)2^{-q}\|u(t)\|_{\dot{B}_{2,1}^{5/2}}\|a(t)\|_{\dot{B}_{2,1}^{2}}.
\end{align*}
Notice that
\begin{align*}
\|\nabla S_{k-1}a(t)\|_{L^{\infty}} \lesssim &\, \sum_{j\leq k-2}\|\nabla\Delta_{j}a(t)\|_{L^{\infty}}  \\
\lesssim &\, \sum_{j\leq k-2}2^{\frac{1}{2}j}2^{2j}\|\Delta_{j}a(t)\|_{L^{2}}\lesssim 2^{\frac{k}{2}}\|a(t)\|_{\dot{B}_{2,1}^{2}},
\end{align*}
this gives
\begin{align*}
\|\mathrm{div}\, [T_{a}, \Delta_{q}]\nabla u(t)\|_{L^{2}}\lesssim &\, \sum_{|q-k|\leq 4}\|\nabla S_{k-1}a(t)\|_{L^{\infty}}\|\nabla\Delta_{k}u(t)\|_{L^{2}} \\
\lesssim &\, \sum_{|q-k|\leq 4}d_{k}(t)2^{\frac{1}{2}k} 2^{-\frac{3}{2}k}\|u(t)\|_{\dot{B}_{2,1}^{5/2}}\|a(t)\|_{\dot{B}_{2,1}^{2}}  \\
\lesssim &\, d_{q}(t)2^{-q}\|u(t)\|_{\dot{B}_{2,1}^{5/2}}\|a(t)\|_{\dot{B}_{2,1}^{2}}.
\end{align*}
Summing up all the above estimates, we arrive at
\begin{align*}
\|\mathrm{div}\,[\Delta_{q}, a]\nabla u(t)\|_{L^{2}} \lesssim d_{q}(t)2^{-q}\|a(t)\|_{\dot{B}_{2,1}^{2}}\|u(t)\|_{\dot{B}_{2,1}^{5/2}}.
\end{align*}
Similarly, we have
\begin{align*}
[\nabla a\cdot \nabla, \Delta_{q}]u(t) = &\, [T_{\nabla a}, \Delta_{q}]\nabla u - \Delta_{q}\mathcal{R}(\nabla a, \nabla u) \\
&\, - \Delta_{q}(T_{\nabla u}\nabla a) + R(\nabla a, \nabla\Delta_{q}u).
\end{align*}
Then following the same line of reasoning as above, we obtain
\begin{align*}
\|[\nabla a\cdot\nabla, \Delta_{q}]u(t)\|_{L^{2}}\lesssim \,d_{q}(t)2^{-q}\|a(t)\|_{\dot{B}_{2,1}^{2}}\|u(t)\|_{\dot{B}_{2,1}^{5/2}}.
\end{align*}
Finally, we have
\begin{align*}
\|\nabla a(t) \Pi(t)\|_{\dot{B}_{2,1}^{1}} \lesssim &\, \|\nabla a(t)\|_{\dot{B}_{2,1}^{1}}\|\Pi(t)\|_{\dot{B}_{2,1}^{3/2}}.
\end{align*}
At this point, we list all the different estimates related to the restriction of $s$.
\end{proof}

Next, we will give the proof of (\ref{estimate in paper}).
\begin{proof}
Applying Lemma \ref{linear estimate momentum} to (\ref{d u b}) together with Gronwall's inequality gives
\begin{align*}
&\, \|(\nabla\bar{u}, \nabla\bar{B})\|_{\tilde{L}_{t}^{\infty}(\dot{B}_{2,1}^{1/2})} + \|(\nabla\bar{u}, \nabla\bar{B})\|_{L_{t}^{1}(\dot{B}_{2,1}^{5/2})} \\
\leq &\, C\exp{\left\{ \int_{0}^{t}\|\bar{u}(\tau)\|_{\dot{B}_{2,1}^{5/2}} + \|\bar{B}(\tau)\|_{\dot{B}_{2,1}^{5/2}}\,d\tau \right\}}
\Big\{\|(\nabla\bar{u}_{0}, \nabla\bar{B}_{0})\|_{\dot{B}_{2,1}^{1/2}} \\
&\, + \|\nabla\bar{a}\|_{\tilde{L}_{t}^{\infty}(\dot{H}^{1})}\left( \|\nabla\bar{\Pi}\|_{\tilde{L}_{t}^{1}(\dot{H}^{1})}
+ \|\bar{u}\|_{L_{t}^{1}(\dot{B}_{2,1}^{3})} \right) \Big\},
\end{align*}
where we also used (\ref{estimate a}).
Then thanks to (\ref{global reference 5/2}) and (\ref{estimate a}), we conclude for every $\eta > 0$ that
\begin{align}
\label{dudb}
\begin{split}
&\, \|(\nabla\bar{u}, \nabla\bar{B})\|_{\tilde{L}_{t}^{\infty}(\dot{B}_{2,1}^{1/2})} + \|(\nabla\bar{u}, \nabla\bar{B})\|_{L_{t}^{1}(\dot{B}_{2,1}^{5/2})}  \\
\leq &\, C + C_{\eta}\left( \|\Delta\bar{u}\|_{L_{t}^{1}(L^{2})} + \|\nabla\bar{\Pi}\|_{L_{t}^{1}(L^{2})} \right)
+ \eta \left( \|\bar{u}\|_{L_{t}^{1}(\dot{B}_{2,1}^{7/2})} + \|\nabla\bar{\Pi}\|_{\tilde{L}_{t}^{1}(\dot{H}^{3/2})} \right)
\end{split}
\end{align}
Notice that $\mathrm{div}\,\bar{u} = \mathrm{div}\,\bar{B} = 0$, we get by taking $\mathrm{div}$ to (\ref{d u b}) that
\begin{align*}
\mathrm{div} \left( (1+\bar{a})\nabla\partial_{i}\bar{\Pi} \right) = & -\mathrm{div} \partial_{i} \left[ (\bar{u}\cdot\nabla)\bar{u} \right]
+ \mathrm{div} \partial_{i} \left[ \bar{a} \, \Delta \bar{u} \right] - \mathrm{div}\left[ \partial_{i}\bar{a}\,\nabla \bar{\Pi} \right] \\
& + \mathrm{div}\partial_{i}\left( \bar{B}\cdot\nabla\bar{B} \right) + \mathrm{div} \left( \partial_{i}\bar{a}\,\bar{B}\cdot\nabla\bar{B} \right)  \\
& + \mathrm{div}\left( \bar{a}\,\partial_{i}\bar{B}\cdot\nabla\bar{B} \right)
+ \mathrm{div}\left( \bar{a}\,\bar{B}\cdot\nabla\partial_{i}\bar{B} \right).
\end{align*}
From this and Lemma \ref{pressure_estimate}, we deduce that
\begin{align*}
\|\nabla^{2}\bar{\Pi}\|_{L_{t}^{1}(\dot{B}_{2,1}^{1/2})} \lesssim &\, \|\partial_{i}[\bar{u}\cdot\nabla\bar{u}]\|_{L_{t}^{1}(\dot{B}_{2,1}^{1/2})}
+ \|\partial_{i}[\bar{a} \,\Delta\bar{u}]\|_{L_{t}^{1}(\dot{B}_{2,1}^{1/2})}    \\
&\, + \|\partial_{i}\bar{a}\,\nabla\bar{\Pi}\|_{L_{t}^{1}(\dot{B}_{2,1}^{1/2})} + \|\partial_{i}(\bar{B}\cdot\nabla\bar{B})\|_{L_{t}^{1}(\dot{B}_{2,1}^{1/2})} \\
&\, + \|\partial\bar{a}\,\bar{B}\cdot\nabla\bar{B}\|_{L_{t}^{1}(\dot{B}_{2,1}^{1/2})} + \|\bar{a}\,\partial_{i}\bar{B}\cdot\nabla\bar{B}\|_{L_{t}^{1}(\dot{B}_{2,1}^{1/2})}
\\
&\, + \|\bar{a}\,\bar{B}\cdot\nabla\partial_{i}\bar{B}\|_{L_{t}^{1}(\dot{B}_{2,1}^{1/2})}.
\end{align*}
Applying the product laws in Besov spaces yields that for any $\epsilon > 0$
\begin{align*}
\|\nabla^{2}\bar{\Pi}\|_{L_{t}^{1}(\dot{B}_{2,1}^{1/2})} \lesssim &\, \|\bar{u}\|_{L_{t}^{\infty}(\dot{B}_{2,1}^{3/2})}\|\bar{u}\|_{L_{t}^{1}(\dot{B}_{2,1}^{5/2})} \\
&\, + \|\nabla\bar{a}\|_{\tilde{L}_{t}^{\infty}(\dot{B}_{2,1}^{3/2})}\|\Delta\bar{u}\|_{L_{t}^{1}(\dot{B}_{2,1}^{1/2})}
 + \|\bar{a}\|_{L_{t}^{\infty}(\dot{B}_{2,1}^{3/2})}\|\bar{u}\|_{L_{t}^{1}(\dot{H}^{7/2})} \\
& + \|\bar{B}\|_{L_{t}^{\infty}(\dot{B}_{2,1}^{3/2})}\|\bar{B}\|_{L_{t}^{1}(\dot{B}_{2,1}^{5/2})}
+ \|\bar{a}\|_{L_{t}^{\infty}(\dot{B}_{2,1}^{3/2})}\|\bar{B}\|_{L_{t}^{\infty}(\dot{B}_{2,1}^{3/2})}\|\bar{B}\|_{L_{t}^{1}(\dot{B}_{2,1}^{5/2})} \\
& + \|\bar{a}\|_{\tilde{L}_{t}^{\infty}(\dot{H}^{2})}\left( \epsilon \|\nabla^{2}\bar{\Pi}\|_{\tilde{L}_{t}^{1}(\dot{H}^{1/2})}
+ C_{\eta}\|\nabla\bar{\Pi}\|_{L_{t}^{1}(L^{2})} \right).
\end{align*}
Thanks to Theorem \ref{decay_main_theorem}, Proposition \ref{global reference 5/2} and estimates (\ref{estimate a}), we get by
taking $\epsilon$ sufficiently small in the above inequality that
\begin{align*}
\|\nabla^{2}\bar{\Pi}\|_{L_{t}^{1}(\dot{B}_{2,1}^{1/2})} \leq C\, \left( 1 + \|\nabla \bar{u}\|_{\tilde{L}_{t}^{\infty}(\dot{B}_{2,1}^{1/2})}
+ \|\nabla\bar{u}\|_{L_{t}^{1}(\dot{B}_{2,1}^{5/2})} + \|\nabla\bar{B}\|_{\tilde{L}_{t}^{\infty}(\dot{B}_{2,1}^{1/2})} \right).
\end{align*}
Substituting the above inequality into (\ref{dudb}) and then taking $\eta$ sufficiently small, we arrive at
\begin{align*}
\|\nabla\bar{u}\|_{\tilde{L}_{t}^{\infty}(\dot{B}_{2,1}^{1/2})} + & \|\nabla\bar{B}\|_{\tilde{L}_{t}^{\infty}(\dot{B}_{2,1}^{1/2})}
+ \|\nabla\bar{u}\|_{L_{t}^{1}(\dot{B}_{2,1}^{5/2})} \\
&\quad +\|\nabla\bar{B}\|_{L_{t}^{1}(\dot{B}_{2,1}^{5/2})} + \|\nabla\bar{\Pi}\|_{L_{t}^{1}(\dot{B}_{2,1}^{3/2})} \leq \,C.
\end{align*}
\end{proof}

\section{Acknowledgements}
J. Peng and J. Jia's research is support partially by National Natural Science Foundation of China under the grant no.11131006,
and by the National Basic Research Program of China under the grant no.2013CB329404.
K. Li's research is support partially by National Natural Science Foundation of China under the grant no.11201366.

\end{document}